\documentclass[12pt, reqno]{amsart}

\usepackage{latexsym,amsfonts,amsmath,amssymb,amsthm,url,graphics,psfrag,amscd,stmaryrd, mathrsfs}
\usepackage[english]{babel}
\usepackage[T1]{fontenc}

\usepackage{graphics}
\usepackage{graphicx}
\usepackage{subcaption}
\usepackage{color}

\usepackage{ifthen}
\newboolean{longversion}
\setboolean{longversion}{true}
\usepackage{bbm}

\usepackage{hyperref}

\newcommand{\Var}{\mathrm{Var}}
\newcommand{\tmix}{t_{\mathrm{mix}}}
\newcommand{\tmixh}{\hat{t}_{\mathrm{mix}}}
\newcommand{\muhbb}{\hat{\mu}_n^{\beta,B}}
\newcommand{\mubb}{\mu^{\beta,B}_n}
\newcommand{\hvarphibb}{\hat{\varphi}_{\beta, B}}
\newcommand{\hvarphibo}{\hat{\varphi}_{\beta, 0}}
\newcommand{\varphibb}{\varphi_{\beta, B}}
\newcommand{\varphibo}{\varphi_{\beta, 0}}
\newcommand{\tvarphibo}{\tilde{\varphi}_{n}^{\beta,0}}

\newcommand{\w}{\mathrm{W}}
\newcommand{\R}{\mathbb{R}}
\newcommand{\T}{\mathbb{T}}

\newcommand{\E}{\mathbb{E}}
\newcommand{\Z}{\mathbb{Z}}

\newcommand{\pp}{\mathbb{P}}

\newcommand{\kC}{\mathcal{C}}

\newcommand{\kO}{\mathcal{O}}

\newcommand{\kF}{\mathcal{F}}

\newcommand{\kE}{\mathcal{E}}
\newcommand{\kH}{\mathcal{H}}

\newcommand{\en}{\mathbb{E} }

\newcommand{\saw}{\mathrm{saw}}

\newcommand{\lin}{\left[\kern-0.15em\left[}
\newcommand{\rin} {\right]\kern-0.15em\right]}
\newcommand{\linf}{[\kern-0.15em [}
\newcommand{\rinf} {]\kern-0.15em ]}
\newcommand{\ilin}{\left]\kern-0.15em\left]}
\newcommand{\irin} {\right[\kern-0.15em\right[}

\newtheorem{lem}{Lemma}[section]

\newtheorem{prop}[lem]{Proposition}
\newtheorem{theo}[lem]{Theorem}

\newtheorem {rem}[lem] {Remark}

\newtheorem*{ack}{Acknowledgments}

\usepackage{color}
\newcommand{\eqn}[1]{\begin{equation} #1 \end{equation}}

\newcommand{\colrev}[1]{\textcolor[rgb]{0,0,0}{#1}}
\newcommand{\vep}{\varepsilon}
\newcommand{\Prob}{{\mathbb P}}
\newcommand{\expec}{{\mathbb E}}
\newcommand{\op}{o_{\sss \Prob}}

\newcommand{\indic}[1]{\mathbbm{1}_{\{#1\}}}
\newcommand{\indicwo}[1]{\mathbbm{1}_{#1}}
\newcommand{\e}{{\mathrm e}}
\newcommand{\sss}{\scriptscriptstyle}

\numberwithin{equation}{section}

\title[Ising model on random regular graphs]
{\bf Glauber dynamics for Ising models on
random regular graphs: cut-off and metastability}
\author{Van Hao Can}
\address{VHC: Research Institute for Mathematical Sciences,
Kyoto University, Kyoto 606-8502, Japan 
 \&	Institute of Mathematics, Vietnam Academy of Science and Technology, 18 Hoang Quoc Viet, 10307 Hanoi, Vietnam}
\author{Remco van der Hofstad}
\address{RvdH: Eindhoven University of Technology, P.O. Box 513, 5600 MB Eindhoven, The Netherlands}
 \author{Takashi Kumagai}
\address{TK: Research Institute for Mathematical Sciences,
Kyoto University, Kyoto 606-8502, Japan} 

\begin{document}

\maketitle

\begin{abstract}
Consider random $d$-regular graphs, i.e., random graphs such that there are exactly $d$ edges from each vertex for some $d\ge 3$. We study both the configuration model version of this graph, which has occasional multi-edges and self-loops, as well as the simple version of it, which is a $d$-regular graph chosen uniformly at random from the collection of all $d$-regular graphs.

In this paper, we discuss mixing times of Glauber dynamics for the Ising model with an external magnetic field on a random $d$-regular graph, both in the quenched as well as the annealed settings. Let $\beta$ be the inverse temperature, $\beta_c$ be the critical temperature and $B$ be the external magnetic field. Concerning the annealed measure, we show that for $\beta > \beta_c$ there exists $\hat{B}_c(\beta)\in (0,\infty)$ such that the model is {\em metastable} (i.e., the mixing time is exponential in the graph size $n$) when $\beta> \beta_c$ and $0 \leq B < \hat{B}_c(\beta)$, whereas it exhibits the cut-off phenomenon at $c_\star n \log n$ with a window of order $n$ when  $\beta < \beta_c$ or $\beta > \beta_c$ and $B>\hat{B}_c(\beta)$. Interestingly, $\hat{B}_c(\beta)$ coincides with the critical external field of the Ising model on the $d$-ary tree (namely, above which the model has a unique Gibbs measure). Concerning the quenched measure, we show that there exists $B_c(\beta)$ with $B_c(\beta) \leq \hat{B}_c(\beta)$ such that for $\beta> \beta_c$, the mixing time is at least exponential along some subsequence $(n_k)_{k\geq 1}$ when $0 \leq B <  B_c(\beta)$, whereas it is less than or equal to $Cn\log n$ when $B>\hat{B}_c(\beta)$. The quenched results also hold for the model conditioned on simplicity, for the annealed results this is unclear.
\end{abstract}

\section{Introduction}

The Ising model is a paradigmatic model in statistical mechanics. It was invented by Ising and his PhD-supervisor Lenz to model magnetism, for which it was considered on regular lattices (see \cite{Niss05, Niss09} for the interesting history of the Ising model, as well as the standard books \cite{Bovi06, Elli85} and the references therein). With the view that the Ising model can also model cooperative behavior, its relevance 
in the area of  
network science has increased, and the literature on the Ising model on random graphs, invented to model complex networks, has exploded. See \cite{DorGolMen08} for the physics perspective on critical phenomena on random graphs. Initially, the focus was on establishing the thermodynamic limit of the quenched Ising model on random graphs \cite{DM, DemMon10b, DomGiaHof10} as well as on its critical behavior \cite{DomGiaHof12, GiaGibHofPri15}. In the past years, also the annealed Ising model has attracted considerable attention \cite{C, Can17b, CanGiaGibHof19, DomGiaGibHofPri16, GiaGibHofPri16}. As we explain in more detail below, the quenched and annealed settings for the Ising model describe different physical realities in the dynamics of the underlying graph and the Ising model on it. The local weak limit of the Ising model on locally tree-like random graphs was studied in \cite{BasDem17, MonMosSly12}.

\colrev{Recently, the {\em dynamical} properties of the Ising model have attracted attention, focusing on its metastable behavior \cite{BMP,D, DHJN,HJ} in two frameworks. In the first one (see \cite{D,DHJN}),  large random graphs were considered, on which an Ising model lives with a slightly positive field at very low temperature. In such settings, the all-plus state is the ground state, and thus the most likely state for the system to be in. However, due to the strong coupling, the all-minus state is a metastable state, and it takes the system a very long time to leave this state when started from it. For large network size, the main results in \cite{D, DHJN} give detailed estimates for the transition time to move from the metastable state to the stable state, as the temperature tends to zero, for graphs of fixed large size.  In the second framework, the  temperature of the system is fixed and we are interested in what happens as the system size tends to infinity  The concrete analyses of metastable hitting time  for the Glauber dynamics  on dense Erd\H{o}s-R\'eyni random graphs have been given in \cite{BMP, HJ}, by using a potential theoretic approach as discussed in detail in \cite{BH}.  }

 \colrev{In this paper, we follow the second framework. However, rather than treating it as a {\em metastable} system, we approach it as a Markov chain mixing-time problem. Our paper provides the first results dealing with {\em sparse random graphs} instead of dense ones.}  In general, it is believed that for supercritical temperature, mixing is fast (mixing time of order $n\log{n}$ where $n$ is the size of the graph) \cite{MS}, while for subcritical temperatures and small external fields, the mixing time is exponentially large in the graph size. We focus on both the quenched as well as the annealed Ising model on random regular graphs, where our results are the strongest in the annealed setting. In particular, our main results and innovations are as follows:

\begin{itemize}
\item[(a)] For subcritical temperatures, where the Ising model at zero external field has two Gibbs measures, we identify the critical value for the field in the annealed setting. More precisely, for large field, mixing is rapid (mixing time of order $n\log{n}$), while for small field, mixing is slow (mixing time exponential in the system size). The latter corresponds to the metastable setting. The proofs rely on a close relation between the annealed Ising model and birth-death chains, for which such results have been established in \cite{BBF,CC,DLP,DLP2,LLP}. 

\item[(b)] For the quenched Ising model, we prove similar properties, however, we are not able to identify the {\it exact} critical value of the external magnetic field, but instead resort to bounds on it. The results rely on proofs of mixing times for Glauber dynamics on general graphs such as proved in the literature, and most effort is in proving that the necessary conditions hold for degree-regular configuration models.

\item[(c)] While the mixing time is at least $\e^{\lambda n}$ in both models for some appropriate $\lambda>0$ in the slow-mixing regime, in the annealed setting, the constant $\lambda$ is {\em independent} of $\beta$, while in the quenched setting, the constant $\lambda$ is {\em linear} in $\beta$ for large $\beta$.
\end{itemize}

The remainder of this section is organized as follows. We start in Section \ref{sec-RRG} by defining the random regular graphs that we will be working on. In Section \ref{sec-Ising}, we define the Ising model on random graphs in its quenched and annealed settings. In Section \ref{sec-previous-Ising}, we recall some previous results on the Ising model proved by the first author \cite{C} in the annealed setting and by Dembo and Montanari \cite{DM} in the quenched setting. In Section \ref{sec-results}, we state our main results. We close in Section \ref{sec-disc} with discussion and open problems.

\subsection{Random regular graphs}  
\label{sec-RRG}
Let us start by defining the configuration model introduced by Bollob\'as \cite{Boll80b} in the degree-regular case, and by Molloy and Reed \cite{MolRee95} in the general degree case.  We consider a sequence $(G_n)_{n\geq 1}$ of such graphs. To define it, start for each $n$ with the vertex set $[n]=\{1, \ldots, n\}$. Construct the edge set as follows. Consider a sequence of degrees $(d_i)_{i \in[n]}$ and assume that $\ell_n =\sum_{i\in[n]} d_i$ is even. For each vertex $i\in[n]$, start with $d_i$ half-edges incident to $i$.  Denote the set of all the half-edges by $\kH$. Select $h_1\in \kH$ arbitrarily, and then choose a half-edge $h_2$ uniformly from $\kH \setminus \{h_1\}$, and pair $h_1$ and $h_2$ to form an edge. Next, select an arbitrarily half-edge $h_3\in \kH \setminus \{h_1, h_2\}$, and pair it to $h_4$ uniformly chosen from  $\kH \setminus \{h_1, h_2, h_3\}$. Continue this procedure until there are no more half-edges.  The resulting graph is called the {\em configuration model}, see \cite[Chapter 7]{Hofs17} for an extensive introduction. In particular, it is known that the configuration model conditioned on simplicity is a {\em uniform} random graph with the prescribed degrees $(d_i)_{i \in[n]}$ \cite[Proposition 7.7]{Hofs17}.

In this paper, we consider the {\it random $d$-regular graph}, that is $d_i =d$ for all $i\in[n]$, with $d\geq 3$ and $nd$ assumed to be even. We let $\pp$ and $\en$ denote the probability measure and expectation with respect to the random regular graph. We say that a sequence of events $(A_n)_{n\geq 1}$ occurs with high probability (which we abbreviate as whp) if $\pp(A_n)=1-o(1)$ as $n\rightarrow \infty$.

In the degree-regular setting, it is also known that the probability that the configuration model is simple converges to a positive value \cite[Theorem 7.12]{Hofs17} (see also \cite{AngHofHol16, Boll80b, Jans06b, Jans13a}), which implies that any result that holds whp for the configuration model, also holds whp for the random regular graph. This implies that all the results in the {\em quenched} setting also hold for the random regular graph. In the annealed setting, this is less obvious, as we are taking expectations with respect to exponential functionals in the Ising model, as we now explain in more detail.

\subsection{Ising model and Glauber dynamics} 
\label{sec-Ising}
In this section, we define the quenched and annealed Ising models, as well as Glauber dynamics for it.

\subsubsection{Ising model}
Let $\Omega_n=\{-1,1\}^n$ be the space of spin configurations. For any spin vector $\sigma=(\sigma_1, \ldots, \sigma_n) \in \Omega_n$, the Hamiltonian is given by  
	\[
	H_n(\sigma)=H_n^{\beta, B}(\sigma)= -\beta  \sum_{i,j\in[n]\colon  i \leq j} k_{i,j} \sigma_i \sigma_j - B \sum_{i\in[n]} \sigma_i,
	\]
where $k_{i,j}$ is the number of edges between vertices $i$ and $j$, $\beta \geq 0$ is the inverse temperature and $B \in \R$ is the uniform external magnetic field.  

The {\it quenched} measure is defined  by Boltzmann law, defined, for any $\sigma \in \Omega_n$, as
	\[
	\mu_n^{\beta, B}(\sigma)= \frac{\exp (-H_n(\sigma))}{Z_n^{\beta, B} },
	\] 
where $Z_n^{\beta,B}$ is the partition function given by
	\[
	Z_n^{\beta,B}= \sum\limits_{\sigma \in \Omega_n} \exp( -H_n^{\beta,B}(\sigma)).
	\]
Similarly, the {\it annealed } measure is defined, for any $\sigma \in \Omega_n$, as 
	\[
	\muhbb (\sigma)= \frac{\E[\exp (-H_n(\sigma))]}{\E[Z_n^{\beta, B} ]},
	\] 
where $\E$ is the expectation \colrev{with respect to} the random graph under consideration. These two measures concern different physical realities (see  \cite{CanGiaGibHof19}). While  the random graph in the quenched measure is fixed, or can be thought of as varying very slowly compared to the Ising Glauber dynamics,  in the annealed law, the Glauber dynamics only observes an {\em average} graph \colrev{instant}, which, by the ergodic theorem, can be thought of as an expectation with respect to the graph randomness. As discussed before, we will work on the $d$-regular configuration model, but in our discussion we will also discuss more general degree settings. Throughout the paper, we consider the case $B\geq 0$; the other case can be treated identically by symmetry.

\subsubsection{Glauber dynamics} 
Let us first recall the Glauber dynamics for a given reversible measure $\nu_n$ on $\Omega_n$.  Let $(\xi_t)_{t\geq 0}$ be a discrete Markov chain on $\Omega_n$ with transitions as follows. Assume that $\xi_t= \sigma$, let $I$ be a random index in $[n]$ chosen uniformly at random.  Then,  
	\begin{equation} \label{trsp}
	\xi_{t+1} =  
	\begin{cases}
	\sigma^{+I}  & \textrm{ with probability } \frac{\nu_n(\sigma^{+I})}{\nu_n(\sigma^{+I}) +\nu_n(\sigma^{-I})},\\
	\sigma^{-I}  & \textrm{ with probability } \frac{\nu_n(\sigma^{-I})}{\nu_n(\sigma^{+I}) +\nu_n(\sigma^{-I})}, 
	\end{cases}
	\end{equation}
where, for all $\sigma \in \Omega_n$ and $i\in[n]$, we define $\sigma^{+i}$ and $\sigma^{-i}$ to be the $i$-spin-flipped version of $\sigma$, i.e.,
	\[
	\sigma^{+i}_j = \sigma^{-i}_j=\sigma_j \mbox{ for }j\neq i,\quad  \mbox{and }\sigma^{+i}_i = -\sigma^{-i}_i = 1.
	\]
We denote by $(\xi_{t}^{\sigma})_{t\geq 0}$ the Glauber dynamics starting from the configuration $\sigma$. It is well known that $(\xi_t)_{t\geq 0}$ is a reversible Markov chain with stationary measure $\nu_n$. Hence, as $t \rightarrow \infty$, the distribution of $(\xi_t)_{t\geq 0}$ converges to $\nu_n$.  The distance to stationary of the Glauber dynamics is defined as
 	$$
	d_n(t)= \max \limits_{\sigma  \in \Omega_n} \|\pp(\xi^{\sigma}_t \in \cdot) - \nu_n(\cdot)\|_{\sss \rm TV},
	$$
where $\|\mu - \nu\|_{\sss \rm TV}$ is the total variation distance between the probability measures $\mu$ and $\nu$.  Then, the {\it mixing time}  is defined as  
	$$
	\tmix = \min \{ t\colon d_n(t) \leq \tfrac{1}{4}\}.
	$$ 
The value $\tfrac{1}{4}$ is arbitrary and can be replaced by any other value $\vep\in (0,1)$.

In this paper, we study the mixing time with respect to the quenched and annealed measures $\mubb$ and $\muhbb$ defined previously.  From now on, we let $\tmix$ be the quenched mixing time for $\mubb$, and $\tmixh$ the annealed mixing time for $\muhbb$, respectively. Notice that while $\tmix$ is random as it depends on the random graph, $\tmixh$ is non-random.

For $d\ge 3$, the Ising model on random regular graphs exhibits a phase transition at the critical value $\beta_c= \textrm{atanh}(1/(d-1))$, see e.g.\ \cite{DM}. (Note that this model does not exhibit a phase 
transition for $d=1,2$, and that is why we consider $d\ge 3$.) The mixing time has been studied in the high-temperature regime:
	\begin{theo}[Mixing times high-temperature Ising model \protect{\cite[Theorem 1]{MS}}]
	\label{lot}
	There exists a positive constant $C$, such that if $\beta < \beta_c$, then whp $\tmix \leq C n \log n$.
	\end{theo}
We notice that the results in \cite{MS} hold for the general class of finite graphs with degree bounded by $d$ for any $\beta<\textrm{atanh}(1/(d-1))$. For us, it is crucial that $\textrm{atanh}(1/(d-1))$ is the critical value for the Ising model on the configuration model, so that Theorem \ref{lot} holds throughout the high-temperature regime.

In this paper, apart from studying the mixing times of the Ising model, we also study its cut-off behavior. This is related to the speed at which $t \mapsto d_n(t)$ decreases. Formally, with $\tmix(\vep)=\min \{ t\colon d_n(t) \leq \vep\}$, we say that the cut-off phenomenon occurs at $a_n$ with a window of order $b_n\ll a_n$ when, for all $\vep\in (0,\tfrac{1}{2})$,
	\eqn{
	\tmix(\vep)/a_n \rightarrow 1,
	\qquad
	\text{while}
	\qquad
	\tmix(\vep)-\tmix(1-\vep)=\Theta(b_n).
	}

\subsection{Previous results about Ising models on the configuration model}
\label{sec-previous-Ising}
In this section, we state some important previous results about the annealed and quenched Ising model that we shall rely upon. We first define the {\it fixed-spin partition function}, which will play an important role in our results. For $\sigma \in \Omega_n$, we write
	\[
	\sigma_+= \{i\in[n]\colon \sigma_i=1\} \qquad \textrm{and} \qquad \sigma_-= \{i\in[n]\colon \sigma_i=-1\}.
	\] 
For any $t\in (0,1)$, we define 
	\begin{equation*}
	\varphi_{n}^{\beta,B}(t) = \frac{1}{n }\log Z_n^{\beta,B}(t), 
	\end{equation*} 
where 
	\begin{equation*}
 	Z_n^{\beta,B}(t)=\sum_{\sigma\colon |\sigma_+| =\lceil nt \rceil} \e^{-H_n(\sigma)},
	\end{equation*}
and 
	\begin{equation*}
	\hat{\varphi}_{n}^{\beta,B}(t) = \frac{1}{n }\log \E \big[ Z_n^{\beta,B}(t)\big]. 
	\end{equation*} 
It has been shown in the proof of \cite[Theorem 1.1(i)]{C} that 
	\begin{equation} \label{znphk}
	\sup_{0\leq k \leq n} \Big | \frac{1}{n} \log \sum_{\sigma\colon |\sigma_+|=k} \E \left( \e^{-H_n(\sigma)}\right) -  \hat{\varphi}_{\beta, B}\left(k/n\right) \Big| = \kO(1/n),
	\end{equation}
with 
	\begin{equation} \label{limhvarphibb}
 	\hat{\varphi}_{\beta, B}(t) =  \frac{\beta d}{2}  -B+ I(t) + 2Bt + d \int_0^{t\wedge 1-t} \log f_{\beta}(s)ds,
	\end{equation}
where  $I(0)=I(1)=0$ and, for $t \in (0,1)$
	\begin{equation*}
	I(t) = (t-1) \log (1-t) -t \log t,
	\end{equation*}
and 
	\begin{equation}\label{dnwofbeta}
	f_{\beta}(t) = \frac{\e^{-2\beta}(1-2t)+ \sqrt{1+(\e^{-4\beta}-1)(1-2t)^2}}{2(1-t)}.
	\end{equation}
Thus, for all $t\in [0,1]$,
	\begin{eqnarray*} 
	\lim\limits_{n \rightarrow \infty} \hat{\varphi}_n^{\beta, B}(t) =  \hat{\varphi}_{\beta, B}(t).
	\end{eqnarray*}
\colrev{Denote the quenched and annealed pressures by, respectively},
	\eqn{
	\varphi(\beta,B)=\lim\limits_{n \rightarrow \infty} \frac{1}{n} \log Z_n^{\beta,B},
	\qquad
	\hat{\varphi}(\beta, B) 
	=  \lim\limits_{n \rightarrow \infty} \frac{1}{n} \log \E(Z_n^{\beta,B}).
	}
It is shown in \cite[Theorem 1]{DMSS} and \cite[Theorem 1.1 \& Proposition 3.2]{C} that the annealed and quenched pressures are equal and have a variational expression as
	\begin{eqnarray} 
	\label{an=qu}
	\varphi(\beta,B)\quad \overset{\rm a.s.}{=} \quad \hat{\varphi}(\beta, B) 
	= \max_{t \in [0,1]} \hvarphibb(t).
	\end{eqnarray}
These results will be crucial to establish the `energy landscape' of Glauber dynamics for the Ising model. 

\subsection{Main results}  
\label{sec-results}
In this section, we state our main results. We start with the results for the annealed Ising model, followed by our results on the quenched Ising model.

\subsubsection{Main results for the annealed Ising model} 
For the annealed case with $\beta>\beta_c$, we can show (see Lemma \ref{ll} below) that there is a threshold $\hat{B}_c=\hat{B}_c(\beta)$, such that if $B>\hat{B}_c$, then the function $t\mapsto \hvarphibb(t)$ is unimodular, i.e., $t\mapsto \hvarphibb(t)$ has only one critical point which is the maximizer characterizing the annealed pressure in \eqref{an=qu}. On the other hand, if $0\leq B<\hat{B}_c$, then $t\mapsto \hvarphibb(t)$ has three critical points (one local maximizer, one global maximizer and one local minimizer), and thus the graph of $t\mapsto \hvarphibb(t)$ has a valley. This particular observation suggests that for $\beta>\beta_c$, there is a phase transition in the mixing time of the annealed Ising model when $B$ crosses the value $\hat{B}_c$. Indeed, we will show in Theorem \ref{amt} below that for $\beta>\beta_c$, the mixing time increases exponentially in the graph size when $B<\hat{B}_c$ but it is of order $n\log{n}$ when $\beta>\beta_c, B>\hat{B}_c$, or when $\beta\in[0,\beta_c)$. 

We observe that $\hat{\varphi}_n^{\beta,B}(t) = \hat{\varphi}_n^{\beta,0}(t) + B(2t-1)$ and thus $\hvarphibb(t) = \hvarphibo(t) + B(2t-1)$. Therefore, the unimodularity of $\hvarphibb$ is strongly related  to the reflection point of $t\mapsto \hvarphibo(t)$. Indeed, we will show in Lemma \ref{ll} below that 
	\begin{equation} 
	\label{def-hbc}
	\hat{B}_c=\hat{B}_c(\beta):=-\frac{1}{2}\hvarphibo'(t_u) = \sup_{t \in (0,\tfrac{1}{2})}\big( -\frac{1}{2}\hvarphibo'(t) \big) = \sup_{0<s<t<\tfrac{1}{2}} \frac{\hvarphibo(s)-\hvarphibo(t)}{2(t-s)} ,
	\end{equation} 
where $t_u$ is the reflection point of $\hvarphibo$ determined as in \eqref{eotu}. With this notation in hand, we now state our main result for the Glauber dynamics on the annealed Ising model on the $d$-regular configuration model:

\begin{theo}[Annealed Glauber dynamics] 
\label{amt}  
Consider the annealed Glauber dynamics on the $d$-regular configuration model. 
    \begin{itemize}
	\item [(i)] For $\beta> \beta_c$ and $0 \leq B < \hat{B}_c$,  there exist positive constants $C$ and $\lambda$ such that 
	$$C^{-1}\exp(\lambda n)  \leq \tmixh \leq Cn^4\exp(\lambda n).$$ 
	Moreover, 
		\eqn{
		\label{bounded-valley}
		\sup \{\lambda\colon \beta > \beta_c, 0\leq B <\hat{B}_c \}<\infty.
		} 
		
		\item[(ii)] For $\beta < \beta_c$ or $\beta > \beta_c$ but  $B>\hat{B}_c$, there exists a positive constant $c_\star$, such that the cut-off \colrev{phenomenon} occurs at $c_\star n \log n$ with a window of order $n$. 
	\end{itemize}
\end{theo}
\medskip

\paragraph{\bf Independence of inverse temperature in \eqref{bounded-valley}.} The fact that the constant $\lambda$ in the exponential growth of the mixing time in Theorem \ref{amt}(i) is independent of $\beta$ for large $\beta$ in \eqref{bounded-valley} is rather remarkable. It can be understood as follows. Think of the curve $t\mapsto  \hat{\varphi}_{\beta, B}(t)$ as an energy landscape. It takes an exponential amount of time to cross any energy barrier, meaning a difference in energy, \colrev{so that it takes time of order $\e^{n\lambda}$ to cross an energy barrier $\lambda$. For the annealed Ising model, $\hat{\varphi}_{\beta, B}(s)$ acts as the energy of a configuration with roughly $ns$ plus spins, where $s\in(0,1)$. Then,} let us start from a configuration $\sigma$ for which $|\sigma^+|\approx s n$ where $s\neq t^\star$ and $t^\star$ is such that $\hat{\varphi}(\beta, B)=\hat{\varphi}_{\beta, B}(t^\star)$ (recall \eqref{an=qu}). Let $B>0$, so that $t^\star>\tfrac{1}{2}$. Then, the amount of time to go from $|\sigma^+|\approx s n$ to equilibrium is close to $\e^{n\lambda(s)}$, where
	\eqn{
	\lambda(s)=\sup_{t\in [s,t^\star]}   \hat{\varphi}_{\beta, B}(t^\star)-\hat{\varphi}_{\beta, B}(t).
	}
It can be expected that the worst case for this is when $s=s^\star$ for some \colrev{specific} $s^\star<\tfrac12$, which suggests that $\lambda=\lambda(s^\star)$. Then, to go between any $\sigma$ with $|\sigma^+|\approx s^\star n$ and the stationary distribution \colrev{(having approximately $t^\star n$ with $t^\star>\tfrac{1}{2}$ plus spins)}, the dynamics has to pass through a spin configuration $\sigma'$ with \colrev{roughly $n/2$ plus spins}. 

\colrev{In many settings, for example also in the quenched setting,} this will lead to an energy difference that grows linearly in $n$ as well as in $\beta$ for large $\beta$, \colrev{since the number of edges between any two disjoint sets of size approximately $n/2$ will be {\em linear} in $n$ for $n$ large, so that the energy barrier will be of order $n\beta$.} However, the annealed dynamics on the configuration model is special, as it allows to update the graph as well when changing the spins, \colrev{which makes the energy barrier of order $n$ with a constant that is {\em independent} of $\beta$. 
Indeed,}{} there exists a graph configuration whose probability is exponentially small in $n$ but with an exponential rate that is independent of $\beta$, \colrev{and a partition of the vertices in two sets each of size approximately $n/2$, for which the number of edges between these two sets equals zero. This can be seen by noting that} the probability of splitting the \colrev{$d$-regular random} graph into {\em two} disjoint $d$-regular graphs of about equal size is exponentially small in $n$, with an exponential rate that is obviously independent of $\beta$. By then taking all spins to be plus on one part, and all spins to be minus on the other part, we see that we have a roughly equal number of plusses and minuses, while at the same time having an exponentially small cost whose exponential rate is {\em independent of $\beta$}. 

\colrev{The above argument implies that $\hat{\varphi}_{\beta, B}(\tfrac{1}{2})\leq C$ with $C$ independent of $\beta$, which, in particular, suggests also that $\lambda(s^\star)\leq C$ where $C$ is independent of $n$.} This explains \eqref{bounded-valley}. 
\medskip

\paragraph{\bf Cut-off in subcritical regimes.} In the setting where $\beta < \beta_c$, or $\beta > \beta_c$ but  $B>\hat{B}_c$, there is no valley in the energy landscape, meaning that the dynamics will move quickly from any spin configuration to the stationary distribution. The fact that Theorem \ref{amt}(ii) proves that this dynamics satisfies a cut-off phenomenon is a substantial improvement from the general result in Theorem \ref{lot}, which is a restatement of \cite[Theorem 1]{MS}, however, it is restricted to the annealed setting.

\subsubsection{Main results for the quenched Ising model}
We next state our results for the quenched Glauber dynamics. We start by investigating the fixed-spin partition function:
\begin{prop}[Quenched fixed-spin partition function]
\label{propbcq} The following assertions hold.
	\begin{itemize} 
		\item [(i)] 	For all $\beta, B$ and $t \in (0,1)$
			\[
			\varphi_n^{\beta,B} (t) - \tilde{\varphi}_n^{\beta, B}(t)  \xrightarrow{a.s.}0, 
			\]
		where $\tilde{\varphi}_n^{\beta, B}(t) =\E \big[\varphi_n^{\beta,B} (t)\big]$.
		\item[(ii)] There exists a subsequence $(n_k)_{k\geq 1}$, 
		 such that 
		for all $\beta, B$ the sequence of functions $(\tilde{\varphi}_{n_k}^{\beta,B})_{k\geq 1}$ converges uniformly to a continuous (non-random) function $\varphi_{\beta,B}$ in every compact subset of 
		$[0,1]$.  
		Moreover, 
		\begin{equation*}
		\varphi_{\beta,B} (t) = \varphi_{\beta,0}(t)+B(2t-1).
		\end{equation*}
		Furthermore, let
		\begin{equation}
		\label{def-hbc-que}
		B_c=B_c(\beta):= \sup_{0<s<t<\tfrac{1}{2}} \frac{\varphibo(s)-\varphibo(t)}{2(t-s)}.
		\end{equation}	 
		Then $B_c\in(0,\beta d)$ for all $\beta >\beta_c$. Moreover, there exists a positive constant $c$ such that $B_c \geq c \beta$ if $\beta \geq 12$.
	\end{itemize}
\end{prop}

It follows directly from the definition of $B_c$ that when $B<B_c$, the function $t\mapsto \varphi_{\beta,B}(t)$ is non-unimodular and it has a valley; otherwise it is unimodular. While it is not difficult to prove that the mixing time is of exponential order when $B<B_c$, it is not clear to us how to show that the mixing time is of $\exp(o(n))$ when $B>B_c$.  Instead, we can show 
the following. Let $B_c^G$ be 
the critical external field of Ising model on $d$-ary tree defined by 
	\begin{equation}
	\label{BcG-def}
	B_c^G=\inf\{B\geq 0\colon \textrm{ Ising model on infinite $d$-ary tree has a unique Gibbs measure} \}.
	\end{equation}	
Then, the mixing time is of logarithmic order for $B > B_c^G$:

\begin{theo} [Quenched Glauber dynamics]
\label{mt} 
Consider the quenched Glauber dynamics on the $d$-regular configuration model. There exist positive constants $c_1, c, C$ where  $c, C$ may depend on $\beta, B$ and  $c_1$ is independent of $\beta, B$, such that the following statements hold. 
\begin{itemize}
	\item [(ia)] If  $\beta> \beta_c$ and $0 \leq B <  B_c $ then  
		$$ 
		\limsup_{n \rightarrow \infty}\pp(\tmix \geq \exp(cn)) =1.
		$$
		
		
	\item[(ib)]  If $\beta > \beta_c$ and $0\leq B<c_1 \beta$ then $\tmix \geq \exp(c_1 \beta n)$ whp.  	
	
	\item[(ii)] If $\beta > \beta_c$ and $B>B_c^G$, then $\tmix \leq  C n \log n$ whp.
\end{itemize}
The same results hold for the $d$-regular random graph.
\end{theo}
\medskip

It follows directly from our results that $B_c \leq B_c^G$. Surprisingly, we can show that $B_c^G=\hat{B}_c$. This means that the \colrev{critical magnetic field} for the annealed Ising model on $G_n$ is the same as that of the Ising model on the $d$-ary tree, which equals the local limit of $G_n$:

\begin{prop}[Identification of $B_c^G$] \label{pcex}
For any $\beta>\beta_c$,  $B_c^G(\beta)=\hat{B}_c(\beta)$.
\end{prop}
\medskip

However, we do not know whether the quenched quantity $B_c$ is equal to $B_c^G$ or not. In particular, we do not know the quenched behavior  of $\tmix$ when $B_c \leq B \leq B_c^G$.

\subsection{Discussion} 
\label{sec-disc}

Here we give some comments on our results, and state some open problems.

\paragraph{\bf Metastability for Ising models on configuration models.} There are some related works on the metastability of Glauber dynamics on random graphs at zero temperature. In \cite{D,DHJN}, the authors study the Ising model on configuration models. They show that the hitting time to all-plus configuration for the dynamics starting from all minus (which we denote by $T_n$) grows exponentially fast at zero temperature, i.e., $\lim_{\beta \rightarrow \infty} \mubb(T_n \geq \exp(c \beta n)) =1$ whp for the random graph.
\medskip

\paragraph{\bf Cut-off for quenched setting.} In Theorem \ref{amt}(ii), we show the cut-off phenomenon for the dynamics under the annealed law in the high-temperature regime and in the low-temperature regime with high external field. We guess that the same phenomenon occurs in the quenched setting. As far as we know, the best result is due to Lubetzky and Sly \cite{LS}, who prove the cut-off phenomenon for the Ising model on graphs with bounded degrees at sufficiently high temperature (more precisely, $\beta \leq \varepsilon/\Delta$ with $\Delta$ the maximal degree and $\varepsilon$ a universal constant).
\medskip

\paragraph{\bf Extension to configuration models with general degrees.} It is natural to extend our results to the setting of Ising models on configuration models with general degrees. However, it is not immediate how to appropriately define the critical external fields. One may be tempted to conjecture that the definitions in \eqref{def-hbc} and \eqref{def-hbc-que} are still the right critical values for $B$. However, in the non-regular case, this is quite unclear. Indeed, 
to our best knowledge there is no result for the critical external field $B_c^{GW}$ for the uniqueness of Gibbs measure on a Galton-Watson tree (the weak limit of configuration model), see also \eqref{BcG-def}. Thus, extensions of our results to non-regular cases require considerable novel ideas.
\medskip

\paragraph{\bf Slow mixing times and their dependence on $\beta$ for low temperatures.} Fix $\beta$ large. In Theorem \ref{amt}, we show that $\tmixh = \exp(\lambda n(1+o(1)))$ with $\lambda$ bounded by a universal constant for the whole regime of temperatures and external fields. In contrast, Theorem \ref{mt} says that $\tmix \geq \exp(c_1 \beta n)$  \colrev{whp when}  $B\leq c_1\beta$. That means that the annealed dynamics mixes much faster than the quenched dynamics at low temperature (i.e., with large $\beta$).
\medskip

\paragraph{\bf Local spin and graph configurations for wrong magnetizations.} Recall the discussion of \eqref{bounded-valley} below Theorem \ref{amt}(i). It would be of interest to investigate the local neighborhoods and Ising spin configurations around a uniform vertex when there are around $s n$ plusses, for general $s$. It can be expected that the local graph configuration remains on being a $d$-ary regular tree. If the discussion below Theorem \ref{amt}(i) is indeed correct, then the spin configuration either equals the plus or the minus configuration, each with a specific probability. 
\medskip

\paragraph{\bf Organisation of the proof.}
We start in Section \ref{sec-prelim} by \colrev{quantifying} various preliminaries on the critical external field for the Ising model on $d$-regular configuration models, as well as general results on Markov chain mixing times.
We continue in Section \ref{sec-annealed} by discussing the mixing of annealed Ising models and prove Theorem \ref{amt}. This is achieved by comparing the annealed Ising model to a generalized Curie-Weiss model that makes the message that the annealed setting is close to mean-field precise. 
In Section \ref{sec-fixed-spin}, we identify the quenched fixed-spin partition function in Proposition \ref{propbcq}.
In Section \ref{sec-quenched}, we investigate the mixing of quenched Ising models and prove Theorem \ref{mt}.
\colrev{In Section \ref{sec-BcG}, we prove   Proposition \ref{pcex}, i.e. the identification of $B_c^G$.}
In Section \ref{sec:dntn}, we investigate the annealed cut-off behavior stated formally in Proposition \ref{prop:dntn}, and give a sketch of its proof.
\iflongversion
The full proof is given in Appendix \ref{sec:dntn-app}.

\else
The full proof of Proposition \ref{prop:dntn} is given in the appendix of the extended version of this paper \cite{CanGiaGibHof19b}.
\fi

\section{Preliminaries: Critical external fields and mixing times}
\label{sec-prelim}
In this section, we list some preliminary results that are used later. We start by investigating critical externals fields.

\subsection{Critical external fields}
\colrev{The following lemma gives a quantitative definition of the critical external field $B_c^G$. We will omit the proof, which follows by  standard calculations as in  \cite[Proposition 4.5]{MSW}:}

\begin{lem}[Critical external field for uniqueness of the Gibbs measure]
\label{lem-crit-ext-Gibbs}
The critical external field $B_c^G$ for the uniqueness of the Ising Gibbs measure on a $d$-ary tree satisfies
	\begin{equation*}
	B_c^G=\inf \{B \geq 0\colon \textrm{solution to $\theta=-2B+(d-1) \log L_{\beta}(\e^{\theta})$ is unique}\},
	\end{equation*} 
where
	\begin{equation} \label{def-F}
	L_{\beta}(x)=\frac{\e^{2 \beta} x +1}{\e^{2 \beta} +x}.
	\end{equation}
\end{lem}
\medskip

In the next lemma, we  give a characterization of the annealed critical external field $\hat{B}_c$. 

\begin{figure}[h!]
	\centering
	\begin{subfigure}[b]{0.4\linewidth}
		\includegraphics[width=\linewidth]{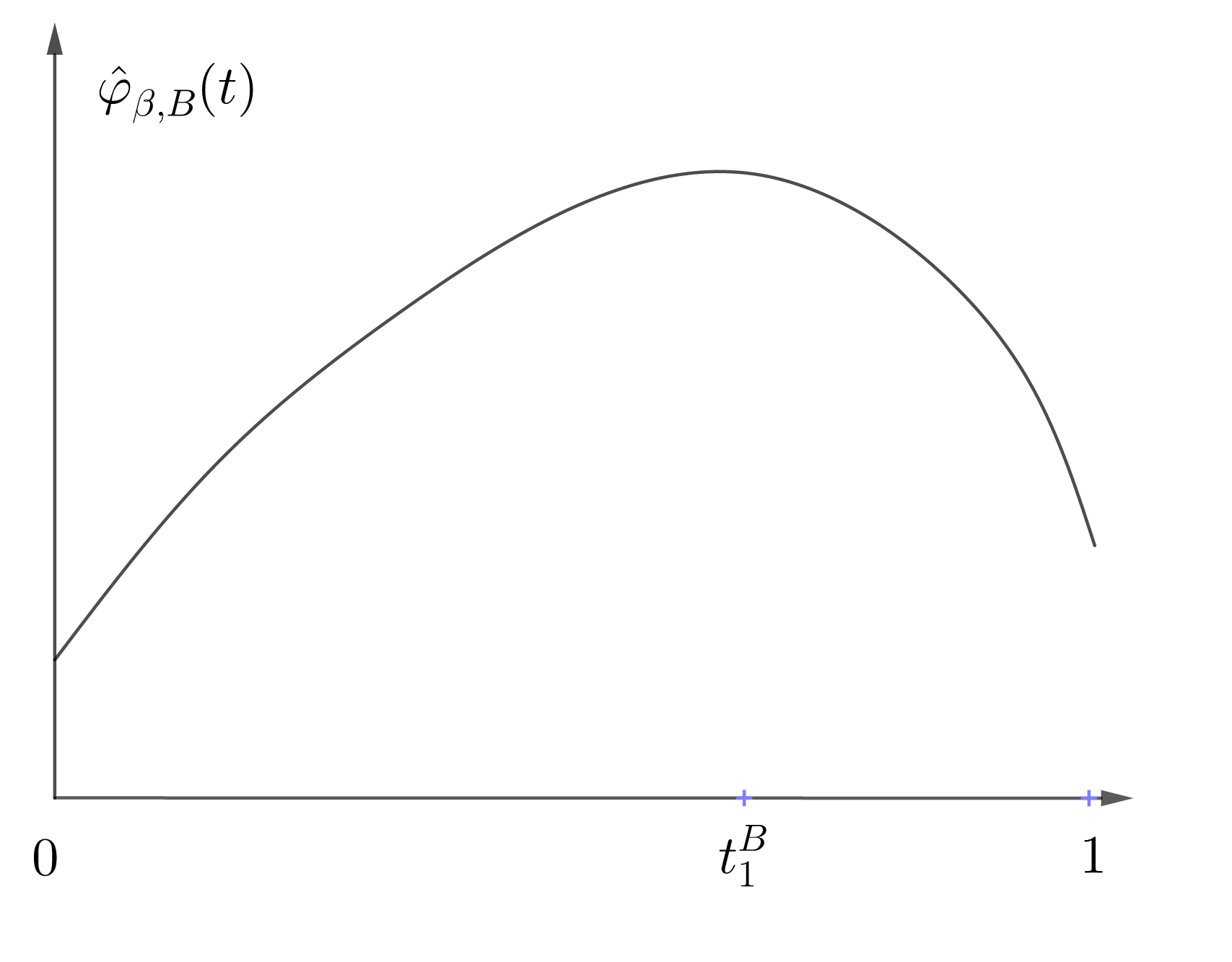}
				\caption{$\beta<\beta_c$, \, or $\beta > \beta_c$, $B > \hat{B}_c$}
	\end{subfigure}
	\begin{subfigure}[b]{0.4\linewidth}
		\includegraphics[width=\linewidth]{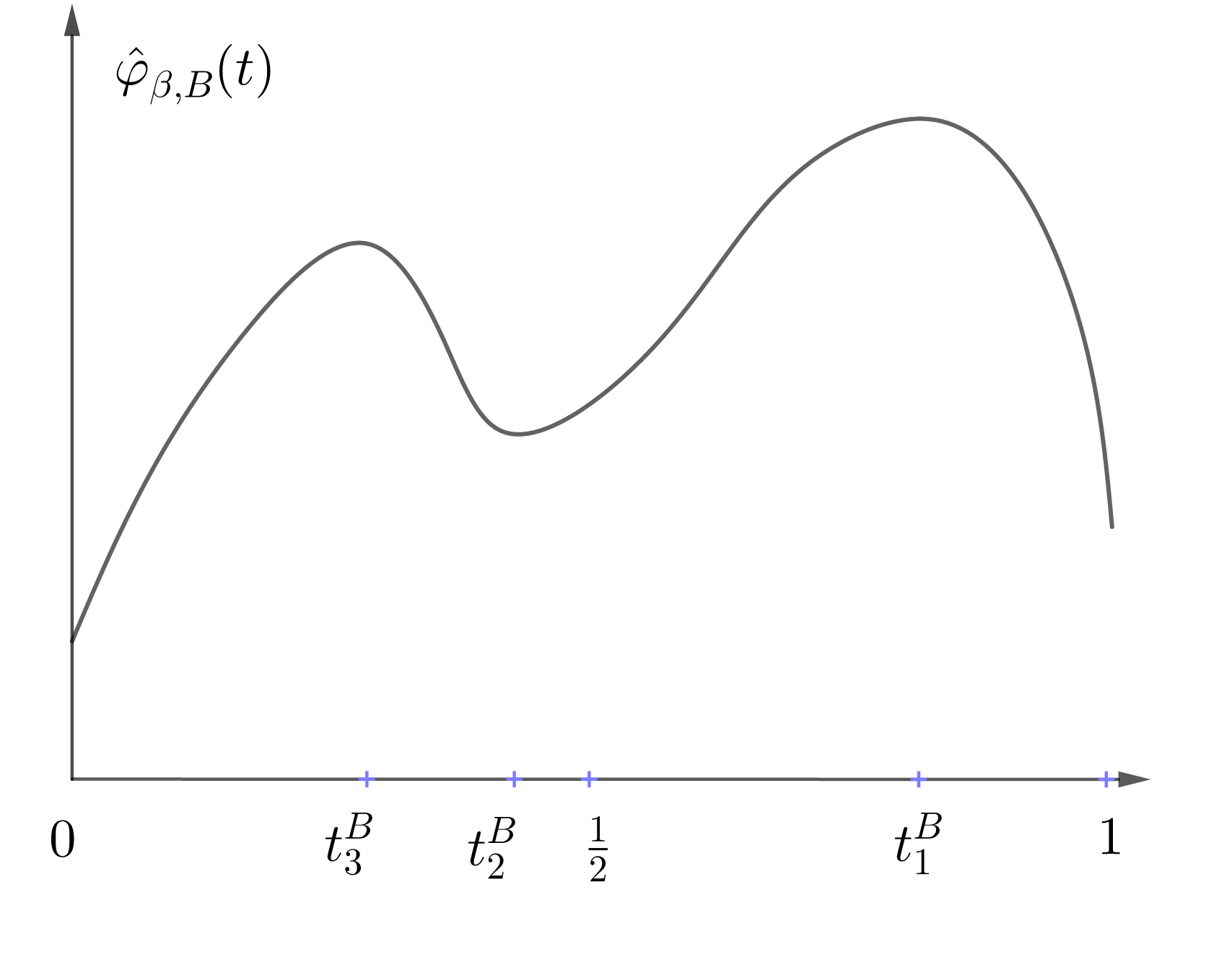}
		\caption{$\beta>\beta_c$,  $0<B<\hat{B}_c$}
	\end{subfigure}
	\caption{Graphs of $\hvarphibb$}
	\label{fig:mag}
\end{figure}
	
\begin{lem}[Critical external field for fast mixing annealed Ising]
\label{ll}
	Fix $\beta > \beta_c$, then the following assertions hold:
	\begin{itemize}
		\item [(i)] For any $B\geq 0$, the function $t\mapsto \hvarphibb''(t)=\hvarphibo''(t)$ has two zeros $t_u$ and $1-t_u$, with 
		\begin{equation} \label{eotu}
		t_u=\frac{1}{2} \Big( 1- \sqrt{1-\frac{4\e^{-4\beta}(d-1)}{(d-2)^2(1-\e^{-4\beta})}} \Big) \in (0, \tfrac{1}{2}).
		\end{equation}
		\item[(iia)] The equation \eqref{def-hbc} holds. Furthermore, $\hvarphibb'(t)$ has three solutions $1>t_1^B>\tfrac{1}{2} \geq t_2^B > t_3^B>0$ for $0\leq B <\hat{B}_c$. In particular, $t_1^B$, $t_2^B$, $t_3^B$ are the global maximizer,  the local  minimizer,  and the local maximizer, respectively, of the function $t\mapsto \hvarphibb(t)$. Moreover, $t_2^0=\tfrac{1}{2}$ for $B=0$. Further, $t_1^0 =1-t_3^0$ and $t_3^0$ are both global maximizers of $\hvarphibo(t)$. In addition, 
			$$
			\sup \{\hvarphibb(t_3^B) -\hvarphibb(t_2^B) \colon \beta > \beta_c, 0\leq B < \hat{B}_c\} < \infty.
			$$
		\item[(iib)] If $\beta <\beta_c$, or $\beta >\beta_c$ and $B>\hat{B}_c$, then $\hvarphibb'(t)=0$ has a unique solution $t_1^B\in (\tfrac{1}{2},1)$, which is the global maximizer of $\hvarphibb(t)$. 
	\end{itemize}
\end{lem}

See Figure \ref{fig:mag} for what $t\mapsto \hvarphibb$ looks like. Figure \ref{fig:mag}(B) gives an example of the setting in Lemma \ref{ll}(iia), while Figure \ref{fig:mag}(A) gives an example of the setting in Lemma \ref{ll}(iib).


\begin{proof}
	It has been shown in the proof of \cite[Theorem 1.1 (ii)]{C} that the equation $\hvarphibb''(t)=\hvarphibo''(t)=0$ is equivalent to 
	\begin{equation*}
	t^2-t+\frac{(\e^{-4\beta}-1)(d-1)}{(d-2)^2(1-\e^{-4\beta})}=0.
	\end{equation*}
The above equation	has two solutions $0<t_u<\tfrac{1}{2}<1-t_u$ as stated in (i), where $t_u$ is given in \eqref{eotu}.

We turn to proving (ii). Since $t_u$ is the unique solution of $\hvarphibo''(t)=0$ in $[0,\tfrac{1}{2}]$ and $\hvarphibo''(0^+)=-\infty$ (since $\hvarphibo''(t)$ equals $I''(t) +$ bounded terms), the function $-\hvarphibo''$ changes its sign from plus to minus at $t_u$ and thus $t_u$ is the maximizer of $-\hvarphibo'$ in $[0,\tfrac{1}{2}]$. Thus, the equation \eqref{def-hbc} holds. Since $\hvarphibb''(t)=0$ has two solutions,  $\hvarphibb'$ has at most three zeros. Moreover, $\hvarphibb'(0^+)=+\infty$ and $\hvarphibb'(1^-)=-\infty$ and $t_u$ is the local minimizer and $1-t_u$ is the local maximizer of $\hvarphibb'$. Therefore, $\hvarphibb'(t)=0$  has a unique solution if $\hvarphibb'(t_u)>0$. Moreover, $\hvarphibb'(t)=\hvarphibo'(t)+2B$. Thus $\hvarphibb'(t_u)>0$ is equivalent to $B > \hat{B}_c$. That gives (iib).

On the other hand, if $B<\hat{B}_c$ (or equivalently $\hvarphibb'(t_u)<0$) then  $\hvarphibb'$  has three solutions $0<t_3^B<t_u<t_2^B<1-t_u<t_1^B<1$. \colrev{Note that $f_{\beta}(t) f_{\beta}(1-t)=1$ for all $t \in (0,1)$.} Thus  for all $0 \leq t\leq 1 $, 
 	\begin{equation} \label{hvbbp}
	\hvarphibb'(t)=\log \left(\frac{1-t}{t}\right)+d\log f_{\beta}(t)+2B.
 	\end{equation}
Thus $\hvarphibb'(\tfrac{1}{2})=2B$, since $f_{\beta}(\tfrac{1}{2})=1$. Therefore, if $B=0$, then $t_2^0=\tfrac{1}{2}$ and if $B>0$ then $t_2^B<\tfrac{1}{2}$. 

It remains to bound the gap  $\hvarphibb(t_3^B) -\hvarphibb(t_2^B)$ uniformly in $\beta>\beta_c$. It has been proved in \cite[Lemma 3.1]{C} that 
	\begin{eqnarray} \label{fohvbb}
	\hvarphibb(t) =  \frac{\beta d}{2}  -B+ I(t) + 2Bt + d t \log f_{\beta}(t) + J(t), 
	\end{eqnarray} 
where $t\mapsto J(t)$ is a uniformly bounded function on $[0,\tfrac{1}{2}]$.
Since $t_3^B$ and $t_2^B$ are solutions of $\hvarphibb'(t)=0$, using \eqref{hvbbp} and \eqref{fohvbb}, we get that 
	\begin{eqnarray*}
	\hvarphibb(t_3^B) &=&  \frac{\beta d}{2}  -B+ I(t_3^B) - t_3^B \log \left(\frac{1-t_3^B}{t_3^B}\right) + J(t_3^B) \\
	&=& \frac{\beta d}{2}  -B - \log (1-t_3^B) + J(t_3^B),
	\end{eqnarray*}
and the same equation holds for $t_2^B$. Hence,
	\begin{equation*}
	|\hvarphibb(t_3^B)-\hvarphibb(t_2^B)| \leq \log 2 + 2 \max_{0\leq t\leq \tfrac{1}{2}} |J(t)|.
	\end{equation*}
This completes the proof of (iia), and thus of the lemma. 
\end{proof}

\subsection{Bottleneck and spectral gap bounds for mixing times}
We next recall a result relating the hitting time and spectral gap of a Markov chain with the bottleneck ratio, which has been proved in the 
book \cite[Theorem 12.3, 12.4 and 13.14]{LPW}:

\begin{lem}[Mixing time bounds]
\label{lbr}
Let $P$ be the transition matrix of a Markov chain on a finite state space $\Omega$ with reversible measure $\pi$ and let $\gamma$ be the spectral gap of this chain. 
\begin{itemize}
	\item [(i)]   In terms of the above notation, 
		\[ 
		\left(\gamma^{-1}-1\right) \log 2\leq \tmix \leq \log \left(\frac{4}{\min_{x \in \Omega} \pi (x)}\right) \gamma ^{-1}. 
		\]
	\item[(ii)] Define the bottleneck ratio as 
		$$
		\Phi^\star= \min \limits_{S\colon \pi(S)\leq \tfrac{1}{2}} \Phi(S),
		\qquad
		\text{where} 
		\qquad
		\Phi(S)=\frac{Q(S,S^c)}{\pi(S)},
		$$
	and
		$$
		Q(x,y)=\pi(x)P(x,y), \qquad Q(K,L)=\sum\limits_{x \in K, y\in L} Q(x,y), \qquad \pi(S)=\sum \limits_{x\in S} \pi(x).
		$$ 
	Then 
		\[
		(\Phi ^\star )^2 \leq \gamma \leq 2 \Phi^\star. 
		\]
As a consequence,
	\[
	\tmix \geq \left(\frac{1}{2\Phi^\star}-1\right) \log 2.
	\]	
\end{itemize}	 
\end{lem}

\section{Mixing of annealed Ising models: Proof of Theorem \ref{amt}}
\label{sec-annealed}

In this section, we consider a generalized Curie-Weiss model whose Hamiltonian depends only on the number of positive spin. This section is organized as follows.
In Section \ref{sec-gen-CW}, we state the model and the main result under some smoothness conditions on the Hamiltonian in Theorem \ref{theo:gcw}. The remainder of the section is devoted to the proof of Theorem \ref{theo:gcw}, as well as on its application to the proof of Theorem \ref{amt}. In Section \ref{subsec:theo3.1(i)}, we prove Theorem \ref{theo:gcw}(i), in Section \ref{poamt(i)}, we use Theorem \ref{theo:gcw}(i) to prove Theorem \ref{amt}(i). A major part of this analysis consists in proving that the Hamiltonian appearing in the annealed Ising model satisfies the requested smoothness conditions. \colrev{We conclude   Section \ref{sec-cut-off-annealed} with the proof of the cut-off phenomenon for generalized Curie-Weiss models in Theorems \ref{theo:gcw}(ii) and \ref{amt}(ii)}.

\subsection{Mixing of a generalized Curie-Weiss model}
\label{sec-gen-CW}

\colrev{We first recall  the transition probabilities  of the Glauber dynamics, which are given in equation \eqref{trsp}. Assume that $\xi_t= \sigma$, let $I$ be a random index in $[n]$ chosen uniformly at random.  Then,}  
\begin{equation*} 
\xi_{t+1} =  
\begin{cases}
\sigma^{+I}  & \textrm{ with probability } \frac{\nu_n(\sigma^{+I})}{\nu_n(\sigma^{+I}) +\nu_n(\sigma^{-I})},\\
\sigma^{-I}  & \textrm{ with probability } \frac{\nu_n(\sigma^{-I})}{\nu_n(\sigma^{+I}) +\nu_n(\sigma^{-I})}.
\end{cases}
\end{equation*}
Assume that we are considering Glauber dynamics on $\Omega_n=\{-1,+1\}^n$ with Hamiltonian given by
	\eqn{
	\label{hamilt-GCW}
	-H_n(\sigma) = n F_n(|\sigma_+|),
	}
for some function $F_n\colon [0,n]\mapsto \R^+$.

For any $0 \leq a < b \leq 1$, let us further consider the Glauber dynamics constrained to the subspace 
 	\[
	\Omega_n^{[a,b]} =\{\sigma \colon [an] \leq |\sigma_{+}| \leq [bn]\}.
	\]
Then, the constrained Gibbs measure \colrev{is defined as} 
	\begin{equation} 
	\label{def:gcw}
	\mu_n^{\beta,B; [a,b]} (\sigma) = \frac{\e^{nF_n(|\sigma_+|)}}{\sum_{k=[an]}^{[bn]}\binom{n}{k}\e^{nF_n(k)}}, \qquad \textrm{for all $\sigma \in \Omega_n^{[a,b]}$}.
	\end{equation} 
Clearly, $\Omega_n=\Omega_n^{[0,1]}$, and $\mu_n^{\beta,B}=\mu_n^{\beta,B; [0,1]}$. 

We first formulate a smoothness condition on $F_n$ that quantifies \colrev{how close} $F_n(k)$ is to $F(k/n)$ for some \colrev{limiting} function $F\colon [0,1]\mapsto \R^+$. For this, we assume that there is a function $F \in \kC^2$ and a 
constant $C$ such that for all 
$[an]\leq k \leq [bn]$,
	\[
	(\textrm{C1}) \hspace{1.5 cm} \Big|F_n(k)-F\left(k/n\right)\Big| \leq \frac{C}{n} \quad \textrm{and} \quad \Big|F_n(k+1)-F_n(k)-\frac{F'(k/n)}{n}\Big| \leq \frac{C}{n^2}.
	\]

Define  
	\begin{equation} 
	\label{def:G}
	G(s)=F(s)+ I(s) \hspace{0.5 cm} \textrm{with} \hspace{0.5 cm} I(s)=(s-1) \log (1-s) -s\log s.
	\end{equation}
Our aim is to show some sufficient conditions on $G$ under which the Glauber dynamics on the generalized Curie-Weiss model exhibits the  metastability or the cut-off phenomenon. For this, we consider the following two further conditions:
	\begin{itemize}
	\item   [(C2)] There exist $a<s_3<s_2<s_1<b$ such that $G$ is strictly increasing in the intervals $(a,s_3)$ and $(s_2,s_1)$ and strictly
	decreasing in the intervals $(s_3,s_2)$ and $(s_1,b)$. 
	\item[(C3)]  The function 
	$G'(s)=0$ has a unique solution $s^\star \in (a,b)$, which is the maximizer of $G(s)$.
	\end{itemize}
We remark that the condition (C3) implies that $G$ is strictly increasing in $(a,s^\star)$ and strictly decreasing in $(s^\star, b)$, $G'(s^\star) =0$ and $G''(s^\star)<0$. Condition (C2) implies that the Glauber dynamics mixes slowly, while condition (C3) implies that the Glauber dynamics mixes quickly, as formalized in the following theorem:

\begin{theo}[Mixing times of generalized constrained Curie-Weiss models]
\label{theo:gcw}
For $0\leq a < b \leq 1$,  consider the Glauber dynamics on the generalized constrained Curie-Weiss model defined by \eqref{def:gcw}.
\begin{itemize}
	\item [(i)] Suppose that (C1) and (C2) hold. Then  there exists a positive constant $C$, such that, for all $n$ large enough,
		\[
		C^{-1} \exp ( \lambda n) \leq \tmix \leq Cn^4 \exp ( \lambda n), 
		\]
	with 	
		\[
		\lambda = \min \{G(s_3),G(s_1)\}-G(s_2)>0.
		\]
	\item[(ii)] Suppose that (C1) and (C3) hold. Then the dynamics exhibits the cut-off phenomenon at $c_\star n \log n$ with window of order $n$, where $c_\star = (2s^\star(1-s^\star)|G''(s^\star)|)^{-1}$.
\end{itemize}
\end{theo}
\medskip

Below we only present the proof for the case $a=0$ and $b=1$, since the proof for the general case is exactly the same.

 \subsection{Slow mixing for generalized Curie-Weiss models: Proof of Theorem \ref{theo:gcw}(i)} 
 \label{subsec:theo3.1(i)}
  Let $(\xi_t)_{t\geq 0}$ be the Glauber dynamics. Define the projection chain $(X_t)_{t\geq 0}$ by 
 	\begin{equation*}
 	X_t=|\{i\colon \xi_t(i)=1\}|.
 	\end{equation*} 
Then $(X_t)_{t\geq 0}$ is a birth-death process on $\{0,\ldots,n\}$ with probability transitions given by 
 	\begin{eqnarray}
 	p_n(k) &=& \pp (X_{t+1}=k+1 \mid X_t =k) =\frac{n-k}{n} \times \frac{\e^{nF_n(k+1)}}{\e^{nF_n(k+1)}+\e^{nF_n(k)}}, \label{pnk} \notag \\
 	q_n(k) &=& \pp (X_{t+1}=k-1 \mid X_t =k) =\frac{k}{n} \times \frac{\e^{nF_n(k-1)}}{\e^{nF_n(k-1)}+\e^{nF_n(k)}}, \label{qnk}\\
 	r_n(k) &=& 1-p_n(k)-q_n(k). \label{rnk} \notag
 	\end{eqnarray} 
The crucial observation for the proof of Theorem \ref{theo:gcw}(i) is that the Glauber dynamics of generalized Curie-Weiss models and their projections have the same spectral gap:

 \begin{prop}[Spectral gap Curie-Weiss and its projection \protect{\cite[Proposition 3.9]{DLP2}}]
 \label{prop:gap}
 	The Glauber dynamics $(\xi_t)_{t\geq 0}$ of the generalized Curie-Weiss model and the projection chain $(X_t)_{t\geq 0}$ have the same spectral gap.
 \end{prop}
 \medskip
 
Combining this result with Lemma \ref{lbr}(i), we can derive bounds for the mixing time of $(\xi_t)_{t\geq 0}$ from the spectral gap of $(X_t)_{t\geq 0}$. The spectral gap of birth-death processes are well understood, as shown in the following proposition:

\begin{prop}[Spectral gaps of birth-death chains \protect{\cite[Theorem 1.2]{CC}}]
\label{prop:gapbd}
	The spectral gap $\gamma$ of an irreducible birth-death chain on $\{0,\ldots,n\}$ with transition probabilities $(p_n(k),q_n(k),$ $r_n(k))$ and stationary measure $\nu_n$   satisfies
		\eqn{
		\label{spec-gap-bound-BDC}
		\frac{1}{4 \ell (i_0)}\leq \gamma \leq  \frac{2}{\ell (i_0)},
		}
	where $i_0$ is the  state such that $\nu_n([0,i_0])\geq \tfrac{1}{2}$ and $\nu_n([i_0,n])\geq \tfrac{1}{2}$, and
		\eqn{
		\label{ell-nu-def}
		\ell (i)=\ell_\nu (i):= \max \left( \max_{j\colon j<i} \sum_{k=j}^{i-1} \frac{\nu_n([0,j])}{\nu_n(k)p_n(k)}, \,\, \max_{j\colon j>i} \sum^{j}_{k=i+1} \frac{\nu_n([j,n])}{\nu_n(k)q_n(k)}  \right), 
		}
	where, for $i\leq j$, 
		$$
		\nu_n([i,j])=\sum_{k=i}^j \nu_n(k).
		$$	
\end{prop} 
\medskip

Now we are ready to give the proof of Theorem \ref{theo:gcw}(i). We investigate the birth-death chain $(X_t)_{t\geq 0}$.  Under condition (C1),
 	\begin{equation} \label{apppq}
 	\Big |\frac{\e^{nF_n(k+1)}}{\e^{nF_n(k+1)}+\e^{nF_n(k)}} -\frac{\e^{F'(k/n)}}{\e^{F'(k/n)}+1} \Big | + \Big |\frac{\e^{nF_n(k-1)}}{\e^{nF_n(k-1)}+\e^{nF_n(k)}} -\frac{1}{\e^{F'(k/n)}+1} \Big |  \leq \frac{C}{n},
 	\end{equation}
for some constant $C>0$. Hence, 
 	\begin{equation} \label{apqnk}
	\frac{n-k}{An}\leq  p_n(k) \leq \frac{A(n-k)}{n} \quad \textrm{and} \quad \frac{k}{An}\leq  q_n(k) \leq \frac{A k}{n},
 	\end{equation}
 for some universal constant $A\ge 1$.   
 The stationary measure of $(X_t)_{t\geq 0}$ is given by 
 	\begin{equation} \label{nunk}
 	\nu_n(k)=\frac{\pi_n(k)}{\pi_n(0) + \cdots + \pi_n(n)},
 	\end{equation}
where (see e.g., \cite[Section 2.5]{LPW} or \cite[(2.3)]{BBF})
 	\begin{eqnarray}  \label{nunk22}
 	\pi_n(k) = \prod_{j=0}^{k-1} \frac{p_n(j)}{q_n(j+1)} = \prod_{j=0}^{k-1} \left(\frac{n-j}{j}\right) \left(\frac{\e^{nF_n(j+1)}}{\e^{nF_n(j)}}\right) =\binom{n}{k}\e^{n[F_n(k)-F_n(0)]}.
 	\end{eqnarray}
It follows from Stirling's formula that  
 	\begin{equation*}
 	\binom{n}{k} =  \exp(n[I(k/n)+o(1)]).
 	\end{equation*}
Combining the last two estimates  with (C1) implies that 
 	\begin{eqnarray} \label{pik}
 	\pi_n(k) = \exp \left(n \Big[ I(k/n) + F(k/n) -F(0) +o(1)\Big]\right)= \exp \left(n \Big[  G(k/n) -G(0) +o(1)\Big]\right),~~~
 	\end{eqnarray}
and thus 
	\begin{eqnarray} \label{pinsk}
	\frac{\pi_n(s)}{\pi_n(k)} = \exp \left(n \Big[ G(s/n)- G(k/n) +o(1) \Big]\right).
	\end{eqnarray} 
Let us define 
 	\begin{eqnarray*}
 	\ell_{1}(i,j) &=& \sum_{k=j}^{i-1} \frac{\pi_n([0,j])}{\pi_n(k)p_n(k)} = \frac{1}{p_n(k)} \sum_{k=j}^{i-1} \sum_{s=0}^{j}\frac{\pi_n(s)}{\pi_n(k)} \qquad \textrm{for } i>j, \\
 	\ell_2(i,j) &=& \sum^{j}_{k=i+1} \frac{\pi_n([j,n])}{\pi_n(k)q_n(k)} = \frac{1}{q_n(k)} \sum_{k=i+1}^{j} \sum_{s=j}^{n}\frac{\pi_n(s)}{\pi_n(k)} \qquad  \textrm{for } i<j.
 	\end{eqnarray*}
Then recall \eqref{ell-nu-def} to see that
	\begin{equation} 
	\label{ll1l2}
	\ell_\nu(i) =\max \big ( \max_{j \colon j<i} \ell_1(i,j), \,\, \max_{j \colon j>i} \ell_2(i,j)\big ).
	\end{equation}
Using \eqref{apqnk} and \eqref{pinsk} we obtain that there exists a positive constant $C$, such that, for all $1\leq j < i \leq n-1$,
 	\begin{equation} \label{l1ij}
 	C^{-1} \exp \left(n L_1(i,j) \right) \leq \ell_1(i,j) \leq C n^3 \exp \left(n L_1(i,j) \right),
 	\end{equation}
where 
 	\[
	L_1(i,j)= \max_{0 \leq s <j \leq k<i} \Big[ G(s/n)- G(k/n) \Big],
	\]
and, for all $1\leq i<j\leq n-1$,
 	\begin{equation} \label{l2ij}
 	C^{-1} \exp \left(n L_2(i,j) \right) \leq \ell_2(i,j) \leq C n^3 \exp \left(n L_2(i,j) \right),
 	\end{equation}
where 
 	\[
	L_2(i,j)= \max_{i<k\leq j \leq s \leq n } \Big[ G(s/n)- G(k/n) \Big]. 
	\]
By assumption (C2), the function $G$ has two local maximizers $s_3$ and $s_1$. We consider the case  that $G(s_3)<G(s_1)$, the other case \colrev{is}  exactly the same.
 
Since $G(s_3)<G(s_1)$,  $s_1$ is the global maximizer, and thus there exist $\varepsilon, \delta>0$, such that  $s_2 < s_1-\delta <s_1+\delta <1$ and 
 	\begin{equation} \label{gt1d}
 	\varepsilon \leq G(s_1)-\max_{|x-s_1|\geq \delta} G(x), \quad \textrm{ and }  \quad G(s_1)-G(s_1 \pm \delta) \leq \lambda/2,
 	\end{equation} 
where 
 	\[
	\lambda=G(s_3)-G(s_2).
	\] 
 By \eqref{pinsk} and \eqref{gt1d}, if $|k-[ns_1]|\geq \delta n$ and $n$ is sufficiently large, then 
 	\begin{equation*}
 	\pi_n(k) \leq \pi_n([ns_1]) \exp(-(\varepsilon +o(1)) n).
 	\end{equation*}
Therefore, $$\nu_n([0,[n(s_1-\delta)]]\cup[[n(s_1+\delta)],n]) \leq 1/4,$$
so that $i_0$ satisfies $[n(s_1-\delta)] \leq i_0 \leq [n(s_1+\delta)]$. If $[n(s_1-\delta)] \leq i_0 \leq [ns_1]$, then, by \eqref{gt1d} and the assumption (C2),
	\begin{eqnarray*}
	\max_{j <i_0} L_1(i_0,j) &=& G(s_3)-G(s_2) + \kO(1/n) =\lambda + \kO(1/n),\\
	\max_{j>i_0} L_2(i_0,j)  &\leq&  G (s_1) -G(s_1-\delta) + \kO(1/n) \leq \lambda/2 + \kO(1/n).
	\end{eqnarray*} 
Similarly, if $[ns_1]<i_0 \leq [n(s_1+\delta)]$, then 
	\begin{eqnarray*}
	\max_{j <i_0} L_1(i_0,j) &=& \max \{G(s_3)-G(s_2), \, G (s_1) -G(s_1+\delta)  \} + \kO(1/n)=\lambda + \kO(1/n),
	\end{eqnarray*}
while 
	\[
	\max_{j>i_0} L_2(i_0,j) = 0.
	\]
	  
Combining the above estimates with \eqref{ll1l2}, \eqref{l1ij} and \eqref{l2ij}, we obtain
	\begin{equation*}
	A^{-1} \exp (\lambda n)\leq \ell_\nu(i_0) \leq A n^3 \exp (\lambda n),
	\end{equation*} 
for some constant $A$. By \eqref{spec-gap-bound-BDC}, the same bound (with a slightly larger $A$) holds for the \colrev{inverse} spectral gap. Hence, using Lemma \ref{lbr}, Propositions \ref{prop:gap} and \ref{prop:gapbd}, and noting that $-\log (\min_{\sigma \in \Omega_n} \mu_n(\sigma)) \asymp n$, we get that the mixing time of the Glauber dynamics on the generalized Curie-Weiss model satisfies
	\begin{equation*}
	C^{-1} \exp (\lambda n)\leq \tmix \leq C n^4 \exp (\lambda n),
	\end{equation*} 
for some constant $C$ independent of $n$. \hfill $\square$

\subsection{Verifying conditions generalized Curie-Weiss model: Proof Theorem \ref{amt}(i)} 
\label{poamt(i)}
By Theorem \ref{theo:gcw}(i), we only need to verify the conditions (C1) and (C2) for the annealed Ising model on random regular graphs.

It is known (see for instance \cite[Lemma 2.1(i) and (3.2)]{C}) that if $|\sigma_+|=j$ then
	\begin{equation*}
	\hat{\mu}_n(\sigma)= \frac{\exp(nF_n(j))}{\sum_{k=0}^n \binom{n}{k} \exp(nF_n(k))}, 
	\end{equation*}
with
	\begin{equation*}
	F_n(k)=\frac{\beta d}{2}+ \frac{1}{n} \log g(dk,dn) + B\left( \frac{2k}{n}-1\right),
	\end{equation*}
 where  $g(k,m)$ satisfies that, for all $0\leq k \leq \ell \leq  m$,
 	\begin{equation} \label{ppog}
 	\Big | \frac{1}{m}\log g(k,m)-  \frac{1}{m}\log g(\ell,m) - \int^{\ell/m \wedge (1-\ell/m)}_{k/m \wedge(1-k/m)} \log f_{\beta} (s) ds\Big | \leq C\frac{|k-\ell|}{m^2},
 	\end{equation}
with $C$ being a universal constant independent of $k,\ell,m$. By \eqref{ppog},
 	\begin{equation*}
 	\Big | F_n(k) - F(k/n) \Big | \leq 
 	\frac{C}{n^2},
 	\end{equation*}
and
 	\eqn{
	\label{(C1b)}
 	\Big|[F_n(k+1)-F_n(k)]- \frac{F'(k/n)}{n} \Big | \leq \frac{C}{n^2},
 	}
where 
	\eqn{
	\label{F-def}
	F(t) = \frac{\beta d}{2} + B(2t-1)+ d \int_0^{t\wedge 1-t} \log f_{\beta}(s)ds,
	}
which implies that  (C1) holds. The function $G(t)=I(t)+F(t)$ is indeed the function $\hat{\varphi}^{\beta,B}(t)$. Hence, Lemma \ref{ll} implies that (C2) holds. 

Finally, by Lemma \ref{ll}(ii) and Theorem \ref{theo:gcw}(i),  
	\eqn{
	\label{bound-annealed-ind-beta}
	\sup_{\substack{\beta > \beta_c\\0\leq B < \hat{B}_c}} \lambda = \sup_{\substack{\beta > \beta_c\\0\leq B < \hat{B}_c}} G(t_3^B)-G(t_2^B) < \infty,
	}
which completes the proof of Theorem \ref{amt}(i).\hfill $\square$

 \subsection{Cut-off generalized Curie-Weiss models: Proof Theorems \ref{theo:gcw}(ii) and \ref{amt}(ii)}
 \label{sec-cut-off-annealed}
For $\gamma >0$, define 
	\begin{equation*}
	T_n^+(\gamma) = c_\star n  \log n + \gamma n, \qquad T_n^-(\gamma) = c_\star n \log n - \gamma n,
	\end{equation*}
where 
	\begin{equation*}
	c_\star = -[s^\star (1-s^\star) G''(s^\star)]^{-1} >0,
	\end{equation*}
since $G''(s^\star)<0$.   Theorem \ref{theo:gcw}(ii) follows from the following proposition:

\begin{prop}[Annealed cut-off behavior]
\label{prop:dntn}
Suppose that Conditions (C1) and (C3) hold. Then,
	\begin{eqnarray*}		
	\lim_{\gamma \rightarrow \infty} \liminf_{n \rightarrow \infty} d_n(T_n^-(\gamma)) =1,
	\qquad
	\lim_{\gamma \rightarrow \infty} \limsup_{n \rightarrow \infty} d_n(T_n^+(\gamma)) =0,
	\end{eqnarray*}
\colrev{where we recall that, for $T\geq 1$,}
\[ d_n(T) = \sup_{A \subset \Omega_n} \sup_{\sigma \in \Omega_n} |\mu_n(A) - \pp_{\sigma} (\xi_{T} \in A)|. \]
\end{prop}
For the proof, we will closely follow the strategy of proving cut-off phenomena for censored Curie-Weiss model  in \cite{DLP}. Since the proof is long and technical, we defer it to Section \ref{sec:dntn}.
\medskip

\noindent
{\it Proof of Theorem \ref{amt}(ii).}~
Thanks to Theorem \ref{theo:gcw}, we only need to show that Conditions (C1) and (C3) hold for the annealed Ising model. The condition (C1) is already proved in Section \ref{poamt(i)}. Condition (C3) follows from Lemma \ref{ll} (iib), since $G(t)=F(t)+I(t)=\hvarphibb(t)$.  \hfill $\square$

\section{Quenched fixed-spin partition function: Proof of Proposition \ref{propbcq}}
\label{sec-fixed-spin}
In this section, we prove Proposition \ref{propbcq}. We start in Section \ref{sec-pressures-rel-graphs} by relating the fixed-spin pressures for different values of the total spin.
We continue in Section \ref{sec-conc-pressure} to prove concentration properties of the fixed-graph pressure, and prove Proposition \ref{propbcq}(i). In Section \ref{sec-conv-mean-pressure}, we show that the mean pressure converges along a subsequence and use this to prove Proposition \ref{propbcq}(ii). We conclude in Section \ref{sec-quen-vs-anne-press} by showing where the quenched and annealed pressures agree using large deviation ideas.

\subsection{Relating pressures with different total spins}
\label{sec-pressures-rel-graphs}
Let $G=(V,E)$ be a graph  with degrees bounded by $d$  and consider the Ising model  on $G$ with  Hamiltonian given by
	 \[
	 H^{\beta,0}_G(\sigma)=-\beta\sum_{(i,j)\in E} \sigma_i \sigma_j = -(\beta |E| -2\beta e(\sigma_{+},\sigma_{-})),
	 \]
where $e(A,B)$ is the number of edges between $A$ and $B$ in $G$. For any $0\leq k \leq |V|$, define 
	\[
	Z^{\beta,0}_{G,k} = \sum_{\sigma \colon |\sigma_{+}|=k} \exp(-H^{\beta,0}_G(\sigma)). 
	\]
We have 
  	\begin{eqnarray*}
 	Z_{G,k+\ell}^{\beta, 0} &=& \e^{\beta |E|} \sum_{\substack{U\subset V \\ |U|=k+\ell }} \e^{-2\beta e(U,U^c)} \notag \\
 	&=& \e^{\beta |E|} \frac{1}{\binom{k+\ell }{\ell}} \sum_{\substack{A\subset V \\ |A|=k}} \quad  \sum_{\substack{B\subset V\setminus A \\ |B|=\ell}} \e^{-2\beta e(A\cup B,(A\cup B)^c)} \notag \\
 	&=& \e^{\beta |E|}  \sum_{\substack{A\subset V \\ |A|=k}} \e^{-2\beta e(A,A^c)} \quad   \frac{1}{\binom{k+\ell }{\ell}}\sum_{\substack{B\subset V\setminus A \\ |B|=\ell}} \e^{2\beta [e(A,B)-e( B,(A\cup B)^c)]}. 
 	\end{eqnarray*}
Using the naive bound
	\begin{equation*}
	-d \ell \leq  -d|B| \leq e(A,B)-e( B,(A\cup B)^c) \leq d |B| \leq d \ell,
	\end{equation*}
we get
 	\begin{eqnarray*}
	 \binom{|V|-k}{\ell} \exp(-2\beta d\ell)\leq \sum_{\substack{B\subset V\setminus A \\ 
 	|B|=\ell}} 
 	\e^{2\beta [e(A,B)-e( B,(A\cup B)^c)]} \leq \binom{|V|-k}{\ell} \exp(2\beta d\ell).
 	\end{eqnarray*}
Therefore,
	\begin{eqnarray} \label{zgklk}
	\e^{-2\beta d\ell} \binom{|V|-k}{\ell}/\binom{k+\ell}{\ell} \leq \frac{Z_{G,k+\ell}^{\beta, 0}}{Z_{G,k}^{\beta, 0}}\leq  \e^{2\beta d\ell} \binom{|V|-k}{\ell}/\binom{k+\ell}{\ell}.
	\end{eqnarray}
This shows that the fixed-spin pressures cannot change too much when changing the value of the total spin, a fact that will prove to be useful when deriving properties of the limiting fixed-spin pressure.

\subsection{Concentration of finite-graph pressure: Proof of Proposition \ref{propbcq}(i)} 
\label{sec-conc-pressure}
In this section, we will use a vertex-revealing process, combined with the Azuma-Hoeffding inequality, to show that the quenched finite graph fixed-spin pressure is whp close to its mean. We start by setting up the necessary notation. Let $G'=(V',E')$ be a graph obtained from $G=(V,E)$ by adding one vertex $v$ with at most $d$ edges between $v$ and $V$. Then,
	\begin{eqnarray*}
	Z^{\beta,0}_{G',k} &=& \e^{\beta |E'|} \sum_{\substack{U'\subset V' \\ |U'|=k }} \e^{-2\beta e(U',V'\setminus U')}  \notag \\
	&=& \e^{\beta |E'|} \Big[ \sum_{\substack{U\subset V \\ |U|=k }} \e^{-2\beta e(U,(V\setminus U) \cup \{v\})} + \sum_{\substack{U\subset V \\ |U|=k-1 }} \e^{-2\beta e(U \cup \{v\},V\setminus U) } \Big].
	\end{eqnarray*} 
We observe that $0 \leq |E'|-|E| \leq  d$, and
	\begin{eqnarray*}
 	0 &\leq &e(U,(V\setminus U) \cup \{v\})-e(U,V\setminus U) \leq d, \\
 	0 &\leq& e(U \cup \{v\}, V\setminus U)-e(U,V\setminus U) \leq d. 
	\end{eqnarray*}
Therefore,
	\begin{eqnarray} \label{zgpg}
	\e^{-2\beta d} Z^{\beta,0}_{G,k} \leq Z^{\beta,0}_{G',k} \leq \e^{\beta d} (Z^{\beta,0}_{G,k}+Z^{\beta,0}_{G,k-1}).
	\end{eqnarray}
It follows from \eqref{zgklk} that
	\begin{eqnarray} \label{zgkk1}
	\frac{Z_{G,k}^{\beta, 0}}{Z_{G,k-1}^{\beta, 0}}, \,\, \frac{Z_{G',k+1}^{\beta, 0}}{Z_{G',k}^{\beta, 0}} 
	\in \Big[\e^{-2\beta d} \frac{|V|-k+1}{k+1}, 
 	\e^{2\beta d} \frac{|V|-k+1}{k}\Big].
	\end{eqnarray}
Combining  \eqref{zgpg} and \eqref{zgkk1} we obtain
	\begin{eqnarray} \label{icmoz}
\e^{-2\beta d} \Big(1\wedge \e^{-2\beta d}
\frac{|V|-k+1}{k+1}\Big) &\leq &	\frac{Z^{\beta,0}_{G',k}}{Z^{\beta,0}_{G,k}}, \,\, \frac{Z^{\beta,0}_{G',k+1}}{Z^{\beta,0}_{G,k}}  \\
& \leq& 
	\e^{\beta d} \Big(1+\e^{2\beta d}\frac{k+1}{|V|-k+1}\Big)\Big(1\vee \e^{2\beta d}
	\frac{|V|-k+1}{k}\Big). \notag
	\end{eqnarray}
For $t\in(0,1)$, we define
	\[
	Z^{\beta,0}_G(t)=\sum_{\sigma \colon |\sigma_{+}|=\lceil t|V|\rceil} \exp(-H^{\beta,0}_G(\sigma)). 
	\]
Since   
$0\le\lceil t|V'| \rceil -\lceil t|V| \rceil \leq 1$, 
the inequality \eqref{icmoz} implies that 
	\begin{equation} \label{zgptzgt}
	\e^{-2\beta d} \Big(1\wedge \e^{-2\beta d}
	\frac{1-t}{t+1}\Big) \leq \frac{Z^{\beta,0}_{G'}(t)}{Z^{\beta,0}_{G}(t)} \leq 
	\e^{\beta d} \Big(1+\e^{2\beta d}\frac{t}{1-t}\Big)\Big(1\vee \e^{2\beta d}
	\frac{2-t}{t}\Big).
	\end{equation} 
For any random graph $G_n$, we may consider $\log Z^{\beta,0}_{n}(t)$ as a martingale with respect to the vertex-revealing filtration (see e.g. \cite[Chapter 11.4]{B} or \cite{DMSS} for the construction of the filtration). Indeed, for $k\in[n]$, add vertex $k\in[n]$ along with its edges to the graph, and let $G_{k,n}$ denote the graph with vertex set $[k]$ and edge set the restriction of $G_n$ to $[k]$. Let $\mathscr{F}_k$ denote the sigma-algebra generated by $(G_{l,n})_{l\leq k}$, and define
	\eqn{
	M_k=\expec[\log Z^{\beta,0}_{n}(t)\mid \mathscr{F}_{k-1}],
	}
Then, $M_n= \log Z^{\beta,0}_{n}(t)$ and $M_0=\expec[\log Z^{\beta,0}_{n}(t)]$. Due to \eqref{zgptzgt}, the increments of the martingale are uniformly bounded by 
	$$
	|M_{k+1}-M_k|\leq \vartriangle =\beta d 
	+\log(1+\e^{2\beta}\frac{t}{1-t})+\log(1\vee \e^{2\beta d}\frac{2-t}{t}).
	$$
Hence, it follows from the Azuma-Hoeffding inequality that
	\begin{equation*}
	\pp \left(|n^{-1}\log Z^{\beta,0}_{n}(t) -n^{-1} \E(\log Z^{\beta,0}_{n}(t))| \geq \varepsilon\right) \leq 
	2\exp (-\varepsilon ^2 n/2 \vartriangle ^2).
	\end{equation*}
In other words, 
	\begin{equation}
	\pp \left( |\varphi_{n}^{\beta,0} (t)-\tilde{\varphi}_{n}^{\beta,0} (t)| \geq \varepsilon \right) \leq  2\exp (-\varepsilon ^2 n/2 \vartriangle ^2),
	\end{equation}
which  proves that 
	\begin{equation*}
	\tilde{\varphi}_{n}^{\beta,0} (t) -  \varphi_{n}^{\beta,0} (t) \xrightarrow{a.s.} 0.
	\end{equation*}
The same statement holds for the sequence $\varphi_{n}^{\beta,B} (t)$ since
	\begin{equation} \label{vphB0}
	\varphi_{n}^{\beta,B} (t)= \varphi_{n}^{\beta,0} (t) + B \left(2 \frac{\lceil nt \rceil}{n} -1\right).
	\end{equation}
This completes the proof of Proposition \ref{propbcq}(i).\qed

\subsection{Convergence of the mean pressure: Proof of Proposition \ref{propbcq}(ii)} 
\label{sec-conv-mean-pressure}
In this section, we prove that there exists a subsequence $\{n_k\}_{k\geq 0}$ along which $t\mapsto \tilde{\varphi}_{n}^{\beta,0} (t)$ converges uniformly to a continuous function $t\mapsto \varphi_{\beta ,0}(t)$. By  \eqref{zgklk}, for $0\leq t<t+s \leq \tfrac{1}{2}$,
	\begin{eqnarray} \label{zntts}
	\e^{-2\beta d \ell} \binom{n-k}{\ell}/\binom{k+\ell}{\ell} \leq \frac{Z_{n}^{\beta, 0}(t+s)}{Z_n^{\beta, 0} (t)}\leq  \e^{2\beta d\ell} \binom{n-k}{\ell}/\binom{k+\ell}{\ell},
	\end{eqnarray}
with $k=\lceil nt \rceil$ and $\ell = \lceil ns \rceil$. By Stirling's formula,
	\begin{eqnarray*}
 	\binom{n-k}{\ell}/\binom{k+\ell}{\ell} =  
 	\exp \left[ (1-t)n I\left(\tfrac{s}{1-t} \right)- (t+s)n I \left(\tfrac{s}{t+s}\right) +\kO(\log n)\right].
	\end{eqnarray*} 
Therefore, with $\varphi_n^{\beta,0}(t)=\tfrac{1}{n} \log Z_n^{\beta, 0}(t)$,
there exist non-random constants $c_1,c_2>0$ such that 
	\begin{equation} \label{vpntst}
	-2\beta ds + J(t,s)  - 
	\frac{c_1(\log n)}{n} 
	\leq \varphi_n^{\beta,0}(t+s) -\varphi_n^{\beta,0}(t) \leq 2\beta ds +  J(t,s) + 
	\frac{c_2(\log n)}{n},
	\end{equation}
where 
	\begin{eqnarray*}
	J(t,s)&=&  (1-t) I\left(\tfrac{s}{1-t} \right)- (t+s) I \left(\tfrac{s}{t+s}\right)\\
	&=&s\log(\tfrac 1{t+s}-1)
	-(1-t)\log(1-\tfrac s{1-t})+t\log(1-\tfrac s{t+s}).
	\end{eqnarray*}
This inequality holds with probability $1$, so the same inequality also holds for $\tvarphibo$, i.e.,
	\begin{equation}
	-2\beta ds + J(t,s)  - 
	\frac{c_1(\log n)}{n} 
	\leq \tvarphibo(t+s) -\tvarphibo(t) \leq 2\beta ds +  J(t,s) 
	+ \frac{c_2(\log n)}{n},
	\end{equation}
We observe the following:
\begin{itemize}
	\item [(O1)] $J(t,s)\geq 0$ for all $0<t<t+s\leq \tfrac{1}{2}$, since the function $xI(1/x)$ is increasing in $x>1$ and $1-t\geq \tfrac{1}{2} \geq t+s$,
	\item[(O2)] $(t,s)\mapsto J(t,s)$ is continuous on 
	$\{(t,s)\colon t,s\in [0,\tfrac{1}{2}], s+t\le \tfrac{1}{2}\}$ (and hence uniformly continuous on this compact set). 
	\item[(O3)] For any $t \in (0,1)$ and for $s$ small, $J(t,s)=
	s\log \left(\tfrac{1-t}{t} \right) + \kO(s^2)$. 
\end{itemize}
By (O2), we conclude that the collection of functions $\{\tvarphibo(t), 0 \leq t \leq \tfrac{1}{2}\}$ is equicontinuous.  
Besides, $\tvarphibo(0)=\beta dn/(2n)=d\beta/2$ for all $n\in {\mathbb N}$. Hence, using Ascoli-Arzel$\rm\grave{a}$'s theorem, 
we can extract a subsequence $(n_k)_{k\geq 1}$  along which  $\{\tvarphibo, n\geq 1\}$ converge uniformly in $[0, \tfrac{1}{2}]$ to a continuous function $\varphi_{\beta ,0} $.  
By symmetry,  $ \tvarphibo(t)= \tvarphibo(1-t) + \kO(1/n)$. Moreover, by \eqref{vphB0}, $ \tilde{\varphi}_n^{\beta, B}(t)= \tvarphibo(t) + B(2t-1) +\kO(1/n)$. 
Therefore, $\{\tilde{\varphi}_{n_k}^{\beta,B}, k\geq 1\}$ converges uniformly in $[0,1]$ to the continuous function $\varphibb$ defined by 
	\begin{equation*}
	\varphibb(t)=\varphibo(t)+B(2t-1),
	\end{equation*}
where $\varphi_{\beta, 0}$ is extended to all the interval $[0,1]$ by $\varphi_{\beta ,0}(t)= \varphi_n^{\beta, 0}(1-t)$  for $t\in [\tfrac{1}{2},1]$.

Recall that
	\begin{equation*}
	B_c:=\sup_{0<t<u \leq \tfrac{1}{2}} \frac{\varphibo(t)-\varphibo(u)}{2(u-t)}.
	\end{equation*}
By (O1), we have $\tvarphibo(t) -\tvarphibo(t+s) \leq 
2 \beta d s + 
c_2(\log n/n)$ 
on $0<t<t+s\le \tfrac{1}{2}$, and thus 
	\begin{equation} 
	\label{bcbd}
	B_c \leq \beta d.
	\end{equation}
By Jensen's inequality, $\E[\log Z_n^{\beta,0}(t)] \leq \log \E[Z_n^{\beta,0}(t)]$ for all $t\in(0,1)$. Therefore,
	\begin{equation} 
	\label{plhp}
	\varphibo(t) \leq \hvarphibo (t).
	\end{equation}
We claim  that 
	\begin{equation} \label{svp=shvp}
	\sup_{0\leq  t\leq \tfrac{1}{2}} \varphibo(t) = \sup_{0 \leq t \leq \tfrac{1}{2} } \hvarphibo (t) = \hvarphibo(t_3^0),
	\end{equation} 
where $t_3^0 \in (0,\tfrac{1}{2})$ is one of two global maximizers of $\hvarphibo$ as stated in Lemma \ref{lbr}. 

Suppose for the moment \eqref{svp=shvp} holds, then, for all $\beta >\beta_c$, 
	\begin{eqnarray} \label{bcghvp}
	B_c \geq \sup_{t_3^0<u \leq \tfrac{1}{2}} \frac{\hvarphibo(t_3^0)-\hvarphibo(u)}{2(u-t_3^0)} >0.
	\end{eqnarray}
We now show that $B_c$ grows linearly in $\beta$ when $\beta$ tends to infinity. We first show that there exists a positive constant $c>0$, such that, for $\beta \geq 12$,
	\begin{equation} \label{zoah}
	\lim_{n\rightarrow \infty} \pp \left( Z_n^{\beta,0}(0) \geq \e^{c\beta n}Z_n^{\beta,0}(\tfrac{1}{2})\right) =1. 
	\end{equation}
Indeed, 
	\begin{eqnarray*}  \frac{Z_n^{\beta,0}(0)}{Z_n^{\beta,0}(\tfrac{1}{2})}& =&   \left(\sum_{\sigma \colon |\sigma_{+}|=\lceil n/2\rceil} \exp(-2\beta e(\sigma_+,\sigma_-))\right)^{-1} \notag \\
	&\geq&  \left( \binom{n}{\lceil n/2\rceil} \exp\Big(-2 \beta i_{n,d} \lceil n/2\rceil\Big)\right)^{-1},
	\end{eqnarray*}
where $i_{n,d}$ is the isoperimetric number of the random regular graph defined as 
	$$
	i_{n,d}:=\inf \Big \{ \frac{e(A,A^c)}{|A|} \colon A\subset [n],\, |A|\leq n/2 \Big \}, 
	$$
where $[n]$ is the set of vertices in $G_n$. Bollob\'as \cite{Bo} showed that 
	\begin{equation} \label{ison}
	\lim_{n \rightarrow \infty} \pp \left( i_{n,d} \geq   \frac{d}{2} - \sqrt{d \log 2}\right) =1. 
	\end{equation}
Therefore, whp,
	\begin{equation}
	\label{ratio-Zns}
	\frac{Z_n^{\beta,0}(0)}{Z_n^{\beta,0}(\tfrac{1}{2})} \geq \exp \left( \big(\beta (\tfrac{d}{2} - \sqrt{d \log 2}\big) -\log 2)n\right) 
	\geq \exp( c \beta n),
	\end{equation}
for some $c>0$. Notice that for the second inequality we have used that $\beta \geq 12$ and $d\geq 3$.

Using \eqref{zoah}, we thus conclude that
	\begin{eqnarray} 
	\label{bc1/2}
	B_c \geq \liminf_{n\rightarrow \infty} \frac{\tvarphibo(0)-\tvarphibo(\tfrac{1}{2})}{1} = \liminf_{n\rightarrow \infty} \frac{\varphi_n^{\beta,0}(0)-\varphi_n^{\beta,0}(\tfrac{1}{2})}{1} \geq c\beta.
	\end{eqnarray}
Combining this with \eqref{bcbd} and \eqref{bcghvp}, we obtain the desired result in Proposition \ref{propbcq}(ii). We are thus left to prove \eqref{svp=shvp}.

Now we prove \eqref{svp=shvp}, which is a direct consequence of Proposition \ref{propbcq}(i) and the fact that
	\begin{equation} 
	\label{vpnt30}
	\varphi_{n}^{\beta,0}(t_3^0) \xrightarrow{{\rm a.s.}} \hvarphibo(t_3^0) =\sup_{t\in [0,\tfrac{1}{2}]} \hvarphibo(t).
	\end{equation}
Furthermore,  \eqref{vpnt30} follows from 
\colrev{ the following two claims}
	\begin{eqnarray} 
	\label{def:ke}
 	\pp \left(	\lim_{n\to\infty}\frac{\log Z^{\beta, 0}_n}{n} =  \hvarphibo(t_3^0) =\sup_{t\in [0,\tfrac{1}{2}]} \hvarphibo(t) 
 	\right)=1,
	\end{eqnarray}
and, for any $\delta >0$ and for all $n$ large enough,
	\begin{equation} 
	\label{znt30}
	\pp \left(\Big | \frac{\log Z_n^{\beta, 0}(t_3^0)}{n} -\frac{\log Z_n^{\beta, 0}}{n}\Big | \geq \delta \right) \leq n^{-2}.
	\end{equation}
The first claim \eqref{def:ke} follows from \eqref{an=qu}, so it remains to prove \eqref{znt30}. By \eqref{znphk}, 
	\begin{equation*}
	\E(Z_n^{\beta, 0}(t)) \leq \exp(n\hvarphibo(t)+ \kO(1) ).
	\end{equation*}
Hence, Markov's inequality implies that for any $\varepsilon >0$
	\begin{equation} \label{mifvpt}
	\pp \left(Z_n^{\beta, 0}(t) \geq \exp(n(\hvarphibo(t) + \varepsilon) ) \right) \leq \exp(-\varepsilon n/2).
	\end{equation}
For $k\in[n]$, define $A_k= \{\sigma \colon |\sigma_+|=k\}$ and write 
	\[
	Z^{\beta,0}_{n,k} = \sum_{\sigma \in A_k} \exp(-H_n^{\beta,0}(\sigma)). 
	\]
Then, by symmetry, 
	\begin{eqnarray} \label{zn1/2}
	\sum_{k\leq \lceil n/2 \rceil} Z_{n,k}^{\beta,0} \geq \frac{Z_n^{\beta,0}}{2}.
	\end{eqnarray}
For any $\delta>0$, we define
	\begin{equation*}
	Z_n^{\beta,0}(t_3^0,\delta^+) = \sum_{k=0}^{\lceil (t_3^0-\delta)n \rceil} Z_{n,k}^{\beta,0} +\sum_{k=[(t_3^0+\delta)n] } ^{\lceil n/2 \rceil}Z_{n,k}^{\beta,0},
	\end{equation*}
and 
	\begin{equation*}
	Z_n^{\beta,0}(t_3^0,\delta^-) = \sum_{k=[(t_3^0-\delta)n] }^{\lceil (t_3^0+\delta)n \rceil} Z_{n,k}^{\beta,0}.
	\end{equation*}
As for \eqref{mifvpt}, using Markov's inequality, we see that for any $\varepsilon>0$
	\begin{equation*}
	\pp \left(Z_n^{\beta,0}(t_3^0,
	\delta^+) \geq \exp \Big ( n \sup \{ \hvarphibo(t)\colon  0\leq t_3^0
 	\leq \tfrac{1}{2}, |t-t_3^0| \geq \delta \} + n \varepsilon \Big ) \right) \leq \exp(-\varepsilon n/3).
	\end{equation*}
Combining this with \eqref{def:ke}, for some $\varepsilon =  \varepsilon(\delta)>0$,
	\begin{equation*}
	\pp(Z_n^{\beta,0}(t_3^0,\delta^+) \geq Z_n^{\beta,0}/4) \leq \exp(-\varepsilon n /4).
	\end{equation*}
Thus, by \eqref{zn1/2},
	\begin{equation} \label{znt3m}
	\pp(Z_n^{\beta,0}(t_3^0,\delta^-) \geq Z_n^{\beta,0}/4 ) 
	\geq 1-\pp(Z_n^{\beta,0}(t_3^0,\delta^+) \geq Z_n^{\beta,0}/4 ) 
	\geq 1- \exp(-\varepsilon n /4).
	\end{equation}
On the other hand, using \eqref{zgklk}, \eqref{vpntst} and (O3), it follows that for given small $\delta>0$ for all  $n$ large enough and for $k$ such that $|k-[t_3^0n]| \leq \delta n$ 
and $(\log n)/n\le \delta$, one has, for some $\tilde c>0$,
	\begin{eqnarray*}
	\Big | \frac{\log Z_n^{\beta, 0}(t_3^0)}{n} -\frac{\log Z_{n,k}^{\beta, 0}}{n}\Big | &\leq& 
	2 \beta \delta d+ J(t_3^0,\delta) + 
	\tilde c\left(\frac{ \log n}{n}\right) 
	\leq  2C \delta,
	\end{eqnarray*}
where $C=\beta d+ \log \left(\tfrac{1-t_3^0}{t_3^0}\right)+\tilde c$. Therefore,
	\begin{eqnarray*}
	\Big | \frac{\log Z_n^{\beta, 0}(t_3^0)}{n} -\frac{\log Z_n^{\beta, 0}(t_3^0,\delta^-)}{n}\Big | \leq  3C \delta.
	\end{eqnarray*}
Combining this with \eqref{znt3m} and the fact that  $Z_n^{\beta, 0}(t_3^0,\delta^-) \leq Z_n^{\beta,0}$, we obtain \eqref{znt30}.
This completes the proof of \eqref{svp=shvp}, and thus of Proposition \ref{propbcq}(ii).
\qed


\subsection{Relating the quenched and annealed pressures}
\label{sec-quen-vs-anne-press}

\begin{figure}[h!] 
	\centering
		\includegraphics[width=0.6\linewidth]{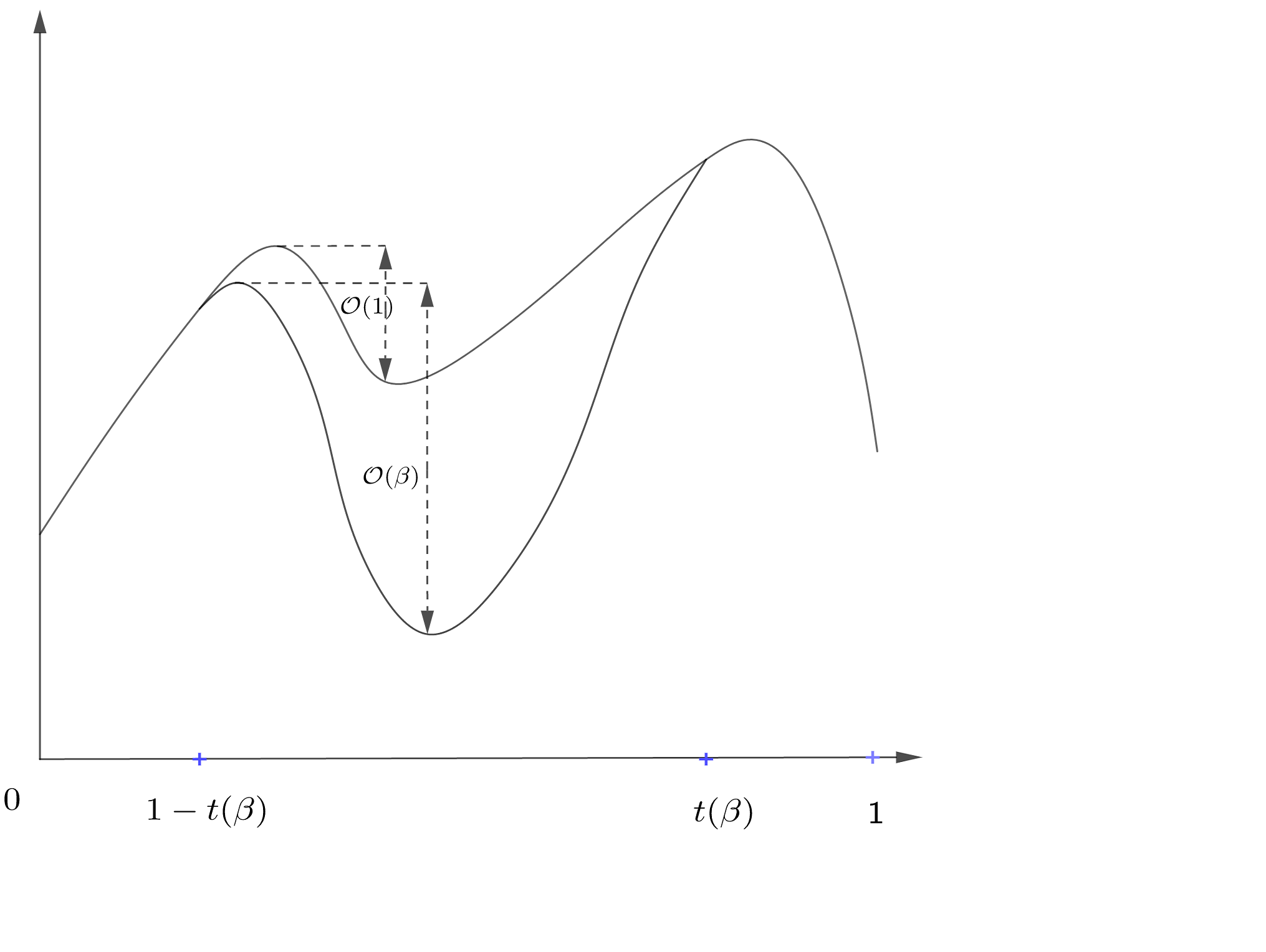}
	\caption{Graphs of $\hvarphibb$ and $\varphi_{\beta, B}$. The two functions are equal in $[0, 1-t(\beta)] \cup [t(\beta),1]$, and in the remaining interval, $\hvarphibb$ is above $\varphi_{\beta,B}$.}
	\label{fig:avq}
\end{figure}

In this section, we relate the pressure in the quenched and annealed settings in more detail. 
 \colrev{Using \cite[Proposition 1.2]{DMSS}, we have $\tfrac{1}{n} \log Z_n^{\beta,B} - \tfrac{1}{n} \E[ \log Z_n^{\beta,B}] \rightarrow 0$ a.s., and thus }
  \begin{equation} \label{vp-vpt}
  \varphi(\beta, B) = \lim_{n\rightarrow \infty} \frac{1}{n} \E [\log Z_n^{\beta,B}] \textrm{ a.s.}
  \end{equation}
 On the other hand, since $Z_n^{\beta,B} =\sum_{i=0}^n Z_n^{\beta,B}(\tfrac{i}{n})$,
   \begin{equation*}
  \frac{1}{n} \max_{0 \leq i \leq n} \log Z_n^{\beta,B}\left(\dfrac{i}{n}\right) \leq \frac{1}{n} \log Z_n^{\beta,B} \leq \frac{1}{n} \left[\log (n+1) + \max_{0 \leq i \leq n} \log Z_n^{\beta,B}\left(\dfrac{i}{n}\right) \right],
   \end{equation*} 
 and hence
   \begin{equation} \label{vpt-max}
 \frac{1}{n}\E[\log Z_n^{\beta,B}] = \max_{0 \leq i \leq n}\tilde{\varphi}_n^{\beta, B}\left(\dfrac{i}{n}\right) + o(1),
 \end{equation}  
 where we recall that $\tilde{\varphi}_n^{\beta, B}(t)= \tfrac{1}{n}\E[\log Z_n^{\beta,B}(t)]$. By Proposition \ref{propbcq} (ii), there is a  subsequence $(n_k)_{k\geq 1}$, such that $(\tilde{\varphi}_{n_k}^{\beta, B})_{k\geq 1}$ converges uniformly to a continuous function $\varphi_{\beta,B}$. Hence
  \begin{equation*}
  \lim_{k\rightarrow \infty} \max_{0 \leq i \leq n_k}\tilde{\varphi}_{n_k}^{\beta, B}\left(\dfrac{i}{n_k}\right) = \max_{0\leq t \leq 1} \varphi_{\beta,B}(t).
  \end{equation*} 
 Combining this with \eqref{vp-vpt} and \eqref{vpt-max}, we obtain that 
 \begin{equation*}
 \varphi(\beta, B) = \max_{0\leq t \leq 1} \varphi_{\beta,B}(t) \textrm{ a.s.}
 \end{equation*}
 In addition, recall from \eqref{an=qu} that $\hat{\varphi}(\beta,B)=\lim_{n\rightarrow \infty} \frac{1}{n} \log \E[Z_n^{\beta,B}]$ exists, and satisfies
 \begin{equation*}
 \varphi(\beta,B) \overset{\textrm{ a.s. }}{=} \hat{\varphi}(\beta,B) =\max_{0\leq t \leq 1} \hat{\varphi}_{\beta,B} (t).
 \end{equation*}
Then, $\varphi_{\beta,B}(t)\leq \hat{\varphi}_{\beta,B}(t)$ for all $t$ by \eqref{plhp}, but 
$Z^{\beta,B}_{n}(t)=\e^{n \varphi_{\beta,B}(t)(1+\op(1))}$ has expected value $\expec[Z^{\beta,B}_{n}(t)]=\e^{n \hat{\varphi}_{\beta,B}(t)(1+o(1))}$. Thus, it is impossible that $\varphi_{\beta,B}(t)<\hat{\varphi}_{\beta,B}(t)$ for all $t\in[0,1]$, while at the same time and instead, these two functions must have the same optimizer and value in their optimum.  Since the optimizer equals $t_1^B$, we thus arrive at the fact that $\varphi_{\beta,B}(t_1^B)=\hat{\varphi}_{\beta,B}(t_1^B)$  holds for all $\beta, B$.


However, we also know that, for every $b$, 
	\[
	\varphi_{\beta,b}(t_1^b)=\varphi_{\beta,B}(t_1^b)+2(b-B)t_1^b, 
	\qquad 
	\hat{\varphi}_{\beta,b}(t_1^b)=\hat{\varphi}_{\beta,B}(t_1^b)+2(b-B)t_1^b.
	\] 
As a result, for all $b$, $\varphi_{\beta,B}(t_1^b)=\hat{\varphi}_{\beta,B}(t_1^b)$ for every $b\in \R$. Now, by varying $b$, we vary $t_1^b$ in a neighborhood of $t_1^B$, which shows that $\varphi_{n}^{\beta,B}(t) \xrightarrow{{\rm a.s.}} \hvarphibb(t)$ for all $t$ in a neighborhood of $t_1^B$. Further, since $t_1^b\rightarrow 1$ as $b\rightarrow \infty$, it follows that 
$\varphi_{\beta,B}(t)=\hat{\varphi}_{\beta,B}(t)$ for all $t\in [t_1^B-\delta,1]$ for some $\delta>0$. 

Next take $\beta>\beta_c$. Then, since $t_1^b\rightarrow t(\beta)\equiv \tfrac{1}{2}(m(\beta)+1)$ as $b\rightarrow 0$, where $m(\beta)$ is the spontaneous magnetization, we conclude that $\varphi_{\beta,B}(t)=\hat{\varphi}_{\beta,B}(t)$ for all $t\in [t(\beta),1]$.
%
Next, we start from $\varphi_{\beta,B}(t_1^b)=\hat{\varphi}_{\beta,B}(t_1^b)$ for all $b$, but now take $b<-B$, so that $t_1^b<\tfrac{1}{2}$. Again, we can take any $b<-B$ in this inequality, and letting $b\searrow -B$ gives $t_1^b\rightarrow 1-t(\beta)$, while for $b\rightarrow -\infty$, $t_1^b\rightarrow 0$. Thus, we conclude also that $\varphi_{\beta,B}(t)=\hat{\varphi}_{\beta,B}(t)$ for all $t\in [0,1-t(\beta)]$. Unfortunately, however, this does not imply that $\varphi_{\beta,B}(t_3^B)=\hat{\varphi}_{\beta,B}(t_3^B)$, since $t_3^B\in (1-t(\beta),\tfrac{1}{2})$.

Since $t(\beta)>\tfrac{1}{2}$ for $\beta>\beta_c$, we cannot conclude that $\varphi_{\beta,B}(t)=\hat{\varphi}_{\beta,B}(t)$ for all $t\in (\tfrac{1}{2}-t(\beta), t(\beta)).$ In fact, this is certainly not true, as we next show.


Indeed, we claim that $\varphi^{\beta,B}(\tfrac{1}{2})\neq\hat{\varphi}^{\beta,B}(\tfrac{1}{2})$ for $\beta$ very large. Indeed, by Lemma \ref{ll}(iia),
	\[
	\hvarphibb(0)-\hvarphibb(\tfrac{1}{2})\leq \hvarphibb(t_3^B)-\hvarphibb(t_2^B) \leq  \lambda,
	\]
which remains finite for large $\beta$. Instead, \eqref{ratio-Zns} shows that 
	\[
	\varphibb(0)-\varphibb(\tfrac{1}{2})\geq c\beta.
	\]
Together with $\varphi_{\beta,B}(0)=\hat{\varphi}_{\beta,B}(0)$ this completes the proof of the claim.\qed

\section{Mixing of quenched Ising models: Proof of Theorem \ref{mt}}
\label{sec-quenched}
In this section, we investigate the mixing of the quenched Ising model on the $d$-regular random graph. We start in Section \ref{sec-quenched-slow} to investigate the slow mixing for $\beta>\beta_c$, $B<B_c(\beta)$. We continue in Section \ref{sec-quenched-fast} to prove the rapid mixing for $\beta>\beta_c$ and $B>B_c^G$.

 \subsection{Slow mixing in the quenched setting: Proof of Theorem \ref{mt}(i)} 
 \label{sec-quenched-slow}
First, we consider the case $\beta \geq 12$.  As we have shown in \eqref{zoah}, whp $Z_n^{\beta,0}(0)\geq \exp(c \beta n) Z_n^{\beta,0}(\tfrac{1}{2})$ for some positive constant $c>0$. Therefore, whp $Z_n^{\beta,B}(0)\geq \exp((c \beta -B)n) Z_n^{\beta,B}(\tfrac{1}{2})$. Hence, for $B \leq c \beta/2$,  whp 
	\begin{equation} 
	\label{zn01/2}
	Z_n^{\beta,B}(0)\geq \exp(c \beta n/2) Z_n^{\beta,B}(\tfrac{1}{2}).
	\end{equation}
\colrev{By Lemma \ref{lbr}, we have} 
  \begin{equation} \label{ltmix}
  \tmix \geq \left( \frac{1}{2 \Phi(U)} -1 \right) \log 2, \qquad \Phi(U) = \frac{Q(U,U^c)}{\mubb(U)},
  \end{equation}
for any $U\subset \Omega_n$ with $\mubb(U) \leq \tfrac{1}{2} +o(1)$. 

We start by proving Theorem \ref{mt}(ib), for which we assume that $\beta>12$ and  consider the bottleneck ratio $\Phi(S)$ of the following set   
 	$$
	S= \{\sigma \colon |\sigma_+| \leq \lceil n/2 \rceil \}.
	$$
 Observe that in each step of the Glauber dynamics, we flip at most one spin. Therefore, after one step, the number of positive spin increases at most by one. Hence,
	\begin{eqnarray} 
	\label{qssc}
	Q(S,S^c)= \sum \limits_{\sigma \in S, \sigma' \in S^c}\!\!\!\mubb (\sigma)P(\sigma, \sigma') = \sum \limits_{\sigma \in A_{\lceil n/2 \rceil}, \sigma' \in A_{\lceil n/2\rceil +1}}
	\!\!\!\!\!\!\mubb(\sigma)P(\sigma, \sigma')  \leq \mubb(A_{\lceil n/2 \rceil}),
	\end{eqnarray} 
\colrev{where, for $k\ge 0$,}
	 \[ A_k = \{ \sigma \in \Omega_n \colon |\sigma_+|=k  \}. \]
Combining this estimate and \eqref{zn01/2} with  the fact that $S$ contains $A_{0}$, we get that, whp,
 	\begin{eqnarray} 
	\label{asos}
	\Phi(S)=\frac{Q(S,S^c)}{\mubb(S)} \leq \frac{\mubb(\sigma \colon |\sigma_+| =\lceil n/2\rceil)}{\mubb(\sigma \colon |\sigma_+|=0)}  =\frac{Z_n^{\beta,B}(\tfrac{1}{2})}{Z_n^{\beta,B}(0)} \leq \exp(-c \beta n/2).
	\end{eqnarray}
To apply \eqref{ltmix}, we need to  show $\mubb(S) \leq  \tfrac{1}{2}+o(1)$ whp. Observe that
	\begin{eqnarray*}
 	\sum_{\sigma\colon |\sigma_{+}| \leq [n/2]} \exp(-H_n^{\beta,B}(\sigma))  &\leq&  \sum_{\sigma \colon |\sigma_{+}| \leq [n/2]} \exp(-H_n^{\beta,0}(\sigma))  
	\leq  \sum_{\sigma\colon |\sigma_{+}| > [n/2]} \exp(-H_n^{\beta,0}(\sigma)) \notag \\
 	&\leq & \sum_{\sigma \colon |\sigma_{+}| > [n/2]} \exp(-H_n^{\beta,B}(\sigma)).
	\end{eqnarray*}
Hence,
	\[
	 \mubb(\sigma \colon |\sigma_+| \leq [n/2]) \leq \tfrac{1}{2}.
	\]
In addition, by \eqref{zn01/2},
	\[
	\mubb (\sigma \colon |\sigma_+| = \lceil n/2\rceil ) = o(1).
	\]
Combining the last two estimates we get $\mubb(S) \leq  \tfrac{1}{2}+o(1)$ whp. 
In conclusion, by \eqref{ltmix} and \eqref{asos}, \colrev{ $\tmix \geq \exp(c \beta n/4)$} whp if $\beta \geq 12$ and $B\leq c \beta/4$.

We now continue to prove Theorem \ref{mt}(ia). Let us fix $\beta >\beta_c$ and  $B<B_c$. Then there exists $0<t<u\leq \tfrac{1}{2}$, such that 
	\begin{equation*}
	\varphibo(t)-\varphibo(u) \geq (u-t)(B_c+B),
	\end{equation*} 
which implies that 
	\begin{equation*}
	\varphibb(t)-\varphibb(u) \geq (u-t)(B_c-B),
	\end{equation*} 
since $\varphibb(t)=\varphibo(t)+B(2t-1)$. Let $(n_k)$ be the subsequence along which $\{\varphi^{\beta, 0}_n\}_{ n\geq 0}$ converges almost surely to $\varphibo$.  
(We can take such a subsequence thanks to Proposition \ref{propbcq}.) Therefore, whp
	\begin{equation} \label{zntu}
	\frac{Z_{n_k}^{\beta,B}(t)}{Z_{n_k}^{\beta,B}(u)} \geq  \exp(n_k(u-t)(B_c-B)/2 ).
	\end{equation} 
Now, by repeating the same arguments for Theorem \ref{mt}(ia) (here we use $S=\{\sigma \colon |\sigma_{+}|=[un]\},$ for which it is obvious that $\mubb(S)\leq \tfrac{1}{2} + o(1)$), 
%
%
%
we obtain that 
	\begin{equation*}
	\lim_{k\rightarrow \infty} \pp (\tmix^{n_k} \geq \exp (c n_k) ) =1,
	\end{equation*}
for some $c>0$, where we write  $\tmix^{n_k}$ for the mixing time considered on the graph $G_{n_k}$.  The above estimate proves Theorem \ref{mt}(ia).
\qed





\subsection{Rapid mixing in quenched setting: Proof of Theorem \ref{mt}(ii)} 
\label{sec-quenched-fast}
We first recall a result in \cite{MS} which gives a sufficient condition for the rapid mixing of Glauber dynamics. 
For a graph $G = (V , E)$ and vertex $v \in V$, we write $B(v, R)$ for the ball of radius $R$ around $v$ and we write
$ S(v, R) = B(v, R) \setminus B(v, R -1)$ for the sphere of radius $R$ around $v$.

\begin{theo}[Fast mixing of Gibbs samplers \protect{\cite[Theorem 3]{MS}}]
\label{rmi}
Let $G$ be a graph on $n \geq  2$ vertices such that there exist constants $
R, T, \chi \geq  1$ such that the following three conditions \colrev{hold} for all $v \in V$:
\begin{description}
	\item[$\rhd$ Volume] The volume of the ball $B(v, R)$ satisfies $|B(v, R)| \leq  \chi$.
	
	\item[$\rhd$ Local mixing] For any configuration $\eta$ on $S(v, R)$, the mixing
	time of the Gibbs sampler on $B(v, R - 1)$ with 
	fixed boundary condition $\eta$ is 
	bounded above by $T$.

    \item[$\rhd$ Spatial mixing] For each vertex 
     $u \in S:=S(v, R)$, define
    	$$
	a_u := \sup_{\eta^+, \eta^-}
    	P \left( \sigma_v = + \mid \sigma_S = \eta^+\right)
    	- P \left( \sigma_v = + \mid \sigma_S = \eta^-\right),
	$$
where the supremum is over configurations $\eta^+, \eta^-$  on $S(v, R)$ differing only at $u$ with $\eta^+_u\equiv +, \eta^-_u \equiv -$.  Then the spatial mixing assumption states that
	\begin{equation*}
	\sum_{u \in S(v,R)} a_u \leq \frac{1}{4}.
	\end{equation*}
\end{description}
Under the above three conditions, there exists a positive constant $C$, such that the mixing time of the Gibbs sampler satisfies $\tmix \leq C T \log (\chi) n\log n$.
\end{theo}

\begin{rem}[Continuous vs.\ discrete time]
\label{rem-cont-disc}
\normalfont 
In fact, \cite[Theorem 3]{MS} applies to the {\em continuous-time} Gibbs sampler and has a factor $n$ less in the bound. However, by \cite[Lemma 15]{MS}, the above bound with the extra factor $n$ holds for the discrete-time system.
\end{rem}
\medskip

The volume and local mixing time conditions are verified for any graph with maximum degree at most $d$ in \cite{MS}, as stated in the following two lemmas:
\begin{lem}[Volume bounds \protect{\cite[Lemma 1]{MS}}]
\label{vol}
	Let $G = (V , E)$ be a graph of maximal degree $d$. Then the volume of $B(v,R)$ is less than 
	$$	
	\chi = 1 + d \sum_{i=1}^R (d-1)^{i-1}.
	$$
\end{lem} 

\begin{lem}[Local mixing bounds \protect{\cite[Lemma 2]{MS}}]
\label{lmi}
  Let $G = (V, E)$ be a graph of maximal degree $d$, and consider the ferromagnetic Ising model on $G$. Then, local mixing holds with 
  	$$
  	T = 80d^3\chi^3\e^{5\beta d(\chi+1)}, \quad \chi = 1 + d \sum_{i=1}^R (d-1)^{i-1}.$$
\end{lem}
\medskip

By Lemmas \ref{vol}--\ref{lmi}, to apply Theorem \ref{rmi}, we only need to prove spatial mixing estimates, as we claim in the following proposition: 
\begin{prop}[Spatial mixing bounds]
\label{smi}
Suppose that $\beta > \beta_c$ and $B > B_c^G(\beta)$. Then there exists a  large $R$ such that the spatial mixing condition holds whp.
\end{prop}   
\medskip

\noindent
{\it Proof of Theorem \ref{mt}(ii).} Theorem \ref{mt}(ii) follows immediately from Theorem \ref{rmi} and Lemma \ref{vol}, \ref{lmi}, and Proposition \ref{smi}.  \hfill $\square$
 
\medskip
It remains to prove Proposition \ref{smi}. We follow the approach in \cite{MS} using Weitz's result to compare the Gibbs measure on graphs to the one on the tree of self-avoiding paths. For any $v \in V$, we denote the tree of paths in $G$ \colrev{starting at $v$} that do not intersect themselves, except possibly at the terminal vertex of the path, by $T_{\saw}:=T_{\saw}(G,v)$. 
By the construction, each path in  $T_{\saw}$ can be naturally mapped to a vertex in $G$ which is the terminal vertex. 
For any set $\Lambda \subset V$, let $\mathrm{W}(\Lambda) \subset T_{\saw}(G,v)$ be the pullback of this natural map. 

Using this, we relate configurations $\eta_\Lambda$ to the corresponding configurations $\eta_{\w(\Lambda)}$ on $\w(\Lambda)$. Furthermore, it holds that $d_G(A,B)=d_{T_{\saw}}(\w(A),\w(B))$. \colrev{Here, $d_G(A,B)$ is the graph distance in $G$ between two vertex sets $A$ and $B$, while with a slight  abuse of notation, we denote the graph distance in $T_{\saw}$ between  the vertex sets of the two sub-trees $\w(A)$ and $\w(B)$ by $d_{T_{\saw}}(\w(A),\w(B))$.}  The following lemma describes a relation between the Glauber dynamics for the Ising model on subsets of $V(G)$ and that on $T_{\saw}$:
 
\begin{lem}[Glauber dynamics on $G$ and on $T_{\saw}$ \protect{\cite[Theorem 3.1]{W}}]
\label{low}
For a graph $G$ and $v \in V(G)$, there exists $A \subseteq T_{\saw}(G,v)$ and a configuration $\nu_A$ on $A$ such that for any $\Lambda \subset V$ and configuration $\eta_{\Lambda}$ on $\Lambda$,
	$$
	P_G(\sigma_v =+ \mid \sigma_{\Lambda}) = P_{T_{\saw}} (\sigma_v =+ \mid \sigma_{\w(\Lambda) \setminus A} = \eta_{\w(\Lambda) \setminus A}, \sigma_A = \nu_A).
	$$
Here, the set $A$ is the set of leaves in $T_{\saw}$ corresponding to the terminal vertices of paths which return to a vertex already visited by the path. 
\end{lem}
\medskip
 
 We note that the construction of $\nu_A$ is described in the proof of 
 \cite[Lemma 13]{MS} and  \cite[Theorem 3.1]{W}. 
 	
When $G$ is $d$-regular then $T_{\saw}(G,v)$ is a $d$-regular tree, denoted by $\T_d$. We now prove the spatial mixing of the Glauber dynamics on $\T_d$ under the uniqueness regime of Gibbs measure, which is the main innovation of this paper in the rapid mixing regime:

\begin{lem}[Spatial mixing in the uniqueness regime]
\label{dtk}
Suppose that $\beta > \beta_c$ and $B>B^G_c(\beta)$.   Let $o$ be the root of $\T_d$, let $U \subseteq \Lambda \subset \T_d$,  and let $\eta^+, \eta^-$ be two configurations that differ only on $U$ with $\eta^+_U \equiv  +, \eta_U^- \equiv -$. Then there exist positive constants $\varkappa,  C, K$ depending only on $\beta$ and $B$,  such that $\varkappa < 1/(d-1)$ and
		$$
		|P(\sigma_o=+ \mid \sigma_{\Lambda} =\eta^+)- P(\sigma_o=+ \mid \sigma_{\Lambda} =\eta^-)| \leq C |U| \varkappa^{d(o,\Lambda)-K}.
		$$
\end{lem} 

\begin{proof} For any $v \in \T_d$ and a boundary condition $\eta$, let us define 
	$$
	R^{\eta}_v=\frac{P(\sigma_v = - \mid \eta)}{P(\sigma_v =+\mid \eta)},
	$$	
where $P(\cdot\mid \eta)$ is the Gibbs measure on the subtree $T_v \subseteq \T_d$ with boundary condition $\eta$ outside $T_v$.	

It is well known that 
	\begin{equation} 
	\label{rvw}
	R^{\eta}_v = \e^{-2B}\prod_{w\prec v} L_{\beta}(R^{\eta}_w),
	\end{equation}
where $w\prec v$ means $w$ is a child of $v$ in $\T_d$, and 
	$$
	L_{\beta}(x)=\frac{\e^{2 \beta} x +1}{\e^{2 \beta} +x}.
	$$ 
Equation \eqref{rvw} can be obtained by taking the conditional expectation of $\sigma_v$ conditionally on $\sigma_w$ for $w\prec v$, and noticing that on a tree these \colrev{expectations factorize}.

\colrev{A boundary condition $\eta$ can be characterized by its  {\it support}, which is defined as  $A_{\eta} =\{v\colon \eta_v \textrm{ is fixed} \}$, and the values of $\eta$ on $A_{\eta}$. Then for any vertex $v$ and boundary condition $\eta$, we set} 
 \begin{equation}
  d(v,\eta) = d(v, A_{\eta}). 
 \end{equation} 
 For any $\ell$, let us denote $R^+_{v, \ell} = R^{\eta^+_{v, \ell}}_v$ (respectively, $R^-_{v, \ell} = R^{\eta^-_{v, \ell}}_v$), where $\eta^+_{v, \ell}$ (resp. $\eta^-_{v, \ell}$) is the plus (respectively minus) boundary condition on $T(v, \ell)$, the subtree of height $\ell$ starting at $v$.   Using \eqref{rvw} and the fact that $L_{\beta}(\cdot)$ is an increasing function, we obtain that, for any boundary condition $\eta$, 
	\begin{equation} 
	\label{mor}
	R^{-}_{v, \ell}\leq R^{\eta}_v \leq R^+_{v, \ell},
	\end{equation}
where $\ell =d(v, \eta)$. By \eqref{rvw}, we have $\log R^+_{v, \ell+1}  = -2B + (d-1)  \log L_{\beta}(R^+_{v, \ell})$ and $\log R^-_{v, \ell+1}  = -2B + (d-1) \log L_{\beta}(R^-_{v, \ell})$. It is well known that if $\beta > \beta_c$ and $B>B_c^G(\beta)$, then  the fixed point equation 
	\begin{equation}
	\label{wibqx}   
	r =e^{-2B}  L_{\beta}(r)^{d-1}
	\end{equation}
has a unique solution, denoted by $r^\star$, and thus  the Gibbs measure is unique (see the proof of \cite[Proposition 4.5]{MSW})
In particular,  as $\ell \rightarrow \infty$
 	\begin{equation*}
 	R^{-}_{v, \ell}, R^{+}_{v, \ell} \rightarrow r^\star.
  	\end{equation*}
 Hence, it follows from \eqref{mor} that 
	\begin{equation*} 
	R^{\eta}_v \rightarrow r^\star \qquad \textrm{ as }\qquad d(v, \eta)  \rightarrow \infty.
	\end{equation*}
In particular, for any $\delta >0$, there exists $K=K(\delta)$, such that if $d(v, \eta) \geq K$, then
	\begin{equation} \label{rrd}
	|R^{\eta}_v-r^\star| \leq \delta.
	\end{equation}
We are now ready to prove the lemma. First we consider the case that $U=\{u\}$. Then $\eta^+$ differs from $\eta^-$ only at $u$. 
Let $o=u_0, \ldots, u=u_{\ell}$ be the geodesic path from $o$ to $u$. We observe that for $\delta >0$ and $K$ as in \eqref{rrd},  if $d(u_1, \eta) \geq K$, then

	\begin{eqnarray} \label{rctr}
	&& |P(\sigma_o=+ \mid \eta^+) - P(\sigma_o=+ \mid \eta^-) | = \Big | \frac{1}{1+ R_o^{\eta^+}} - \frac{1}{1+ R_o^{\eta^-}}\Big |  \notag\\
	&&\quad \leq |R_o^{\eta^+} -R_o^{\eta^-}| = \e^{-2B} \prod_{w \prec o, w \neq u_1} L_{\beta}(R^{\eta}_w) |L_{\beta}(R^{\eta^+}_{u_1})-L_{\beta}(R^{\eta^-}_{u_1})| \notag \\
	&&\quad \leq  \e^{-2B} \prod_{w \prec o, w \neq u_1} L_{\beta}(R^{\eta}_w) \max_{x \in [R^{\eta^-}_{u_1}, R^{\eta^+}_{u_1}]} L_{\beta}'(x) |R^{\eta^+}_{u_1}-R^{\eta^-}_{u_1}| \notag \\
	&&\quad \leq \e^{-2B} \left(\max_{|x-r^\star| \leq \delta} L_{\beta}(x)\right)^{d-1} \left(\max_{|x-r^\star| \leq \delta} L_{\beta}'(x)\right) |R^{\eta^+}_{u_1}-R^{\eta^-}_{u_1}|.
	\end{eqnarray} 
Here, in the second line, we have used \eqref{rvw},  and the fact that the boundary conditions $\eta^+$ and $\eta^-$ constrained on $T_w$ are the same if $w \neq u_1$, while in the third line we have used the mean-value theorem and \eqref{rrd} and 
in the last line, we have used the fact that $L_{\beta}'(x) \geq 0$ for all $x$. Similarly, if $d(u_i, \eta) \geq K$ then 
	\begin{eqnarray} \label{rct}
	|R_{u_i}^{\eta^+} -R_{u_i}^{\eta^-}|  &\leq& \e^{-2B} \prod_{w \prec u_i, w \neq u_{i+1}} L_{\beta}(R^{\eta}_w)  
	\max_{x \in [R^{\eta^-}_{u_{i+1}}, R^{\eta^+}_{u_{i+1}}]} L_{\beta}'(x) |R^{\eta^+}_{u_{i+1}}-R^{\eta^-}_{u_{i+1}}| \notag \\
	&\leq& \e^{-2B} \left(\max_{|x-r^\star| \leq \delta} L_{\beta}(x)\right)^{d-2} \left(\max_{|x-r^\star| \leq \delta} L_{\beta}'(x)\right) |R^{\eta^+}_{u_{i+1}}-R^{\eta^-}_{u_{i+1}}|.
	\end{eqnarray} 
Note that $u_i$ for $i\geq 1$ has only $(d-2)$ children instead of $(d-1)$ as the root. 

Next, we claim that there is $\varkappa < 1/(d-1)$, such that
	\begin{equation} \label{ffp}
	\e^{-2B}L_\beta(r^\star)^{d-2}L_\beta'(r^\star)< \varkappa. 
	\end{equation}
We postpone the proof of this claim to Section \ref{sec-BcG}, since we will use some properties of the equation \eqref{wibqx} proved in Section \ref{sec-BcG}.

We turn to prove the proposition. Using \eqref{rctr}--\eqref{ffp}, for $K$ large enough and $\ell =d(o,\eta)\geq K$, 
 	\begin{eqnarray} \label{pu1}
  	|P(\sigma_o=+ \mid \eta^+) - P(\sigma_o=+ \mid \eta^-) |  &\leq& \sup_{x} L_{\beta}(x) \varkappa^{\ell -K} |R^{\eta^+}_{u_{\ell-K}}-R^{\eta^-}_{u_{\ell -K}}| \notag \\
  	&&\quad \leq \sup_{x} L_{\beta}(x) \varkappa^{\ell -K} \sup_{x}L_{\beta}(x)^{d-1} \notag \\
  	&&\quad \leq \e^{2 \beta d} \varkappa^{\ell -K}.
 	\end{eqnarray}
Note that in the first line, the term $\sup_x L_{\beta}(x)$ appears since in \eqref{rctr} the exponent of $\max_{|x-r^\star| \leq \delta} L_{\beta}(x)$ is $(d-1)$ instead of $(d-2)$ as in \eqref{rct}, and in the second line, we use \eqref{rvw} to obtain the uniform upper bound for $R_u^{\eta}$. 
 
 Now we consider the general case $U=\{u_1, \ldots, u_k\}$. Let $\eta^0, \ldots, \eta^k$ be the interpolation sequence of configurations on $\Lambda$ with $\eta^0=\eta^-$ and $\eta^k=\eta^+$, where $\eta^i$ and $\eta^{i-1}$ differ only at $u_i$ with $\eta^i_{u_i}=+$ and $\eta^{i-1}_{u_i}=-$. By \eqref{pu1}, for all $i$,
  	$$
	|P(\sigma_o=+ \mid \eta^i) - P(\sigma_o=+ \mid \eta^{i-1}) | \leq  \e^{2 \beta d} \varkappa^{\ell -K},
	$$
and thus 
  	$$
	|P(\sigma_o=+ \mid \eta^+) - P(\sigma_o=+ \mid \eta^{-}) | \leq  k \e^{2 \beta d} \varkappa^{\ell -K},
	$$
which proves Lemma \ref{dtk} subject to \eqref{ffp}. 
\end{proof}

With Lemma \ref{dtk} in hand, we are now ready to complete the proof of Proposition \ref{smi}:
\medskip

\noindent
{\it Proof of Proposition \ref{smi}}. For any $v$, let $T$ be the tree of self-avoiding paths starting at $v$. Let $\w(S(v, R))$ denote the vertices in $T$ that correspond to vertices in $S(v, R)$, and, for each $u \in  S(v, R)$, let $ \w (u)$ denote the set of vertices in $T$ that correspond to $u$. Then, by Lemmas \ref{low} and  \ref{dtk},
	\begin{eqnarray} \label{eoau}
	a_u = \sup_{\eta^+, \eta^-} &&\Big |
	P\left(\sigma_v =+ \mid  \sigma_{\w (S(v,R)) \setminus A} =  \eta^+_{\w(S(v,R))\setminus A}, \sigma_A = \nu_A \right)  \notag \\
	&& - P\left(\sigma_v =+ \mid  \sigma_{\w (S(v,R)) \setminus A} = \eta^-_{\w(S(v,R))\setminus A}, \sigma_A = \nu_A \right) \Big | \notag  \\
	& \leq & C |\w(u)|  \varkappa^{R-K},
	\end{eqnarray}
since $d(v, \eta) \geq R$. We observe that for any fixed $R$, 
	\begin{equation*}
	\lim_{n \rightarrow \infty} \pp \left(B(v,R) \textrm{ has at most $1$ cycle for all } v \in [n]\right) =1, 
	\end{equation*}
since by the construction of $d$-regular configuration models using the random pairing of half-edges as described in Section \ref{sec-RRG}, for each $v\in[n]$, the probability that $B(v,R)$ has at least $2$ cycles is of order $(d^R/n)^2 =\kO(n^{-2})$. Assume that $B(v,R)$ has at most $1$ cycle. Then  every $u \in S(v, R)$ appears  at most twice in the tree of self-avoiding paths, which gives $|\w(u)| \leq  2$. Therefore, \eqref{eoau} implies that 
	\begin{equation} 
	\label{soau}
	\sum_{u \in S(v,R)} a_u \leq 2 C (d-1)^R \varkappa^{R-K} \leq \tfrac{1}{4},
	\end{equation}
for some $R=R(K)$ large enough, since $\varkappa < 1/(d-1)$. This completes the proof of Proposition \ref{smi}. \hfill $\square$

\begin{rem}[Spatial mixing and cycles]
\emph{The inequality  \eqref{soau} still holds if $|\w(u)| \leq \varkappa_1^R$, with some $\varkappa_1 < 1/((d-1)\varkappa)$ and $R$ large. Therefore, spatial mixing holds for graphs  satisfying that  the ball $B(v,R)$ has at most $\varkappa_1^R$ cycles for every $v\in V$.}
\end{rem}

\section{Identification of $B_c^G$: Proof of Proposition \ref{pcex}}
\label{sec-BcG}
We recall from \eqref{hvbbp} that the equation of annealed critical points is $\hvarphibb'(t)  =0$, or
     \begin{equation} \label{eqac}
         \log \left(\frac{1-t}{t}\right) +2B+  d \log f_{\beta}(t) =0,
     \end{equation}
where $f_\beta$ is as in \eqref{dnwofbeta}.

On the other hand, the equation \eqref{wibqx}  can be rewritten as 
	\begin{equation} \label{eqgr}
	\log( r L_{\beta}(r)) +2 B - d  \log L_{\beta}(r) =0.
	\end{equation}    
We aim to show that the two equations \eqref{eqac} and \eqref{eqgr} are equivalent. We consider the following change of variable $t=p(r)$, where
	\begin{equation*}
	 p(r)= \frac{1}{r L_{\beta} (r) +1 } = \frac{\e^{2 \beta} +r}{\e^{2\beta}r^2+2r+\e^{2\beta}}.
	\end{equation*}
We remark that the function $p(r)$ is strictly decreasing and is a bijective map from $(0, \infty)$ to $(0,1)$. Then   
	\begin{equation} \label{rpt}
	r =p^{-1}(t)= \frac{\e^{-2\beta}(1-2t) + \sqrt{1+(\e^{-4\beta}-1)(1-2t)^2} }{2t} = 
	\frac{f_{\beta}(t)(1-t)}{t}.
	\end{equation}
Thus
	\begin{equation} \label{eqtr}
	\frac{1-t}{t} = r L_{\beta} (r), \quad f_{\beta}(t) = \frac{t r}{1-t} = \frac{1}{L_{\beta}(r)}.
	 \end{equation}
Combining \eqref{eqac}, \eqref{eqgr} and \eqref{eqtr}, we obtain that the two equations \eqref{eqac} and \eqref{eqgr} are equivalent. In particular,
	$$
	\hat{B}_c= B_c^G,
	$$ 
since 
	\begin{eqnarray*}
	\hat{B}_c &=& \inf \{B >0\colon \textrm{ equation \eqref{eqac} has a unique solution} \}, \\
	B_c^G  &=& \inf \{B >0\colon \textrm{ equation \eqref{eqgr} has a unique solution} \}.
	\end{eqnarray*}
This completes the proof of Proposition \ref{pcex}. \hfil $\square$
\vspace{0.3 cm}
  
\noindent  
{\it Proof of   \eqref{ffp}}. We need to prove that if $B > B_c^G$ then
	\begin{equation} \label{ffpp}
	\e^{-2B}L_\beta(r^\star)^{d-2}L_\beta'(r^\star)< \frac{1}{d-1},
	\end{equation}
where $r^\star$ is the unique solution to \eqref{wibqx}. Let us consider the function 
	\[ 
	g(r) =r - \e^{-2B}  L_{\beta}(r)^{d-1}.  
	\]
Then $g(0)<0$ and $r^\star$ is the unique solution to the equation $g(r)=0$.  Therefore $g'(r^\star)\geq 0$, since otherwise the equation $g(r)=0$ has at least two solutions. 
In other words,
	\begin{equation}
	0 \leq g'(r^\star) 
	= 1- (d-1)\e^{-2B} L_{\beta} (r^\star)^{d-2} L_{\beta}'(r^\star).
	\end{equation} 
Hence, to prove \eqref{ffpp}, we only need to show that $g'(r^\star) \neq 0$. 
  
Assume that $g'(r^\star) =0$. We shall show that this assumption leads to a contradiction.  Since $B>B_c^G=\hat{B}_c$, the equation $\hvarphibb'(t)=0$ has a unique solution $t^\star \in (\tfrac{1}{2},1)$ (which is indeed $t_1^B$). As we have shown in \eqref{rpt}, the unique solution $r^\star$ of  \eqref{eqgr} satisfies
\begin{equation} \label{rs01}
r^\star = p^{-1}(t^\star) \in (0,1), 
\end{equation}
since $t^\star \in  (\tfrac{1}{2},1)$. 

Suppose that $g''(r^\star)<0$. Then there exists $\epsilon>0$ such that $r^\star + \epsilon <1$ and $g''(x)<0$ for all $x\in [r^\star,r^\star +\epsilon]$. Hence, by the Taylor expansion, we have $g(r^\star + \epsilon) =g(r^\star) + \epsilon g'(r^\star) + \tfrac{1}{2} g''(r^\star) x^2 <0$ for some $x \in (r^\star,r^\star +\epsilon)$. Moreover, $g(1)=1-\e^{-2B}>0$. Thus the equation $g(r)=0$ has a solution in $(r^\star +\epsilon,1)$, so  this equation has at least two solutions in $(0,1)$, which is  a contradiction. Using the same argument and the fact that $g(0)<0$, we can also prove that 
$g''(r^\star) >0$ leads to a 
contradiction.  
Therefore, we have $g''(r^\star)=0$.

Since  $g'(r^\star)=g(r^\star) =0$, 
	\begin{eqnarray*}
	1 = \e^{-2B}(d-1)L_{\beta}(r^\star)^{d-2} L_{\beta}'(r^\star), \qquad  r^\star = \e^{-2B}L_{\beta}(r^\star)^{d-1}.
	\end{eqnarray*}
Thus 
	\begin{eqnarray} \label{g=g'=0}
	r^\star = \frac{L_{\beta}(r^\star)}{(d-1) L_{\beta}'(r^\star)}  = \frac{(\e^{2\beta} r^\star +1)(\e^{2\beta}+r^\star)}{(d-1)(\e^{4\beta}-1)}.
	\end{eqnarray} 
On the other hand, $g''(r^\star)=0$ implies that 
  \[ 
   (d-2) L'(r^\star)^2  +L(r^\star)L''(r^\star) =0,
   \]
or equivalently
  \begin{equation} \label{g''=0}
  (d-2)(\e^{4\beta} -1) -2(\e^{2 \beta} r^\star +1) =0.
  \end{equation}
 Combining  \eqref{g=g'=0} and \eqref{g''=0} yields that 
 \[ 
 r^\star = \frac{(d-2)(\e^{2\beta} + r^\star)}{2(d-1)} \quad \Rightarrow \quad r^\star = \frac{(d-2)\e^{2\beta}}{d} >1,
  \]
 since $\beta >\beta_c =\textrm{atanh}(1/(d-1))$. This is contrary to \eqref{rs01}.
 
In conclusion,  $g'(r^\star) >0$ and thus  \eqref{ffpp} (or equivalently \eqref{ffp}) follows.  \hfill $\square$

 \begin{rem} [Relation constrained annealed measure and $d$-ary tree with minus boundary condition]
 \emph{	
 Fix $\beta > \beta_c$ and $0 \leq B < \hat{B}_c$. Then the curve of $\hvarphibb$ has three critical values $t_3^B<t_2^B<t_1^B$. Let $\hat{\mu}_n^-$  be the annealed measure restricted to $\Omega_n^-=\{ \sigma\colon |\sigma_{+}| \leq \lceil t_2^B n \rceil \}$. More precisely, 
	\begin{equation} 
	\hat{\mu}_n^- (\sigma) = \frac{\E[\exp(-H_n (\sigma))]}{\E[Z_n^-]} \indic{\sigma \in \Omega_n^-}, 
	\end{equation}
where 
 	$$
	Z_n^-=\sum_{\sigma \in \Omega_n^-} \exp(-H_n(\sigma)).
	$$ 
Here we omit the superscript $\beta, B$ to simplify the notation. Using the same arguments in \cite[Theorem 1.1]{C}, we can prove that 
	\begin{equation} 
	\label{pamag}
	\lim_{n \rightarrow \infty} \E_{\hat{\mu}_n^-} \left(\frac{\sigma_1 + \ldots+ \sigma_n}{n}\right) =2t_3^B-1.
	\end{equation}
Now we consider the magnetization of the root, denoted by $\rho$, in $\T_d$ with minus boundary conditions at level $\ell$. Recall from \eqref{rvw} that, as $\ell \rightarrow \infty$,
	$$
	R^{\eta}_{\rho, \ell }=\frac{P(\sigma_{\rho} = - \mid \eta^-_{\rho, \ell})}{P(\sigma_\rho =+\mid \eta^-_{\rho, \ell})} \rightarrow \e^{-2B}L_{\beta}(r_-)^d,
	$$	
where $\eta^-_{\rho, \ell}$ is the minus boundary condition on the level $\ell$ of $\T_d$ and $r_-$ is a solution to \eqref{wibqx}, or equivalently $r_-=\e^{-2B}L_{\beta}(r_-)^{d-1}$.  Therefore,
	\[
	\lim_{\ell \rightarrow \infty}  R^{\eta}_{\rho, \ell }= r_-L_{\beta}(r_-),  
	\]
and thus
	\begin{equation} 
	\label{gmag}
 	\lim_{\ell \rightarrow \infty} \langle \sigma_{\rho} \rangle_{\eta^-_{\rho, \ell}} = \lim_{\ell \rightarrow \infty} \frac{1-R^{\eta}_{\rho, \ell }}{1+R^{\eta}_{\rho, \ell }} 
	=  \frac{1-r_- L_{\beta}(r_-)}{1+r_- L_{\beta}(r_-)}.
	\end{equation}
As we have shown above, $t_3^B$ and $r_-$ are related by
 	\[
	t_3^B= \frac{1}{r_- L_{\beta}(r_-)+1}.
	\]
Combining this with  \eqref{pamag} and \eqref{gmag}, we obtain that
	\begin{equation}
	\lim_{n \rightarrow \infty} \langle \sigma_{U_n} \rangle_{\hat{\mu}_n^-}  =  \lim_{\ell \rightarrow \infty} \langle \sigma_{\rho} \rangle_{\eta^-_{\rho, \ell}}, 
	\end{equation} 
where $U_n$ is a random vertex chosen uniformly in $[n]$. We conclude that restricting the annealed measure to $\Omega_n^-$ has the same effect in the large $n$ limit as the minus boundary conditions for the Ising model on the $d$-ary tree. 
}
\end{rem}

\section{Annealed cut-off behavior: Proof of Proposition \ref{prop:dntn}} 
\label{sec:dntn}
Let us recall  the statement of Proposition \ref{prop:dntn}.   For any $\gamma >0$, we define 
 	\begin{equation*}
 	T_n^+(\gamma) = c_\star n  \log n + \gamma n, \qquad T_n^-(\gamma) = c_\star n \log n - \gamma n,
 	\end{equation*}
where 
 	\begin{equation*}
 	c_\star = -[s^\star (1-s^\star) G''(s^\star)]^{-1} >0.
 	\end{equation*}
The aim of Proposition \ref{prop:dntn} is to prove that, under Conditions (C1) and (C3),
  	\begin{equation} \label{dntm}
 	\lim_{\gamma \rightarrow \infty} \liminf_{n \rightarrow \infty} d_n(T_n^-(\gamma)) =1,
 	\end{equation}
and
 	\begin{equation} \label{dntp}
 	\lim_{\gamma \rightarrow \infty} \limsup_{n \rightarrow \infty} d_n(T_n^+(\gamma)) =0.
 	\end{equation}
We give the proof  in three steps. In the first step in Section \ref{sec-proj-chain}, we  analyse the projection chain  $(X_t)_{t\geq 0}$ defined in Section \ref{subsec:theo3.1(i)}.  In the second step in Section \ref{sec-proof-dntm}, we prove \eqref{dntm} and in the final step in Section \ref{sec-proof-dntp}, we prove \eqref{dntp}. 
 
\subsection{A careful analysis of the projection chain} 
\label{sec-proj-chain}
The projection chain $(X_t)_{t\geq 0}$ has been defined as 
  	\begin{equation*}
 	X_t=|\{i \colon \xi_t(i)=1\}|,
 	\end{equation*} 
where $(\xi_t)_{t\geq 0}$ is the Glauber dynamics. Then $(X_t)_{t\geq 0}$ is a birth-death process on $\{0,\ldots,n\}$ with probability transitions given in \eqref{qnk}.  The drift $R_t$ of $(X_t)_{t\geq 0}$ satisfies
 	\begin{eqnarray} \label{def:rt}
 	R_t &:=&\E\left[X_{t+1}-X_t \mid X_t\right] \\
 	&=& \frac{n-X_t}{n} \times \frac{\e^{nF_n(X_t+1)}}{\e^{nF_n(X_t+1)}+\e^{nF_n(X_t)}} -\frac{X_t}{n} \times \frac{\e^{nF_n(X_t-1)}}{\e^{nF_n(X_t-1)}+\e^{nF_n(X_t)}}. \notag 
 	\end{eqnarray}
Then, by \eqref{apppq},
 	\begin{equation} \label{drrt}
 	\Big|R_t-R\left(X_t/n\right)\Big| \leq \frac{C}{n},
 	\end{equation}
where 
 	\begin{equation}\label{defR}
 	R(s)=(1-s)\frac{\e^{F'(s)}}{\e^{F'(s)}+1}-s\frac{1}{\e^{F'(s)}+1} = \frac{\e^{F'(s)}}{\e^{F'(s)}+1} -s.
 	\end{equation}
Combining \eqref{def:rt} and \eqref{drrt}, we get 
 	\begin{equation} 
	\label{ecxt}
 	X_t + R(X_t/n) - C/n \leq \E[X_{t+1} \mid X_t] \leq X_t + R(X_t/n) + C/n.
 	\end{equation}
We recall that $G''(s^\star)<0$  and $G'(s^\star) =0$. Therefore, $I'(s^\star) +F'(s^\star) =0$, or equivalently $s^\star = \e^{F'(s^\star)}/(\e^{F'(s^\star)}+1)$. Thus, $R(s^\star) =0$ and, using that $G''(s)=F''(s)-[s(1-s)]^{-1}$,
	\begin{equation} \label{rrpst}
 	R'(s^\star) = \frac{\e^{F'(s^\star) }  F''(s^\star)}{(\e^{F'(s^\star)}+1)^2} - 1 = s^\star (1-s^\star) F''(s^\star) -1 = s^\star (1-s^\star) G''(s^\star)=-\frac{1}{c_\star} <0.
 	\end{equation}
By \eqref{rrpst} and $R(s^\star) =0$, we can find a positive constant $\delta$ such that 
 	\begin{equation} 
	\label{def:del}
 	R'(s)<0 \textrm{ for all }  s \in (s^\star - 2\delta, s^\star + 2\delta) \subset (0,1).
 	\end{equation}
We define, for $\delta \in (0,1)$
	\begin{equation}
	\label{Omega-delta}
	\Omega_n^{\delta}=\Big \{\sigma \in \Omega_n\colon s^\star - \delta  \leq  \frac{|\sigma_{+}|}{n} \leq s^\star + \delta \Big  \}.
	\end{equation}
Before going into the full details, let us start by summarizing the main steps in the analysis of the projection chain $(X_t)_{t\geq 0}$. Thanks to the assumption that the function $G$ only has one local maximizer at $s^\star$, the measure of $(X_t)_{t\geq 0}$ is highly concentrated around  $ns^\star$. Moreover, we will show in Lemma \ref{lem:hit} that the chain $(X_t)_{t\geq 0}$ accesses $\Omega_n^{\delta}$ very quickly, namely, after a time of order $n$. Then we consider the chain after entering $\Omega_n^{\delta}$. Thus, next assume that $X_0 \in \Omega_n^{\delta}$. Then we will show that the drift function $R(s)$ plays a central role to describe the time for $X_t$ to reach $ns^\star$. Roughly speaking, in Lemma \ref{lem:xt-ns} below, we show that 
	\[
	Y_{t+1} -Y_t \approx \left(1+ \frac{R''(s^\star)}{n}\right) Y_t,
	\]
where $Y_t=X_t-ns^\star$. Notice that $R''(s^\star)<0$, and, therefore, when $Y_0 \in [-\delta n, \delta n]$, the time for $Y_t$ to reach $[-\kO(\sqrt{n}), \kO(\sqrt{n})]$ is $(-R''(s^\star))^{-1} n \log n \pm \kO(n)$. Finally, we prove in Lemma \ref{lem:h1} that if the starting points of $(X_t)_{t\geq 0}$ and $(\tilde{X}_t)_{t\geq 0}$ are both in $[ns^\star - \kO(\sqrt{n}), ns^\star + \kO(\sqrt{n})], $ then the two chains mix after a time of order $n$. Combining the above steps, we conclude that the mixing time of $(X_t)_{t\geq 0}$ concentrates around $c_\star n \log n$ (note here that $c_\star =(-R''(s^\star))^{-1}$) with a window of order $n$.
\medskip

We now present the details of the argument. We start by proving that $\Omega_n^{\delta}$ is hit quickly:
 \begin{lem}[$\Omega_n^{\delta}$ is hit quickly]
 \label{lem:hit} For $\delta$ as in \eqref{def:del}, \colrev{there exists a positive constant} $C=C(\delta)$, such that
 	\begin{equation*}
 	\sup_{\sigma \in \Omega_n} \pp_{\sigma}\left(\inf\{t \colon \xi_t \in \Omega_n^{\delta} \} \geq C n\right) =\kO(1/n).
 	\end{equation*}	
 \end{lem}
\medskip

\iflongversion
Below, we give a sketch of the proof, the full proof can be found in Appendix \ref{sec-proj-chain-app}.

\else
Below, we give a sketch of the proof, the full proof can be found in \cite[Appendix A.1]{CanGiaGibHof19b} of the extended version.
\fi

\begin{proof} Since $G$ is strictly increasing in $(0,s^\star)$ and strictly decreasing in $(s^\star, 1)$, there exists \colrev{a positive constant} $\varepsilon$, such that
 \begin{equation} \label{gped}
 G'(s) \geq \varepsilon \quad \textrm{if} \quad s\in (0,s^\star -\delta) \quad \textrm{and} \quad  G'(s) \leq -\varepsilon \quad \textrm{if} \quad s\in (s^\star +\delta, 1). 
 \end{equation}	
 	As in \eqref{nunk} and \eqref{nunk22},  the stationary measure of $(X)_{t\geq 0}$ satisfies
 \begin{equation} \label{pink}
 \nu_n(k) \sim \pi_n(k) \asymp \exp \left(n \Big[  G(k/n) -G(0)\Big]\right). 
 \end{equation}
 For each $k$ and $\ell$, we define the waiting time for $(X_t)$ going from $k$ to $\ell$ as
 	\begin{equation*}
 	\tau_{k\rightarrow \ell} =\inf \{t: X_t=\ell \mid X_0=k\}.
 	\end{equation*}
  	To prove Lemma \ref{lem:hit}, it suffices to show that 
 	\begin{equation} \label{taunnd}
 	\pp\left(\max\{\tau_{n\rightarrow [n(s^\star+\delta)]},  \tau_{0\rightarrow [n(s^\star-\delta)]}\} 
 	\geq Cn\right) =\kO(1/n),
 	\end{equation}
 for some positive constant $C>0$.  A  standard computation for the birth-death chain (see e.g.\ \cite[Proposition 2]{BBF}) gives that 
 	\begin{eqnarray*}
 	\E(\tau_{k\rightarrow k-1}) &=& \frac{1}{q_n(k)} \sum_{j=k}^n \frac{\pi_n(j)}{\pi_n(k)}, \\
 	\E(\tau^2_{k\rightarrow k-1}) &=& \frac{2}{q_n(k)\pi_n(k)} \sum_{j=k}^n \E(\tau_{j\rightarrow k-1}) \pi_n(j) - \E(\tau_{k\rightarrow k-1}). 
 	\end{eqnarray*}
By \eqref{qnk},  $ 	q_n(k) \asymp (k/n)$. Thus  using \eqref{gped} and \eqref{pink}, we obtain that 
 	\begin{eqnarray} 
	\label{spinjk}
 	\sum_{j=k}^n \frac{\pi_n(j)}{\pi_n(k)} \asymp \sum_{j=k}^n  \exp \left(n \Big[  G(j/n) -G(k/n)\Big]\right)  = \kO(1). 
 	\end{eqnarray}
 Therefore, 
 	\begin{equation*} 
 	\E(\tau_{k\rightarrow k-1}) =\kO(n/k), \quad \E(\tau^2_{k\rightarrow k-1}) =\kO(1).
 	\end{equation*}
Hence, 
	\begin{equation*}
	\max \{\E(\tau_{n\rightarrow [n(s^\star+\delta)]}), \Var(\tau_{n\rightarrow [n(s^\star+\delta)]})\} = \kO(n),
	\end{equation*}
which implies by using Chebyshev's inequality that 
	\begin{equation*}
	\pp\left(\tau_{n\rightarrow [n(s^\star+\delta)]}
	\geq Cn\right) =\kO(1/n), 
	\end{equation*}
for some $C>0$. A similar estimate holds for $\tau_{0\rightarrow [n(s^\star-\delta)]}$, and then we get \eqref{taunnd}.
\end{proof}

 \begin{lem}[It takes long to leave $\Omega_n^{2\delta}$]
 \label{lem:stay}
 For each $\gamma>0$, there exists a positive constant $C=C(\gamma)$, such that 
 	\begin{equation*}
 	\sup_{\sigma \in \Omega_n^{\delta}}\pp_{\sigma } \left(\inf\{t\colon \xi_t \not \in\Omega_n^{2\delta}\} \leq T^{+}_n(\gamma)\right) \leq \frac{C (\log n)^2}{n^2}.
 	\end{equation*}	
\end{lem}
\iflongversion
Again, we give a sketch of the proof below, the full proof can be found in Appendix \ref{sec-proj-chain-app}.

\else
Again, we give a sketch of the proof below, the full proof can be found in \cite[Appendix A.1]{CanGiaGibHof19b} of the extended version.
\fi
 	\begin{proof} 
 	By \eqref{rrpst}, it holds that
 		\[a:=-\inf \{R'(s): |s-s^\star| \leq 2\delta\} >0.\]
 		Therefore there exists a function $R_{2 \delta}: [0,1] \mapsto \R$, such that
 		\begin{equation} \label{r2dt}
 		R_{2\delta} (s) = R(s) \textrm{ for } |s-s^\star|\leq 2 \delta \textrm{ and } 
 		R'_{2\delta}(s) \leq -a \textrm{ for all } s \in [0,1]. 
  		\end{equation}	 
 		Define also
 		\begin{equation*}
 		T_1 =\inf\{t: |X_t-ns^\star| \geq 2 n\delta\}.
 		\end{equation*}
 		Observe that if $|X_0-ns^\star| \leq 2 n\delta$ then   the waiting time $T_1$ depends only on the  transition probabilities $\{(p_n(k),q_n(k),r_n(k)): |k-ns^\star| \leq 2 n\delta \}$. We define an auxiliary birth-death chain $(X'_t)$ with transition probabilities $\{(p'_n(k),q'_n(k),r'_n(k)): k \in \Z \}$ defined as 
 		\begin{equation}
		\label{aux-BD-chain}
 		(p'_n(k),q'_n(k),r'_n(k)) = 
		\begin{cases}
 		(p_n(k),q_n(k),r_n(k)) 
 		&{\rm if } \quad  |k-ns^\star| \leq 2 n \delta, \\
 		(R_{2 \delta} (k/n),0,1-R_{2 \delta} (k/n)) 
 		&{\rm otherwise.} 
 		\end{cases}
 		\end{equation}
 Then $R_{2\delta}$ is related to the drift of $(X_t')_{t\geq 0}$ by
 		\[
		\E[X_{t+1}'-X_t' \mid X_t'] 
 		=R_{2\delta} (X_t'/n)+\kO(\tfrac 1n).
		\]
Moreover,	
 		\begin{equation*}
 		\sup_{|x_0-ns^\star| \leq \delta n} \pp_{x_0} (T_1 \leq T^{+}_n(\gamma)) = \sup_{|x_0'-ns^\star| \leq \delta n} \pp_{x_0'} (T'_1 \leq T^{+}_n(\gamma)),
 		\end{equation*}
where 
 		\begin{equation*}
 		T'_1 =\inf\{t\colon |X'_t-ns^\star| \geq 2 n\delta\}.
 		\end{equation*}
By the strong Markov property,
 		\begin{eqnarray}
		\label{t12p}
 		&&\sup_{|x_0'-ns^\star| \leq \delta n} \pp_{x_0'} (T'_1 \leq T^{+}_n(\gamma)) \notag \\
 		&&\qquad \leq  \sup_{|x_0'-ns^\star| \leq \delta n} \pp_{x_0'} (T'_2 \leq T^{+}_n(\gamma)) \times \sup_{|x_1'-ns^\star| \leq 3 n\delta/2} \pp_{x_1'} (T'_1 \leq T^{+}_n(\gamma)),
 		\end{eqnarray}
where 
 		$$
		T_2'=\inf\{t\colon |X_t'-ns^\star|>3n\delta/2\}.
		$$
By using the same arguments as in \cite[Lemma 4.12]{DLP}, we can prove that 
 		\begin{equation} 
		\label{varxtp}
		\sup_{0\leq t \leq T^+_n(2\gamma)} |\E(X_t')-ns^\star| \leq  9 n\delta/8, \quad  \Var(X_t') = \kO(n).
	 	\end{equation}
The crucial point in the proof of \eqref{varxtp} is the contraction property of $\Var(X_t')$ that $\Var(X_{t+1}') \leq \Var(X_t') (1-\tfrac{\varepsilon}{n})$ for some $\varepsilon >0$. As for  \cite[Lemma 4.12]{DLP}, 
this contraction property can be shown when the drift function satisfies $f(x)=R_{2\delta}(x)-ax$ is decreasing for some $a>0$ (this holds by \eqref{r2dt}). Note that in \cite[Lemma 4.12]{DLP} the authors proved \eqref{varxtp} directly for $(X_t)_{t\geq 0}$ with  the function $f(x)=\textrm{tanh}(\beta x) - (1-\epsilon)x$, for some $\epsilon >0$. In our case, we do not know the behavior of the drift function outside the interval $(s^\star -2 \delta, s^\star + 2 \delta)$ well, so we need to consider the modification chain $(X_t')_{t\geq 0}$.
	
It follows from \eqref{varxtp} and Chebyshev's inequality that 
 		\begin{eqnarray*}
 		\sup_{0\leq t \leq T^+_n(2\gamma)}  \pp \left(|X_t'-ns^\star| \geq  5 n\delta/4\right) \leq \sup_{0\leq t \leq T^+_n(2\gamma)}  \pp \left(|X_t'-\E(X_t')| \geq   n\delta/8\right) = \kO(1/n).
 		\end{eqnarray*}
 As a consequence,
 		\begin{equation} \label{eon}
 		\E(N') \leq C_1 \log n,
 		\end{equation}
 where $C_1=C_1(a,\delta)$ is a large constant and 
  		\[
		N'=\#\{t \leq T_n^+(2\gamma): |X_t'-ns^\star|\geq 5n\delta/4\}.
		\] 
 Since $|X'_{t+1}-X_t'| \leq 1$ for all $t$,   if $|X_{t_0}'-ns^\star| \geq 3n\delta/2$ for some $t_0$ then $|X_{t}'-ns^\star| \geq 5n\delta/4$ for all $t_0 \leq t \leq t_0 + \delta n/4$. Hence,
 		\begin{equation} 
		\label{ceon}
 		\E[N' \mid T_2' \leq T_n^+(\gamma)] \geq \delta n/4.
 		\end{equation}
 Combining \eqref{eon} and \eqref{ceon}, we arrive at
 		\begin{equation*}
 		\sup_{|x_0'-ns^\star| \leq \delta n} \pp_{x_0'}(T_2' \leq T_n^+(\gamma)) \leq \frac{\E[N']}{\E[N \mid T_2' \leq T_n^+(\gamma)]} =\frac{\kO(\log n)}{n}.
 		\end{equation*}
Similarly, we also have 
 		\begin{equation*}
 		\sup_{|x_0'-ns^\star| \leq 3\delta n/2} \pp_{x_0'}(T_1' \leq T_n^+(\gamma)) =\frac{\kO(\log n)}{n}.
 		\end{equation*}
Now combining the last two estimates and \eqref{t12p}, we arrive at the claim in Lemma \ref{lem:stay}.
\end{proof}
 
\begin{lem}[A variance bound on $X_t$]
\label{lem:supvar}
 There exists a positive constant  $C$, such that for all $\gamma >1$  and $n \geq n_0=n_0(\gamma)$, with $n_0$ a large number, 
 	\begin{equation*}
 	\sup_{\sigma \in \Omega_n} \Var_{\sigma}(X_t) \leq Cn \qquad \mbox{\rm for all } \quad Cn \leq t\leq T_n^+(\gamma).
 	\end{equation*}	
\end{lem}

\begin{proof} 
Fix $C>0$ sufficiently large. Define 
 	\begin{equation*}
 	\kE=\{|X_t-ns^\star| \leq 2 \delta n \quad \mbox{\rm for all } \quad  Cn \leq t \leq T_n^+(\gamma)\}.
 	\end{equation*}
By Lemmas \ref{lem:hit} and \ref{lem:stay},
 	\begin{equation*}
 	\pp(\kE^c) =\kO(1/n).
 	\end{equation*} 	
Therefore,
 	\begin{equation} \label{varxte}
 	\Var(X_t) \leq  \Var (X_t\indicwo{\kE}) + \E (X_t^2\indicwo{\kE^c}) = \Var (X_t\indicwo{\kE}) + \kO(n).
 	\end{equation}
Suppose that $\kE$ happens. Then $(X_{t+Cn})_{t\geq 0}$ has the same law as the auxiliary chain $(X'_t)_{t\geq 0}$ defined in \eqref{aux-BD-chain} in the proof of Lemma \ref{lem:stay} for all $Cn \leq t \leq T_n^+(\gamma)$. Therefore,
 	\begin{equation} 
	\label{xtxtp}
 	\sup_{\sigma \in \Omega_n} \Var_{\sigma} (X_t \indicwo{\kE}) \leq \sup_{\sigma \in \Omega^{2\delta}_n} \Var_{\sigma} (X'_{t-Cn} \indicwo{\kE'}),
 	\end{equation}	
where
 	\[
	\kE'=\{|X'_s-ns^\star| \leq 2 \delta n \quad {\rm for \, all } \quad  s \leq T_n^+(\gamma-C)\}.
	\]	
By Lemma \ref{lem:stay}, 
 	\begin{equation*}
 	\pp(\kE'^c) =\kO((\log n)^2/ n^2).
 	\end{equation*}
Hence, by \eqref{varxtp} for all $0\leq s \leq T_n^+(\gamma-C)$,
 	\begin{eqnarray*}
 	\Var(X_s'\indicwo{\kE'}) &\leq& \Var(X_s') + \E(X_s'^2 \indicwo{\kE'^c}) + 2 \E(X_s') \E(X_s'\indicwo{\kE'^c}) \\
 	&\leq& \Var(X_s') + 3 n^2 \pp(\kE'^c) =\kO(n).
 	\end{eqnarray*}
Combining this with \eqref{varxte} and \eqref{xtxtp} we arrive at the claim in Lemma \ref{lem:supvar}.	
\end{proof}

\begin{lem}[Relating $X_{T_n^{\pm}(\gamma)}$ to $ns^\star$]
\label{lem:xt-ns} 
 	The following assertions hold:
 	\begin{itemize}
 		\item [(i)] For any $\alpha>0$, 
 		\begin{equation*}
 		\lim_{\gamma \rightarrow \infty} \limsup_{n\rightarrow \infty} \pp_{{\bf 1}} (X_{T_n^-(\gamma)} \leq ns^\star + \alpha \sqrt{n}) =0.
 		\end{equation*}
 		\item[(ii)] It holds that  
 		\begin{equation*} 
 		\lim_{\ell \rightarrow \infty}\lim_{\gamma \rightarrow \infty} \limsup_{n\rightarrow \infty} \sup_{\sigma \in \Omega_n} \pp_{\sigma}(|X_{T^+_n(\gamma)}-ns^\star| \geq \ell \sqrt{n}) = 0.
 		\end{equation*}
 	\end{itemize}
\end{lem}
 
\begin{proof} By Lemma \ref{lem:hit}, the chain $(X_t)_{t\geq 0}$ is in $\Omega_n^{\delta}$ after $\kO(n)$ steps with probability $1-\kO(1/n)$. Therefore, in the proof of Lemma \ref{lem:xt-ns}, we can assume that $X_0 \in \Omega_n^{\delta}$ by replacing $T_n^{\pm}(\gamma)$ by $T_n^{\pm}(\gamma-C)$. 
 
By Lemma \ref{lem:stay},
	\begin{equation} 
	\label{T1lTn}
	\pp(T_1\leq T_n^+(\gamma)) = \frac{\kO((\log n)^2)}{n^2},
	\end{equation}	
where 
	\begin{equation*}
	T_1 =\inf\{t\colon |X_t-ns^\star| \geq 2 n\delta\}.
	\end{equation*}	
Define 
 	$$
	Y_t = X_t -ns^\star \qquad \textrm{and} \qquad \tilde{Y}_t=Y_t \indic{t \leq T_1}.
	$$
Then,
 	\begin{eqnarray*} 
 	\E[\tilde{Y}_{t+1}-\tilde{Y}_t \mid X_t] &=& \E[Y_{t+1} \indic{t+1 \leq T_1}-Y_t \indic{t \leq T_1} \mid X_t]  \\
 	&=& \E[(Y_{t+1}-Y_t) \indic{t \leq T_1}-Y_{t+1} \indic{T_1=t+1} \mid X_t] \notag \\
 	&=& \indic{t \leq T_1}\E[X_{t+1}-X_t \mid X_t] -\E[Y_{t+1} \indic{T_1=t+1}\mid X_t]. \notag 
 	\end{eqnarray*}
Combining this with \eqref{ecxt} and the fact that $Y_{t+1} \leq n$, we get
 	\begin{eqnarray}
	\label{donaq}
 	\Big | \E[\tilde{Y}_{t+1}-\tilde{Y}_t \mid X_t] - R \left(X_t/n \right)\indic{t\leq T_1} \Big |\leq\frac{C}{n} +n\pp(T_1=t+1 \mid X_t). 
 	\end{eqnarray}
If $t\leq T_1$ then $X_t/n \in (s^\star - 2 \delta, s^\star + 2 \delta)$, and thus using $R(s^\star) =0$ we get
 	\begin{equation*}
 	\Big |R \left(X_t/n\right)\indic{t\leq T_1} -R'(s^\star) \tilde{Y}_t/n \Big| = \Big |R \left(s^\star + \tilde{Y}_t/n\right) -R'(s^\star) \tilde{Y}_t/n \Big| 
	\leq  C_1 \left(\tilde{Y}_t/n \right)^2, 
 	\end{equation*}
where $C_1=\sup_{s\in(s^\star-2\delta, s^\star +2 \delta)} |R''(s)|/2$. Combining this estimate with 
 \eqref{donaq} leads to 
 	\begin{eqnarray*}
 	\Big |\E[\tilde{Y}_{t+1}-\tilde{Y}_t \mid X_t] - R'(s^\star) \tilde{Y}_t/n \Big | \leq C_1\left(\tilde{Y}_t/n \right)^2 
 	+n\pp(T_1=t+1 \mid X_t) +\frac{C}{n},
 	\end{eqnarray*}
and thus by taking the expectation and using \eqref{T1lTn}, we get, for $t\leq T_n^+(\gamma)$,
 	\begin{eqnarray} \label{eyplr}
 	\Big | \E[\tilde{Y}_{t+1}] -    \left(1+ \frac{R'(s^\star)}{n} \right) \E[\tilde{Y}_t] \Big | \leq  \frac{C_1\E[(\tilde{Y}_t)^2]}{n^2}  + \frac{2C}{n}.
 	\end{eqnarray}
 	By Lemma \ref{lem:supvar}, $\Var(\tilde{Y}_t) =\kO(n)$. Thus it follows from \eqref{eyplr} that 
 	\begin{eqnarray*} 
 	\Big | \E[\tilde{Y}_{t+1}] -    \left(1+ \frac{R'(s^\star)}{n} \right) \E[\tilde{Y}_t] \Big | \leq  C_1\left( \frac{\E[\tilde{Y}_t]}{n} \right)^2 + \frac{3C}{n}.
 	\end{eqnarray*} 
Let us denote $b_{t}=\E(\tilde{Y}_t)$. Then $|b_{t+1}-b_t| \leq 2$  since $|\tilde{Y}_{t+1}-\tilde{Y}_t| \leq 2$. Moreover, the above estimate gives that, for all $t\geq 1$,
 	\begin{equation*} 
 	\left(1-\frac{a_0}{n}\right) b_t - a_1\left(\frac{b_t}{n} \right)^2 -\frac{a_2}{n} \leq b_{t+1} \leq \left(1-\frac{a_0}{n}\right) b_t + a_1\left(\frac{b_t}{n} \right)^2 +\frac{a_2}{n},
 	\end{equation*} 
with $|b_0| \leq \delta n$ and $a_0=-R''(s^\star)>0$ and $a_1=C_1=\sup_{s\in(s^\star-2\delta, s^\star +2 \delta)} |R''(s)|/2$. \\

A straightforward analysis using recursion (which is quite similar to the  proof of estimates for a process $(Z_t)_{t\geq 0}$ 
in \cite[Sections 4.3.2 and 4.4.1]{DLP};
\iflongversion
see Appendix \ref{sec-proj-chain-app} for the full proof)
\else
see \cite[Appendix A.1]{CanGiaGibHof19b} of the extended version for the full proof)
\fi
gives that if $b_0=\delta n$, then 
 	\begin{equation} 
	\label{eypt}
 	b_{T_n^-(\gamma)}\geq \sqrt{n} \e^{\gamma a_0/2},
 	\end{equation}
and
 	\begin{equation} 
	\label{btnp}
 	\sqrt{n}\e^{-2 \gamma a_0} \leq b_{T_n^+(\gamma)}\leq \sqrt{n}. 
 	\end{equation}
By Lemma \ref{lem:supvar},
 	\begin{equation} \label{varmp}
 	\Var(\tilde{Y}_{T_n^-(\gamma)}), \Var(\tilde{Y}_{T_n^+(\gamma)}) =\kO(n).
 	\end{equation}
Using Chebyshev's inequality, \eqref{eypt} -- \eqref{varmp}, 
 	\begin{eqnarray*}
 	\pp\left(\tilde{Y}_{T_n^-(\gamma)} \leq \alpha \sqrt{n} \right) \leq \pp\left(\tilde{Y}_{T_n^-(\gamma)} - b_{T_n^-(\gamma)} \leq (\alpha-\e^{\gamma a_0/2}  ) \sqrt{n} \right) 
	\leq \frac{\Var(\tilde{Y}_{T_n^-(\gamma)})}{n(\alpha-\e^{\gamma a_0/2})^2} =\kO(\e^{-\gamma a_0}),
 	\end{eqnarray*} 
and  
 	\begin{eqnarray*}
	\pp\left(\tilde{Y}_{T_n^+(\gamma)} \geq \ell \sqrt{n} \right) \leq \pp\left(\tilde{Y}_{T_n^+(\gamma)} - b_{T_n^+(\gamma)} \geq (\ell-1  ) \sqrt{n} \right) 
	\leq \frac{\Var(\tilde{Y}_{T_n^+(\gamma)})}{n(\ell-1)^2} =\kO(\ell^{-2}).
\end{eqnarray*} 
The first inequality implies Lemma \ref{lem:xt-ns}(i), while the second proves Lemma \ref{lem:xt-ns}(ii) by taking $\ell \rightarrow \infty$.   	 
\end{proof} 
 
\subsection{The left side of the critical window: Proof of \eqref{dntm}}
\label{sec-proof-dntm}

\begin{lem}[Concentration of the number of plus spins]
\label{munalp}
 	We have 
 	\begin{equation*}
 	\lim_{\alpha \rightarrow \infty} \lim_{n \rightarrow \infty} \mu_n\left(\sigma \colon \big | |\sigma_+|-ns^\star \big| > \alpha\sqrt{n}\right)=0. 
 	\end{equation*}
\end{lem}

\begin{proof}
For any $k$, let 
	$$
	\mu_n^{k}=\mu_n(\sigma \colon |\sigma_{+}|=k).
	$$
Then, by the definition of $\mu_n$,
	\begin{equation} 
	\label{mnkl}
	\frac{\mu_n^k}{\mu_n^{\ell}} = \frac{\binom{n}{k} 
	\exp(nF_n(k))}{\binom{n}{\ell} \exp(nF_n(\ell))} 
	\asymp \exp \left(n[G(k/n)-G(\ell/n)]\right).
	\end{equation} 
Let $\delta >0$ be any positive constant such that $(s^\star-2 \delta, s+2 \delta) \subset (0,1)$. Since $s^\star$ is the unique global maximizer of $G$, there exists $\varepsilon >0$, such that $G(k/n)-G(s^\star) \leq -\varepsilon$ for all $k$ satisfying $|k-ns^\star| \geq \delta n$. Therefore, by \eqref{mnkl},
	\begin{equation}
	\label{nsdel}
	\mu_n(\sigma \colon ||\sigma_{+}|-ns^\star| \geq \delta n) \lesssim n \exp(-\varepsilon n).
	\end{equation}
For the case $|k-ns^\star| \leq \delta n$, we use $G'(s^\star)=0$ and a Taylor expansion around $s^\star$ to obtain that 
	\begin{equation*}
	G(k/n)-G(s^\star) \leq -\frac{a}{2}\left(\frac{k}{n} -s^\star\right)^2,
	\end{equation*}
where 
	$$
	-a=\inf_{s\colon |s-s^\star| \leq \delta} G''(s)<0,
	$$
where the fact that $G''(s^\star)<0$ follows from condition (C3), which implies that $G''(s) <0$ in a neighborhood of $s^\star$.

Therefore, as $\alpha \rightarrow \infty$,
	\begin{equation*}
	\mu_n(\sigma \colon \alpha \sqrt{n} \leq ||\sigma_{+}|-ns^\star| \leq \delta n) \lesssim \sum_{\alpha \sqrt{n} \leq \ell \leq \delta n} \exp (-an \ell^2/2) \lesssim \int_{\alpha}^{\infty} \exp(-at^2/2)dt \rightarrow 0,
	\end{equation*}
since $a>0$. Combining this estimate with \eqref{nsdel} we get the desired result.
\end{proof}

We are now in the position to prove \eqref{dntm}:

\medskip
{\it Proof of \eqref{dntm}.} For any $\alpha >0$, let 
	\[
	A_{\alpha}=\{\sigma \colon |\sigma_{+}| \leq ns^\star + \alpha \sqrt{n} \}.
	\]

By definition of the total variation distance, and the fact that $\mu_n$ is the stationary distribution,
	\begin{eqnarray}
	d_n(T_n^-(\gamma)) &=& \sup_{A \subset \Omega_n} \sup_{\sigma \in \Omega_n} |\mu_n(A) - \pp_{\sigma} (\xi_{T_n^-(\gamma)} \in A)| \\
	&\geq &   \pp_{{\bf 1}} (\xi_{T_n^-(\gamma)} \in A_{\alpha} ) -  \mu_n(A_{\alpha}).\nonumber
	\end{eqnarray}
Therefore, for all $\alpha >0$,
	\begin{eqnarray}
	\lim_{\gamma \rightarrow \infty} \liminf_{n \rightarrow \infty} d_n(T_n^-(\gamma)) 
	&\geq &  	\lim_{\gamma \rightarrow \infty} \liminf_{n \rightarrow \infty}  \pp_{{\bf 1}} (\xi_{T_n^-(\gamma)} \in A_{\alpha} ) - \lim_{n\rightarrow \infty} \mu_n(A_{\alpha}) \\
	&\geq & \lim_{\gamma \rightarrow \infty} \liminf_{n \rightarrow \infty}  \pp_{{\bf 1}} (X_{T_n^-(\gamma)} \leq ns^\star + \alpha \sqrt{n}  ) - \lim_{n\rightarrow \infty} \mu_n(A_{\alpha}),\nonumber
	\end{eqnarray}
Hence, by taking $\alpha \rightarrow \infty$ and using Lemma \ref{munalp} as well as Lemma \ref{lem:xt-ns}(i), we obtain that 
	\begin{equation}
	\lim_{\gamma \rightarrow \infty} \liminf_{n \rightarrow \infty} d_n(T_n^-(\gamma)) 
	\geq \lim_{\alpha \rightarrow \infty} \lim_{\gamma \rightarrow \infty} \liminf_{n \rightarrow \infty}  \pp_{{\bf 1}} (X_{T_n^-(\gamma)} \leq ns^\star + \alpha \sqrt{n}  ) =1,
	\end{equation}
which proves \eqref{dntm}.
\qed

\subsection{The right side of the critical window: Proof of \eqref{dntp}}
\label{sec-proof-dntp}
Recall $\Omega_n^\delta$ from \eqref{Omega-delta}. 

Let $(\xi_t)_{t\geq 0}$ and $(\bar{\xi}_t)_{t\geq 0}$ be two Glauber dynamics, and let $(X_t)_{t\geq 0}$ and $(\bar{X}_t)_{t\geq 0}$ be the corresponding projection chains, i.e.,
	\begin{equation}
 	X_t=|\{i\colon \xi_t(i)=1\}|, \qquad \bar{X}_t=|\{i\colon \bar{\xi}_t(i)=1\}|.
	\end{equation}
For $\sigma_0 \in \Omega_n^\delta$ and $\sigma\in \Omega_n$, we define 
	\begin{eqnarray*}
 	&& U(\sigma) = |\{i\colon \sigma(i)=\sigma_0(i)=1\}|, \qquad V(\sigma) 
	= |\{i\colon \sigma(i)=\sigma_0(i)=-1\}|, \\
 	&& \Theta =\{\sigma \colon \min \{U(\sigma), V(\sigma), U(\sigma_0)-U(\sigma), V(\sigma_0)-V(\sigma)\} \geq \delta n/8\}. 
	\end{eqnarray*}
We also define 
 	\begin{eqnarray*}
 	&& D(t) = |U(\xi_t)-U(\bar{\xi}_t)|, \qquad H_1(t)=
 	\{\inf\{s\colon X_s =\bar{X}_s\} \leq t\}, \\
  	&& H_2(t_1, t_2) =\bigcap_{t=t_1}^{t_2} \{\xi_t \in \Theta, \bar{\xi}_t \in \Theta \}.
	 \end{eqnarray*}
We will crucially rely on the following result from \cite{DLP}:
 
\begin{prop}[General mixing time bound \protect{\cite[Theorem 5.1]{DLP}}]
\label{prop:lco}
There exists a positive constant $c$, such that for any possible couplings of $(\xi_t)_{t\geq 0}$ and $(\bar{\xi}_t)_{t\geq 0}$, for any $r_1<r_2$ and $\alpha >0$ and all large $n$,
 	\begin{eqnarray*}
 	\max_{\sigma_0 \in \Omega_n^\delta} \|\pp_{\sigma_0} (\xi_{r_2} \in \cdot) -\mu(\cdot)\|_{\sss \rm TV} 
	&\leq& \max_{\sigma_0, \sigma \in \Omega_n^\delta} \Big[\pp_{\sigma_0,\sigma}\left(D(r_1) > \alpha \sqrt{n}\right) 
	+ \pp_{\sigma_0, \sigma}(H_1(r_1)^c)  \\
 	&& \qquad + \pp_{\sigma_0, \sigma}(H_2(r_1,r_2)^c) + c \alpha \sqrt{\frac{n}{r_2-r_1}} \,\, \Big ].
 	\end{eqnarray*}		
\end{prop}

We next investigate the different terms appearing on the right hand side in Proposition \ref{prop:lco}, where we choose $r_1=r_1(\gamma)= T_n^+(\gamma)-Cn = T_n^+(\gamma-C) $, $r_2 =r_2(\gamma)= T_n^+(2\gamma) -Cn = T_n^+(2\gamma-C) $, with $C$ as in Lemma \ref{lem:hit} and $\alpha = \gamma^{1/4}$ for some $\gamma$ large enough:




\begin{lem}[Unlikely that projection chains remain uncoupled for a long time]
\label{lem:h1}
We have 
 	\begin{equation*}
 	\lim_{\gamma \rightarrow \infty} \limsup_{n\rightarrow \infty} \max_{\sigma,\sigma'} \pp_{\sigma,\sigma'}  (H_1(T_n^+(\gamma))^c) =0. 
 	\end{equation*}	
As a consequence, 
 	\begin{equation*}
 	\lim_{\gamma \rightarrow \infty} \limsup_{n\rightarrow \infty} \max_{\sigma,\sigma'} \pp_{\sigma,\sigma'}  (H_1(r_1(\gamma)))^c) =0. 
 	\end{equation*}		
\end{lem}

\begin{proof} By Lemma \ref{lem:xt-ns}(ii)
 	\begin{equation*} 
 	\lim_{\ell \rightarrow \infty}\lim_{\gamma \rightarrow \infty} \limsup_{n\rightarrow \infty} \max_{\sigma \in \Omega} \pp_{\sigma}(|X_{T^+_n(\gamma)}-ns^\star| \geq \ell \sqrt{n}) = 0.
 	\end{equation*}
Hence, it suffices to show that
 	\begin{equation} 
	\label{tmag}
 	\lim_{\ell \rightarrow \infty} \lim_{\gamma \rightarrow \infty} \limsup_{n\rightarrow \infty} \max_{\sigma, \sigma' \in \Omega^{\ell}} \pp_{\sigma,\sigma'}(\inf\{t:X_t=\bar{X}_t\} \geq \gamma n ) = 0,
 	\end{equation}
where 
 	\begin{equation*}
 	\Omega^{\ell}=\{\sigma \in \Omega_n\colon |\sigma_{+}-ns^\star| \leq \ell \sqrt{n}\}.
 	\end{equation*}
By monotonicity, it is sufficient to prove \eqref{tmag} for the cases that  $X_0=ns^\star+\ell \sqrt{n}$ and $\bar{X}_0=ns^\star-\ell \sqrt{n}$. 
 	
By \eqref{ecxt},
 	\begin{equation} 
	\label{cf}
	 \max \Big \{ 	|\E[X_{t+1}-X_t\mid X_t] -R(X_t/n)|, \, |\E[\bar{X}_{t+1}-\bar{X}_t\mid \bar{X}_t] -R(\bar{X}_t/n)| \Big \} \leq C/n,
 	\end{equation}
where $C$ is a universal constant. Moreover, if $ ns^\star -  \delta n \leq \bar{X}_t \leq X_t \leq ns^\star + \delta n $, then
 	\begin{eqnarray} 
	\label{mtism}
 	R(X_t/n)-R(\bar{X}_t/n)   
 	& \leq&  -\frac{1}{n}\left[\inf_{|s-s^\star|\leq \delta} |R'(s)| (X_t-\bar{X}_t) \right] 
 	\end{eqnarray}
since $R'(s) < 0$ for all $|s-s^\star|\leq \delta$. Define 
 	\begin{equation} 
	\label{def-K}
 	K=\frac{2C}{\inf_{|s-s^\star| \leq \delta}|R'(s)|} \in (0, \infty),
 	\end{equation}
and 
 	\begin{eqnarray*}
 		\tau_{M} &=&\inf \{t\colon |X_t-\bar{X}_t|\leq 2K \textrm{ or } |X_t-ns^\star|\geq \delta n  \textrm{ or } |\bar{X}_t-ns^\star|\geq \delta n \}, \\
 		M_t&=&X_t-\bar{X}_t, \qquad\tilde{M}_t= M_t  \indic{\tau_M\geq t}.
 	\end{eqnarray*}
By \eqref{mtism}, $\tilde{M}_t$ is a supermartingale. Indeed, let $\kF_t=\sigma (X_s, \bar{X}_s\colon s \leq t)$ and observe that if $\tilde{M}_t=0$ then $\tilde{M}_{t+1}=0$, so that
 	\begin{eqnarray*}
 	\E[\tilde{M}_{t+1}-\tilde{M}_t\mid \kF_t] & \leq& \E[(\tilde{M}_{t+1}-\tilde{M}_t ) \indic{\tilde{M}_{t}>0}\mid \kF_t]\\ 
 	& = &\E[X_{t+1}-X_t\mid \kF_t] \indic{\tilde{M}_{t}>0}  -\E[\bar{X}_{t+1}-\bar{X}_t\mid \kF_t] \indic{\tilde{M}_{t}>0}\notag  \\
 	&\leq & \big[R(X_t/n) -R(\bar{X}_t/n)\big] \indic{\tilde{M}_{t}>0} + 2C/n \\
 	&\leq&  -\frac{1}{n}\left[\inf_{|s-s^\star|\leq \delta} |R'(s)| (2K) -2C \right] \leq 0,
	\end{eqnarray*}
where in the third line, we have used \eqref{cf}, while in the last line, we have used \eqref{mtism} and \eqref{def-K} with a notice that if $\tilde{M}_t>0$, then the condition $ns^\star -  \delta n \leq \bar{X}_t \leq X_t \leq ns^\star + \delta n $ needed to apply \eqref{mtism} holds. 
 
We claim that, on the event $\{\tau_M>t\}$,  	
  	\begin{equation} \label{vmt}
 	\mathrm{Var} (\tilde{M}_{t+1}\mid \kF_t) \geq \Delta,
 	\end{equation} 
for some universal constant $\Delta>0$. Indeed, when $\tau_M >t$,
 	\[
	\pp(\tilde{M}_{t+1}-\tilde{M}_t =0\mid \kF_t), \, \pp(\tilde{M}_{t+1}-\tilde{M}_t =1\mid \kF_t) \geq c,
	\]
for some positive constant $c>0$. Thus $\textrm{Var} (\tilde{M}_{t+1}\mid \kF_t) \geq c$.  	Moreover, as we have shown above, $\tilde{M}_t$ is a supermartingale with increment bounded by $2$.   Therefore, using \cite[Lemma 3.5]{DLP2} we obtain that if $M_0=2\ell \sqrt{n}$ then
	\begin{equation*}
	\pp (\tau_M \geq \ell^3 n) = \kO(M_0/\sqrt{\ell^3 n}) = \kO(1/\sqrt{\ell}).
	\end{equation*}
Hence, by definition of $\tau_M$ and Lemma \ref{lem:stay},
 	\begin{eqnarray*}
 		\pp(\inf\{t\colon M_t\leq 2K\} \geq \ell^3 n) &\leq&
 		\pp(\tau_M \geq \ell^3 n) + \pp (\inf\{t\colon |X_t-ns^\star| \geq \delta n\} \leq \ell^3 n) \\
 		& & \quad + \pp (\inf\{t\colon |\bar{X}_t-ns^\star| \geq \delta n\} \leq \ell^3 n) \\
 		& =&  \pp(\tau_M \geq \ell^3 n)  + \kO(n^{-3/2}) = \kO(1/\sqrt{\ell}).
 	\end{eqnarray*}
where we let $\ell^3 = o(\log n)$ and apply Lemma \ref{lem:stay} in the last line. 
\medskip

Thanks to the above estimate,  now we can assume that $M_0=2K$. We observe that 
 	\begin{equation} \label{ldmt}
 	\pp(M_{t+1}-M_t \in \{-1,-2\}) =1-\pp(M_{t+1}-M_t \in \{0,1,2\}) \geq c,
 	\end{equation}
for some universal constant $c$. Let $\tau_0=\tilde{\tau}_0 =0$ and  define for $i\geq 1$,  
 	\begin{eqnarray*}
 	\tau_i &=&\inf\{t>\tilde{\tau}_{i-1}:M_t\geq 4K \textrm{ or } M_t \leq 0\}, \\
 	\tilde{\tau}_i&=& \inf\{t>\tau_i: M_t\leq 2K\}.
 	\end{eqnarray*}
We observe that $((\tau_{i+1}-\tilde{\tau}_i, \tilde{\tau}_i-\tau_i))_{i\geq 0}$ are independent random variables. Moreover, using \eqref{vmt} and the fact that $|M_{t+1}-M_t| \leq 2$, we can show that 
 	\begin{equation*}
 	\sup_{i\geq 0} \left(\E[(\tau_{i+1}-\tilde{\tau}_i)^2] + \E[(\tilde{\tau}_i-\tau_{i})^2] \right) \leq C,
 	\end{equation*}
for some universal constant $C$. Therefore, using Chebyshev's  inequality, 
 	\begin{equation} \label{tlcl}
 	\pp(\tau_{\ell} \geq 4C \ell) =\kO(1/\ell).
 	\end{equation}
By \eqref{ldmt}, 
 	\begin{equation*}
 	\pp(M_{\tau_i}  \leq 0 \mid \tilde{\tau}_{i-1}) \geq c^{2K}.
 	\end{equation*}
Hence,
 	\begin{equation*}
 	\pp( \exists i \leq \ell \colon M_{\tau_i}  \leq 0 ) \geq 1-(1- c^{2K})^{\ell} =1-\kO(1/\ell).
 	\end{equation*}
Combining this estimate with \eqref{tlcl} we obtain that 
 	\begin{equation*}
 	\pp_{2K}\left(\inf\{t:M_t \leq 0\}\geq 4C\ell\right) =\kO(1/\ell).
 	\end{equation*}
This completes the proof of \eqref{tmag}.
 \end{proof} 
 
\begin{lem}[$(\xi_t)_{t\geq 0}$ and $(\bar{\xi}_t)_{t\geq 0}$ are about equally far from $\sigma_0$]
\label{lem:dtng}
Consider two Glauber dynamics $(\xi_t)_{t\geq 0}$ and $(\bar{\xi}_t)_{t\geq 0}$ started at some $\sigma_0 \in \Omega_n^\delta$ and $\sigma \in \Omega_n$, respectively. There exists a positive constant $C$, such that for all $\alpha >0$,
 	\begin{equation*}
 	\lim_{\gamma \rightarrow \infty} \limsup_{n\rightarrow \infty} \pp_{\sigma_0,\sigma} (D(T_n^+(\gamma))\geq \alpha \sqrt{n}) \leq C/\alpha.
 	\end{equation*}	
Consequently,
 	\begin{equation*}
 	\lim_{\gamma \rightarrow \infty} \limsup_{n\rightarrow \infty} \pp_{\sigma_0,\sigma} (D(r_1(\gamma))\geq \alpha \sqrt{n}) \leq C/\alpha.
 	\end{equation*}		
\end{lem}

\begin{lem}[Both chains are in $\Theta$ for all large times]
\label{lem:h2g}
 	Consider two Glauber dynamics $(\xi_t)_{t\geq 0}$ and $(\bar{\xi}_t)_{t\geq 0}$ started at some $\sigma_0 \in \Omega_n^{\delta}$ and $\sigma \in \Omega_n$ respectively. Then for $\gamma_2 > \gamma_1$,
 	\begin{equation}
	\label{gen-Lem7.9}
 	\pp_{\sigma_0,\sigma} (H_2(T_n^+(\gamma_1),T_n^+(\gamma_2))^c)  \leq  \frac{C(\gamma_2-\gamma_1)+C_1}{n}, 
 	\end{equation}	
 where $C$ is a universal constant and $C_1$ is a constant depending on $\gamma_1$.  Consequently,
 	\begin{equation}
	\label{needed-Lem7.9}
 	\lim_{\gamma \rightarrow \infty}  \limsup_{n\rightarrow \infty} \pp_{\sigma_0,\sigma} (H_2(r_1(\gamma), r_2(\gamma) )^c) =0. 
 	\end{equation}	
\end{lem}

The proofs of Lemmas \ref{lem:dtng} and \ref{lem:h2g} are essentially the same as those of \cite[Lemmas 5.5 and 5.6]{DLP}, so we omit them here. 
\iflongversion
Their full proofs will be given in 
Appendix \ref{sec-proof-dntp-app}.
\else
Their full proofs will be given in \cite[Appendix A.2]{CanGiaGibHof19b} of the extended version.
\fi

\vspace{0.3 cm}
\noindent 
{\it Proof of \eqref{dntp}}. By Lemma \ref{lem:hit},
 	\begin{equation} \label{xtio}
 	\limsup_{n\rightarrow \infty} \max_{\sigma \in \Omega_n} 
 	\pp_\sigma(\xi_{Cn} \in \Omega_n^\delta) =1.
	\end{equation} 
Combining \eqref{xtio} and Lemmas \ref{lem:h1}, \ref{lem:dtng}, \ref{lem:h2g} and Proposition \ref{prop:lco} with $r_1= T_n^+(\gamma-Cn)$, $r_2 = T_n^+(2\gamma-Cn) $ and $\alpha = \gamma^{1/4}$, we arrive at
	\begin{equation*}
	\limsup_{n\to\infty}\max_{\sigma \in \Omega}  \|\pp_{\sigma} (\xi_{T_n^+(2\gamma)} \in \cdot) -\mu_n(\cdot)\|_{\sss \rm TV}  \leq C\gamma^{-1/4},
	\end{equation*}
for some $C>0$. This  proves \eqref{dntp} by taking $\gamma \rightarrow \infty$. \hfill $\square$
\iflongversion

\appendix
\section{Proof of Proposition \ref{prop:dntn}} 
\label{sec:dntn-app} 
In this appendix, we give the full proofs of some of the results used in Section \ref{sec:dntn}.
 
\subsection{A careful analysis of the projection chain: Proofs of 
Lemmas \ref{lem:hit}, \ref{lem:stay} and \ref{lem:xt-ns}} 
\label{sec-proj-chain-app} 
In this section, we give the full proofs of 
Lemmas \ref{lem:hit}, \ref{lem:stay} 
and \ref{lem:xt-ns}. Recall the notation at the start of Section \ref{sec-proj-chain}.

 \begin{proof}[Proof of Lemma \ref{lem:hit}]  Since $G$ is strictly increasing in $(0,s^\star)$ and strictly decreasing in $(s^\star, 1)$, there exists $\varepsilon$, such that
 	\begin{equation} 
	\label{gped-app}
 	G'(s) \geq \varepsilon \quad \textrm{if} \quad s\in (0,s^\star -\delta) \quad \textrm{and} \quad  G'(s) \leq -\varepsilon \quad \textrm{if} \quad s\in (s^\star +\delta, 1). 
 	\end{equation}	
 As in \eqref{nunk} and \eqref{nunk22}, the stationary measure of $(X)_{t\geq 0}$ is given by
 	\begin{equation}
 	\nu_n(k)=\frac{\pi_n(k)}{\pi_n(0) + \cdots + \pi_n(n)},
 	\end{equation}
 where
 	\begin{eqnarray} \label{pink-app}
 	\pi_n(k)  =\binom{n}{k} \e^{n[F_n(k)-F_n(0)]} \asymp \exp \left(n \Big[  G(k/n) -G(0)\Big]\right). 
	 \end{eqnarray}
For each $k$ and $\ell$, we define the waiting time for $(X_t)$ going from $k$ to $\ell$ as
 	\begin{equation*}
 	\tau_{k\rightarrow \ell} =\inf \{t\colon X_t=\ell \mid X_0=k\}.
 	\end{equation*}
To prove Lemma \ref{lem:hit}, it suffices to show that 
 	\begin{equation} 
	\label{taunnd-app}
 	\pp\left(\max\{\tau_{n\rightarrow n_{\delta}^+},  \tau_{0\rightarrow n_{\delta}^-}\} 
 	\geq Cn\right) \leq \frac{C}{n},
 	\end{equation}
where
 	\begin{equation*}
 	n_{\delta}^+=[n(s^\star + \delta)], \qquad n_{\delta}^-=[n(s^\star -\delta)].
 	\end{equation*}
Using Chebyshev's inequality, the estimate \eqref{taunnd-app} holds if there exists a constant $C=C(\delta)$ such that
 	\begin{equation} 
	\label{taunp-app}
 	\max \{\E(\tau_{n\rightarrow n_{\delta}^+}), \Var(\tau_{n\rightarrow n_{\delta}^+})\} \leq Cn,
 	\end{equation}
and 
 	\begin{equation} 
	\label{taunm-app}
 	\max \{\E(\tau_{0\rightarrow n_{\delta}^-}), \Var(\tau_{n\rightarrow n_{\delta}^-})\} \leq Cn.
 	\end{equation}
Indeed, by Chebyshev's inequality, for any variable $T$,
	\begin{equation}
	\pp(T > 2Cn) = \pp(T-\E[T] \geq 2Cn - \E[T]) \leq \frac{\Var (T)}{(2Cn-\E[T])^2} =\kO(1/n),
	\end{equation}
provided that $\E[T] \leq Cn$ and $\Var(T) \leq Cn$.

 We now prove \eqref{taunp-app}, the proof of \eqref{taunm-app} is essentially the same and is omitted. A standard calculus for the birth-death chain (see e.g.\ \cite[Proposition 2]{BBF}) gives that 
 	\begin{eqnarray}
 	\E(\tau_{k\rightarrow k-1}) &=& \frac{1}{q_n(k)} \sum_{j=k}^n \frac{\pi_n(j)}{\pi_n(k)},\nonumber\\
 	\E(\tau^2_{k\rightarrow k-1}) &=& \frac{2}{q_n(k)\pi_n(k)} \sum_{j=k}^n \E(\tau_{j\rightarrow k-1}) \pi_n(j) - \E(\tau_{k\rightarrow k-1}). \label{fetau2-app}
 	\end{eqnarray}
 We start computing $\E(\tau_{k\rightarrow k-1})$. By \eqref{qnk},
 	\begin{equation} 
	\label{asqnk-app}
 	q_n(k) \asymp \frac{k}{n}.
 	\end{equation}
By using \eqref{gped-app} and \eqref{pink-app}, we obtain that 
 	\begin{eqnarray} 
	\label{spinjk-app}
 	\sum_{j=k}^n \frac{\pi_n(j)}{\pi_n(k)} \asymp \sum_{j=k}^n  \exp \left(n \Big[  G(j/n) -G(k/n)\Big]\right) &\leq& \sum_{j=k+1}^n  \exp \left(-\varepsilon(j-k)\right) \notag \\
 	&\leq&\frac{\e^{-\varepsilon}}{1-\e^{-\varepsilon}}.
 	\end{eqnarray}
 Therefore, 
 	\begin{equation} 
	\label{etaunnp-app}
 	\E(\tau_{k\rightarrow k-1}) =\kO(n/k).  
 	\end{equation}
 Hence
 	\begin{equation} 
	\label{etaund-app}
 	\E(\tau_{n\rightarrow n_{\delta}^+})= \sum_{k=n_{\delta}^++1}^n \E(\tau_{k\rightarrow k-1})  =\kO(n).
 	\end{equation}
 We now compute $\E(\tau_{k\rightarrow k-1}^2)$. Using \eqref{fetau2-app} and \eqref{asqnk-app},
 	\begin{equation}
 	\E(\tau^2_{k\rightarrow k-1}) \lesssim \frac{n}{k} \sum_{j=k}^n \E(\tau_{j\rightarrow k-1}) \frac{\pi_n(j)}{\pi_n(k)}.
 	\end{equation}
Here and in the following, we write $f \lesssim g$ when $f\le c g$ for some $c>0$. Similarly to  \eqref{etaunnp-app} and \eqref{etaund-app}, if $n_{\delta}^+ \leq k \leq j-1$ then
 	\begin{eqnarray*}
 	\E(\tau_{j\rightarrow k-1 })= \kO(j-k).
 	\end{eqnarray*}
 As in \eqref{spinjk-app}, if  $n_{\delta}^+ \leq k \leq j-1$ then
 	\begin{equation}
 	\frac{\pi_n(j)}{\pi_n(k)} \lesssim \exp \left(-\varepsilon (j-k)\right).
 	\end{equation}
Therefore,
 	\begin{eqnarray}
 	\E(\tau^2_{k\rightarrow k-1})  &\lesssim & \sum_{j=k}^n (j-k) \exp \left(-\varepsilon(j-k)\right) =\kO(1).
 	\end{eqnarray}
In conclusion,
 	\begin{eqnarray}
 	\Var(\tau_{n\rightarrow n_{\delta}^+}) =\sum_{k=n_{\delta}^++1}^n \Var(\tau_{k\rightarrow k-1}) \leq \sum_{k=n_{\delta}^++1}^n \E(\tau^2_{k\rightarrow k-1}) =\kO(n). 
 	\end{eqnarray}
 Combining this with \eqref{etaund-app} we get \eqref{taunp-app}.
 \end{proof}

 \begin{proof}[Proof of Lemma \ref{lem:stay}]
 By \eqref{defR}, it holds that
 		\[
		a:=-\inf \{R'(s)\colon |s-s^\star| \leq 2\delta\} >0.
		\]
 Therefore there exists a function $R_{2 \delta}\colon [0,1] \mapsto \R$, such that
 		\begin{equation} \label{r2dt-app}
 		R_{2\delta} (s) = R(s) \textrm{ for } |s-s^\star|\leq 2 \delta \textrm{ and } 
 		R'_{2\delta}(s) \leq -a 
 	 		\textrm{ for all } s \in [0,1]. 
 		\end{equation}	 
Define also
 		\begin{equation}
 		T_1 =\inf\{t\colon |X_t-ns^\star| \geq 2 n\delta\}.
 		\end{equation}
Observe that if $|X_0-ns^\star| \leq 2 n\delta$ then   the waiting time $T_1$ depends only on the  transition probabilities $\{(p_n(k),q_n(k),r_n(k)): |k-ns^\star| \leq 2 n\delta \}$. We define an auxiliary birth-death chain $(X'_t)$ with transition probabilities $\{(p'_n(k),q'_n(k),r'_n(k)): k \in \Z \}$ defined as 
 		\begin{equation}
 		(p'_n(k),q'_n(k),r'_n(k)) = \begin{cases}
 		(p_n(k),q_n(k),r_n(k)) 
 		& {\rm if } \quad  |k-ns^\star| \leq 2 n \delta, \\
 		(R_{2 \delta} (k/n),0,1-R_{2 \delta} (k/n)) 
  		&{\rm otherwise.} 
 		\end{cases}
 		\end{equation}
Then $R_{2\delta}$ is related to the drift of $(X_t')_{t\geq 0}$ as
 		\[
		\E[X_{t+1}'-X_t' \mid X_t'] 
 		=R_{2\delta} (\tfrac{X_t'}{n})+\kO(\tfrac 1n).
		\]
Moreover,	
 		\begin{equation}
 		\sup_{|x_0-ns^\star| \leq \delta n} \pp_{x_0} (T_1 \leq T^{+}_n(\gamma)) = \sup_{|x_0'-ns^\star| \leq \delta n} \pp_{x_0'} (T'_1 \leq T^{+}_n(\gamma)),
 		\end{equation}
where 
 		\begin{equation*}
 		T'_1 =\inf\{t\colon |X'_t-ns^\star| \geq 2 n\delta\}.
 		\end{equation*}
 By the strong Markov property, 
  		\begin{eqnarray} \label{t12p-app}
 		&&\sup_{|x_0'-ns^\star| \leq \delta n} \pp_{x_0'} (T'_1 \leq T^{+}_n(\gamma)) \notag \\
 		&&\qquad\leq \sup_{|x_0'-ns^\star| \leq \delta n} \pp_{x_0'} (T'_2 \leq T^{+}_n(\gamma)) \times \sup_{|x_1'-ns^\star| \leq 3 n\delta/2} \pp_{x_1'} (T'_1 \leq T^{+}_n(\gamma)),
 		\end{eqnarray}
 where 
 		$$
		T_2'=\inf\{t\colon |X_t'-ns^\star|>3n\delta/2\}.
		$$
By \eqref{ecxt}, there exists a positive constant $C_1$, such that for all $1 \leq k\leq n$,
 		\begin{equation}
		\label{dnwos-app}
 		\Big | (p_n'(k)-q_n'(k))-R_{2\delta}(k/n)\Big | \leq C_1/n.
 		\end{equation}
We have
 		\begin{eqnarray*}
 		&&\E[(X'_{t+1}-ns^\star)^2 \mid X'_t ]  \\
 		&&\qquad =(X'_{t}-ns^\star +1)^2p_n'(X'_t) + (X'_{t}-ns^\star -1)^2q_n'(X'_t) + (X'_{t}-ns^\star )^2r_n'(X'_t) \\
 		&&\qquad = (X'_{t}-ns^\star)^2 + 2 (X'_{t}-ns^\star)[p_n'(X'_t)-q_n'(X'_t)] + [p_n'(X'_t)+q_n'(X'_t)].
 		\end{eqnarray*}
Define 
 		\[
		Y'_t=X_t'-ns^\star.
		\]
Then,
 		\begin{eqnarray*}
 		\E[(Y'_{t+1})^2\mid Y'_t] \leq  (Y_t')^2 +2Y_t' R_{2\delta} (s^\star + \tfrac{Y'_t}{n})+ C_1.
 		\end{eqnarray*}
Combining this with the fact that  $xR_{2 \delta}(s^\star +x) \leq 0$ yields that
 		\begin{equation*}
 		\E(Y_{t+1}'^2) \leq \E(Y_t'^2) + C_1,
 		\end{equation*}
and thus, for all $T$,
 		\begin{equation*}
 		\sup_{0\leq t \leq T}\E(Y_t'^2) \leq \E(Y_0'^2) + TC_1.
 		\end{equation*}
As a consequence, by the Cauchy-Schwarz inequality,
 		\begin{equation} \label{ex0p-app}
 		\sup_{0\leq t \leq T} |\E(X_t')-ns^\star| \leq  \sqrt{\E(|X_0'-ns^\star|^2) + TC_1}.
 		\end{equation}
We now estimate the variance of $X'_t$. We have by \eqref{dnwos-app},  
 		\begin{eqnarray*}
 		\E[(X'_{t+1})^2 \mid X'_t ]  &=& (X'_{t})^2  + 2X'_{t}[p_n'(X'_t)-q_n'(X'_t)] + [p_n'(X'_t)+q_n'(X'_t)] \notag \\
 		&\leq & (X'_{t})^2 \left(1-\tfrac{a}{n}\right)+ 2X'_{t} \left[R_{2 \delta}(X'_t/n) +a X'_t/n\right] + C, 
 		\end{eqnarray*}
and thus 
 		\begin{eqnarray}
 		\E[(X'_{t+1})^2] \leq \left(1-\tfrac{a}{n}\right) \E[(X'_{t})^2 ]+ 2 \E \left(X'_{t} \left[R(X'_t/n) +a X'_t/n\right] \right) + C.
 		\end{eqnarray}
Similarly,
 		\begin{eqnarray}
 		\E[X_{t+1}']^2&=& \Big[\E\big(\E[X'_{t+1} \mid X'_t ]\big)\Big]^2  =\Big[\E (X'_{t})+ \E[p_n'(X'_t)-q_n'(X'_t)]\Big]^2  \notag \\
 		&=& \Big(1-\tfrac{a}{n}\Big) \E(X'_t)^2 + 2 \E(X'_t) \E  \Big[ p_n(X_t')-q_n(X_t')+a X'_t/n \Big] + \Big(\E[p_n(X_t')-q_n(X_t')]\Big)^2  \notag \\
 		&\geq& \Big(1-\tfrac{a}{n}\Big) \E(X'_t)^2 + 2 \E(X'_{t}) \E  \Big[R_{2\delta}(X'_t/n) +a X'_t/n\Big] -C.
 		\end{eqnarray}
Therefore,
 		\begin{eqnarray}
 		\Var(X_{t+1}') &\leq& \left(1-\tfrac{a}{n}\right) \Var(X'_t) + 2 \E \left(X'_{t} \left[R(X'_t/n) +a X'_t/n\right] \right) \notag \\
 		&& \qquad -2 \E(X'_{t}) \E  \left[R_{2\delta}(X'_t/n) +a X'_t/n\right] +2C \notag \\
 		&\leq& \left(1-\tfrac{a}{n}\right) \Var(X'_t) + 2C.
 		\end{eqnarray}
Notice that here we have used the FKG inequality and the fact that $R_{2\delta}(x)-ax$ is decreasing by \eqref{r2dt} in the second line. Hence, by induction we can show that, for all $t$,
 		\begin{equation} 
		\label{varxtp-app}
 		\Var(X_t') \leq \frac{2C n }{a}.
 		\end{equation}
By \eqref{ex0p-app}, if $|x_0'-ns^\star| \leq n\delta$ then for all $n\geq n_0(\gamma)$,
 		\begin{equation*}
 		\sup_{0\leq t \leq T^+_n(2\gamma)} |\E(X_t')-ns^\star| \leq  9 n\delta/8.
 		\end{equation*}
Hence it follows from Chebyshev's inequality that 
 		\begin{eqnarray*}
 		\sup_{0\leq t \leq T^+_n(2\gamma)}  \pp \left(|X_t-ns^\star| \geq  5 n\delta/4\right) \leq \sup_{0\leq t \leq T^+_n(2\gamma)}  \pp \left(|X_t'-\E(X_t')| \geq   n\delta/8\right) \leq \frac{128 C  }{n a \delta^2}.
 		\end{eqnarray*}
As a consequence,
 		\begin{equation} 
		\label{eon-app}
 		\E(N) \leq C_1 \log n,
 		\end{equation}
where $C_1=C_1(a,\delta)$ is a large constant and 
 		\[
		N=\#\{t \leq T_n^+(2\gamma): |X_t'-ns^\star|\geq 5n\delta/4\}.
		\] 
Since $|X'_{t+1}-X_t'| \leq 1$ for all $t$,   if $|X_{t_0}'-ns^\star| \geq 3n\delta/2$ for some $t_0$ then $|X_{t}'-ns^\star| \geq 5n\delta/4$ for all $t_0 \leq t \leq t_0 + \delta n/4$. Hence 
 		\begin{equation} 
		\label{ceon-app}
 		\E[N \mid T_2' \leq T_n^+(\gamma)] \geq \delta n/4.
 		\end{equation}
Combining \eqref{eon-app} and \eqref{ceon-app}, we arrive at
 		\begin{equation*}
 		\sup_{|x_0'-ns^\star| \leq \delta n} \pp_{x_0'}(T_2' \leq T_n^+(\gamma)) \leq \frac{\E[N]}{\E[N \mid T_2' \leq T_n^+(\gamma)]} =\frac{\kO(\log n)}{n}.
 		\end{equation*}
 		Similarly, we also have 
 		\begin{equation*}
 		\sup_{|x_0'-ns^\star| \leq 3\delta n/2} \pp_{x_0'}(T_1' \leq T_n^+(\gamma)) =\frac{\kO(\log n)}{n}.
 		\end{equation*}
 		Now combining the last two estimates and \eqref{t12p}, we get the desired result.
 	\end{proof}

\begin{proof}[Proof of Lemma \ref{lem:xt-ns}] By Lemma \ref{lem:hit}, the chain $(X_t)_{t\geq 0}$ is in $\Omega_n^{\delta}$ after $\kO(n)$ steps with probability $1-\kO(1/n)$. Therefore, in the proof of Lemma \ref{lem:xt-ns}, we can assume that $X_0 \in \Omega_n^{\delta}$. 
 
By Lemma \ref{lem:stay},
	\begin{equation} \label{T1lTn-app}
	\pp(T_1\leq T_n^+(\gamma)) = \frac{\kO((\log n)^2)}{n^2},
	\end{equation}	
where 
	\begin{equation*}
	T_1 =\inf\{t\colon |X_t-ns^\star| \geq 2 n\delta\}.
	\end{equation*}	
Define 
 	$$
	Y_t = X_t -ns^\star \qquad \textrm{and} \qquad \tilde{Y}_t=Y_t \indic{t \leq T_1}.
	$$
Then,
 	\begin{eqnarray} 
 	\E[\tilde{Y}_{t+1}-\tilde{Y}_t \mid X_t] &=& \E[Y_{t+1} \indic{t+1 \leq T_1}-Y_t \indic{t \leq T_1}\mid X_t]  \\
 	&=& \E[(Y_{t+1}-Y_t) \indic{t \leq T_1}-Y_{t+1} \indic{T_1=t+1}\mid X_t] \notag \\
 	&=& \indic{t \leq T_1}\E[X_{t+1}-X_t \mid X_t] -\E[Y_{t+1} \indic{T_1=t+1} \mid X_t]. \notag 
 	\end{eqnarray}
Combining this with \eqref{ecxt} and the fact that $Y_{t+1} \leq n$, we get
 	\begin{eqnarray}
	\label{donaq-app}
 	\Big | \E[\tilde{Y}_{t+1}-\tilde{Y}_t \mid X_t] - R \left(X_t/n \right)\indic{t\leq T_1}\Big |\leq\frac{C}{n} +n\pp(T_1=t+1 \mid X_t). 
 	\end{eqnarray}
If $t\leq T_1$ then $X_t/n \in (s^\star - 2 \delta, s^\star + 2 \delta)$, and thus using $R(s^\star) =0$ we get
 	\begin{equation}
 	\Big |R \left(X_t/n\right)\indic{t\leq T_1} -R'(s^\star) \tilde{Y}_t/n \Big| = \Big |R \left(s^\star + \tilde{Y}_t/n\right) -R'(s^\star) \tilde{Y}_t/n \Big| 
	\leq  C_1 \left(\tilde{Y}_t/n \right)^2,
 	\end{equation}
where $C_1=\sup_{s\in(s^\star-2\delta, s^\star +2 \delta)} |R''(s)|/2$. Combining this estimate with \eqref{donaq-app}, we obtain 
 	\begin{eqnarray}
 	\Big |\E[\tilde{Y}_{t+1}-\tilde{Y}_t \mid X_t] - R'(s^\star) \tilde{Y}_t/n \Big | \leq C_1\left(\tilde{Y}_t/n \right)^2 
 	+n\pp(T_1=t+1 \mid X_t) +\frac{C}{n},
 	\end{eqnarray}
and thus by taking the expectation and using \eqref{T1lTn-app}, we get for $t\leq T_n^+(\gamma)$
 	\begin{eqnarray} 
	\label{eyplr-app}
 	\Big | \E[\tilde{Y}_{t+1}] -    \left(1+ \frac{R'(s^\star)}{n} \right) \E[\tilde{Y}_t] \Big | \leq  \frac{C_1\E[(\tilde{Y}_t)^2]}{n^2}  + \frac{2C}{n}.
 	\end{eqnarray}
By Lemma \ref{lem:supvar}, $\Var(Y_t') =\kO(n)$. Thus it follows from \eqref{eyplr-app} that 
 	\begin{eqnarray*} 
 	\Big | \E[\tilde{Y}_{t+1}] -    \left(1+ \frac{R'(s^\star)}{n} \right) \E[\tilde{Y}_t] \Big | \leq  C_1\left( \frac{\E[\tilde{Y}_t]}{n} \right)^2 + \frac{3C}{n}.
 	\end{eqnarray*} 
Let us denote $b_{t}=\E(\tilde{Y}_t)$. Then $|b_{t+1}-b_t| \leq 2$  since $|\tilde{Y}_{t+1}-\tilde{Y}_t| \leq 2$. Moreover, the above estimate gives that for all $t\geq 1$
 	\begin{equation} 
	\label{reofb-app}
 	\left(1-\frac{a_0}{n}\right) b_t - a_1\left(\frac{b_t}{n} \right)^2 -\frac{a_2}{n} \leq b_{t+1} \leq \left(1-\frac{a_0}{n}\right) b_t + a_1\left(\frac{b_t}{n} \right)^2 +\frac{a_2}{n},
 	\end{equation} 
with $|b_0| \leq \delta n$ and $a_0=-R''(s^\star)>0$ and $a_1=C_1=\sup_{s\in(s^\star-2\delta, s^\star +2 \delta)} |R''(s)|/2$. \\

We claim  that if $b_0=\delta n$ then 
 	\begin{equation} 
	\label{eypt-app}
 	b_{T_n^-(\gamma)}\geq \sqrt{n} \e^{\gamma a_0/2},
 	\end{equation}
and
 	\begin{equation} 
	\label{btnp-app}
 	\sqrt{n}e^{-2 \gamma a_0} \leq b_{T_n^+(\gamma)}\leq \sqrt{n}. 
 	\end{equation}
We first assume \eqref{eypt-app}, \eqref{btnp-app} and prove the result. By Lemma \ref{lem:supvar},
 	\begin{equation} 
	\label{varmp-app}
 	\Var(\tilde{Y}_{T_n^-(\gamma)}), \Var(\tilde{Y}_{T_n^+(\gamma)}) \leq Cn,
 	\end{equation}
 where $C$ is a universal constant. Using Chebyshev's inequality, \eqref{eypt-app} and \eqref{varmp-app}, 
 	\begin{eqnarray*}
 	\pp\left(\tilde{Y}_{T_n^-(\gamma)} \leq \alpha \sqrt{n} \right) \leq \pp\left(\tilde{Y}_{T_n^-(\gamma)} - b_{T_n^-(\gamma)} \leq (\alpha-\e^{\gamma a_0/2}  ) 
	\sqrt{n} \right) \leq \frac{\Var(\tilde{Y}_{T_n^-(\gamma)})}{n(\alpha-\e^{\gamma a_0/2})^2} =\kO(\e^{-\gamma a_0}).
 	\end{eqnarray*} 
 This implies (i). The proof of (ii) is similar. It follows from  Chebyshev's inequality, \eqref{btnp-app} and \eqref{varmp-app} that
 	\begin{eqnarray*}
 	\pp\left(\tilde{Y}_{T_n^+(\gamma)} \geq \ell \sqrt{n} \right) \leq \pp\left(\tilde{Y}_{T_n^+(\gamma)} - b_{T_n^+(\gamma)} \geq (\ell-1  ) \sqrt{n} \right) 
	\leq \frac{\Var(\tilde{Y}_{T_n^+(\gamma)})}{n(\ell-1)^2} =\kO(\ell^{-2}).
 	\end{eqnarray*} 
Taking $\ell \rightarrow \infty$ we get (ii). \\
 	
Now we prove \eqref{eypt-app} and \eqref{btnp-app}. Let us define 
 	\[
	S_k=\inf \{t\colon b_t \leq n \e^{-k}\}. 
	\]
Then $S_k=0$ for $k=0, \ldots, k_0$ with $k_0=[-\log \delta]$.  Hence, it holds that for all $T$ and $k$,
 	\begin{equation} 
	\label{btk-app}
 	b_T \geq k \quad \mbox{\rm  \, \, if } \quad 2+T \leq S_{[\log (n/k)]}.
 	\end{equation}
Let us define 
 	\[
	K_n(\gamma) = \tfrac{1}{2} \log n + 3 \gamma.
	\]
If $t\leq S_{K_n(\gamma)}$ then $b_t \geq \sqrt{n} \e^{- 3\gamma}$. Thus, for $\gamma \geq \gamma_0 =\gamma_0(a_1,a_2)$ we have 
 	\begin{equation*}
 	\e^{8\gamma} \left(\frac{b_t}{n}\right)^2 \geq a_1\left(\frac{b_t}{n} \right)^2 + \frac{a_2}{n}.
 	\end{equation*}  
Hence, by \eqref{reofb-app} for all  $t\leq S_{K_n(\gamma)}$ 
 	\begin{eqnarray*}
 	b_{t}\left(1-\frac{a_0}{n} + \e^{8\gamma} \frac{b_t}{n^2} \right) \geq b_{t+1} \geq b_{t}\left(1-\frac{a_0}{n} - \e^{8\gamma} \frac{b_t}{n^2} \right).
 	\end{eqnarray*}
In addition if $S_k\leq t<S_{k+1}$ then $n\e^{-(k+1)}\leq b_t \leq n\e^{-k}$, and thus 
 	\begin{equation*}
 	b_{t}\left(1-\frac{a_0 - \e^{8\gamma-k}} {n} \right)\geq b_{t+1} \geq b_{t}\left(1-\frac{a_0 + \e^{8\gamma-k}} {n} \right).
 	\end{equation*}
We have  $n\e^{-k} -2 \leq b_{S_k} \leq n\e^{-k}$, since $b_t$ decrease at most by $-2$ each time.  Therefore,
 	\begin{equation*}
 	\frac{n\e^{-(k+1)}}{n\e^{-k}-2} \geq \frac{b_{S_{k+1}}}{b_{S_k}} \geq \left(1-\frac{a_0 + \e^{8\gamma-k}} {n} \right)^{S_{k+1}-S_k},
 	\end{equation*}
and 
 	\begin{equation*}
 	\frac{n\e^{-(k+1)}-2}{n\e^{-k}} \leq \frac{b_{S_{k+1}}}{b_{S_k}} \leq \left(1-\frac{a_0 - \e^{8\gamma-k}} {n} \right)^{S_{k+1}-S_k}.
 	\end{equation*}
From the two above estimates, we can show that
 	\begin{eqnarray}
 	S_{k+1}-S_k = \frac{n}{a_0 + \e^{8\gamma-k}} \pm  \kO(\e^k).
 	\end{eqnarray}
Thus for all $k_0 \leq K \leq K_n(\gamma)$
 	\begin{eqnarray*}
 	S_K =\sum_{k= k_0}^{K-1} \big(S_{k+1}-S_k \big)&\ge& n \sum_{k= k_0}^{K-1} \frac{1}{a_0 + \e^{8\gamma-k}} \pm \kO(\e^{K+1}) \\
 	&= & \frac{n(K-k_0)}{a_0} - n \sum_{k= k_0}^{K-1} \frac{\e^{8\gamma-k}}{a_0(a_0 + \e^{8\gamma-k})} \pm \kO(\sqrt{n}e^{3\gamma})\\
 	&= & \frac{nK}{a_0} \pm   \kO(n \log \gamma). 
 	\end{eqnarray*}
Therefore,
 	\begin{equation} 
	\label{stnm-app}
 	S_{[\tfrac{1}{2} \log n -\tfrac{\gamma a_0}{2}]} \geq \frac{1}{2a_0} n \log n -\frac{\gamma n}{2} 
 	-  \kO(n \log \gamma) \geq T_n^-(\gamma) + \frac{\gamma n}{4}.
 	\end{equation}
This  estimate together with \eqref{btk-app} implies that $b_{T_n^-(\gamma)} \geq \sqrt{n} \e^{\gamma a_0/2}$, which proves \eqref{eypt-app}. Similarly,
 	\begin{equation} 
	\label{stnp-app}
 	S_{[\tfrac{1}{2} \log n + 2 \gamma a_0]} \geq \frac{1}{2a_0} n \log n +2 \gamma n 
 	-  \kO(n \log \gamma) \geq T_n^+(\gamma) + \frac{\gamma n}{2},
 	\end{equation} 
and thus the lower bound in \eqref{btnp-app} that $b_{T_n^+(\gamma)} \geq \sqrt{n} \e^{-2\gamma a_0}$ holds. We now finish the proof by showing the upper bound that $b_{T_n^+(\gamma)} \leq \sqrt{n}$.  By \eqref{stnp-app},  $b_t \geq n^{1/4}$ for all $t\leq T_n^+(\gamma)$. Moreover, by \eqref{reofb-app}
 	\begin{equation*}
 	b_{t+1} -b_t  \leq -a_0 \frac{b_t}{n} + a_1 \left(\frac{b_t}{n}\right)^2 + \frac{a_2}{n} \leq 0,
 	\end{equation*}
if $n^{1/4} \leq b_t \leq \delta n$ for all $\delta \leq \delta_0$ with $\delta_0$ some universal constant. Hence $b_t$ is decreasing in $t$ for $t\leq T_n^+(\gamma)$. Thus $b_t \leq k$ holds if $t\geq S_{[\log (n/k)]}$. In particular, since 
 	\begin{equation*}
 	S_{[\log n/2]} = \frac{1}{2a_0} n \log n \pm \kO(n \log \gamma) \leq T_n^+(\gamma) -\frac{\gamma n}{2},
 	\end{equation*}
we have $b_{T_n^+(\gamma)} \leq \sqrt{n}$. 
\end{proof}

\subsection{Proof of Lemmas \ref{lem:dtng} and \ref{lem:h2g} used in the proof of \eqref{dntp}}
\label{sec-proof-dntp-app}
In this section, we give the full proof of Lemmas \ref{lem:dtng} and \ref{lem:h2g}. Before giving the proof of Lemma \ref{lem:dtng}, we define some notation and state an auxiliary result.

For any $A\subseteq [n]$, let us define 
	\begin{equation}
	X_t(A) = \sum_{i \in A} \indic{\xi_t(i)=1}.
	\end{equation}

\begin{lem} \label{xtasa}
 	There exists a positive constant $C$ such that for all   $\gamma$ and sufficiently large  $n$
 	\begin{equation*}
 	\sup_{ T_n^+(1) \leq t \leq T_n^+(\gamma)}	\sup_{\sigma_0 \in \Omega_n^\delta} 
 	\sup_{A \subset [n]} 
 	\E_{\sigma_0} \left(X_t(A)-s^\star|A|\right)^2 \leq Cn.
 	\end{equation*}
 \end{lem}
 
 \begin{proof} We construct a monotone coupling of $(\xi_t)_{t\geq 0}$ and $(\xi_t^+)_{t\geq 0}$ starting from $\sigma_0$ and the all plus configuration respectively, such that $\xi_t \leq \xi_t^+$ for all $t$. Then
 	\begin{equation*}
 	X_t(A)-s^\star|A| \leq X_t^+(A)-s^\star |A| \leq |X_t^+(A)-s^\star |A||.
 	\end{equation*}	
Similarly
 	\begin{equation*}
 	X_t(A^c)-s^\star|A^c| \leq X_t^+(A^c)-s^\star |A^c| \leq |X_t^+(A^c)-s^\star |A^c||,
 	\end{equation*}	
and thus 
 	\begin{equation*}
 	X_t(A)-s^\star|A| \geq X_t-ns^\star - |X_t^+(A^c)-s^\star |A^c||.
 	\end{equation*}	
Therefore,
 	\begin{equation*}
 	| X_t(A)-s^\star|A|| \leq |X_t-ns^\star| + |X_t^+(A)-s^\star |A|| + |X_t^+(A^c)-s^\star |A^c||.
 	\end{equation*}
Using the Cauchy-Schwarz inequality, we get 
 	\begin{eqnarray} 
	\label{3e-app}
 	\E_{\sigma_0}  \left(X_t(A)-s^\star|A|\right)^2 &\leq& 3 \Big [ \E_{\sigma_0} \left(X_t-ns^\star\right)^2 + \E_+(X_t(A)-s^\star |A|)^2   \notag  \\
 	&& \qquad + \E_+ (X_t(A^c)-s^\star |A^c|)^2 \Big].
 	\end{eqnarray}
Using \eqref{varxtp-app},
 	\begin{equation} 
	\label{sO0-app}
 	\sup_{\sigma_0 \in \Omega_n^\delta}\E_{\sigma_0} \left(X_t-ns^\star\right)^2  = \kO(n).
 	\end{equation} 
We observe that by symmetry for all $i,j$ 
 	\begin{equation} 
	\label{a+ij-app}
 	a:=\pp_+(\xi_t(i)=1) = \pp_+(\xi_t(j)=1), 
 	\end{equation}
and for all pairs $(i,j)$ and $(k,\ell)$ with $i\neq j$ and $k\neq \ell$, 
 	\begin{equation*}
 	\rho:= \textrm{Cov}_+(\xi_t(i)=1,\xi_{t}(j)=1) = \textrm{Cov}_+(\xi(k)=1,\xi_{t}(\ell)=1). 
 	\end{equation*}
By \eqref{sO0-app},
 	\begin{equation} 
	\label{p_+-app}
 	|a-s^\star| = \kO(1/\sqrt{n}),
 	\end{equation}
and 
 	\begin{equation} 
	\label{var_+-app}
 	\Var_+(X_t) =\kO(n).
 	\end{equation}
Using \eqref{a+ij-app} and \eqref{p_+-app}, we get 
 	\begin{equation*}
 	\E_+(X_t(A))=\sum_{i\in A} \pp_+(\xi_t(i)=1) =|A|s^\star + \kO(|A|/\sqrt{n}).
 	\end{equation*}
Thus 
 	\begin{equation} 
	\label{se_+-app}
 	\sup_{A\subseteq [n]}\E_+(X_t(A)-|A|s^\star) =\kO(\sqrt{n}).
 	\end{equation}
If $\rho \geq 0$, then 
 	\begin{equation*}
 	\Var_+(X_t(A)) \leq \Var_+(X_t([n])) =\Var_+(X_t) = \kO(n),
 	\end{equation*}
by using \eqref{var_+-app}. Otherwise, if $\rho <0$ then
 	\begin{equation*}
 	\Var_+(X_t(A)) \leq \sum_{i\in A} \Var(\indic{\xi_t(i)=})) \leq |A| =\kO(n).
 	\end{equation*}
In conclusion,
 	\begin{equation*} 
 	\sup_{A\subset[n]}\Var_+(X_t(A)-|s|_\star|A|) =\kO(n).
 	\end{equation*}
Combining this estimate with \eqref{se_+-app} yields that
 	\begin{equation} 
	\label{se_+2-app}
 	\sup_{A\subset[n]} \E_+(X_t(A)-s^\star |A|)^2 =\kO(n).
 	\end{equation}
Combining \eqref{3e-app}, \eqref{sO0-app} and \eqref{se_+2-app} we obtain the desired result. 
 \end{proof}

\begin{proof}[Proof of Lemma \ref{lem:dtng}]
Let $A=\{i\colon \sigma_0(i)=1\}$ and define 
 	\begin{equation*}
 	\kE=\{|X_{T_n^+(\gamma/2)}-ns^\star | \leq n \delta, \, |\bar{X}_{T_n^+(\gamma/2)}-ns^\star | \leq n \delta \}.
 	\end{equation*}
By definition,
 	\begin{eqnarray*}
 	D(t) = |U(\xi_t)-U(\bar{\xi}_t)|& =& \Big | \sum_{i \in A} \xi_t(i) -\sum_{i \in A} \bar{\xi}_t(i) \Big | = |X_t(A)-\bar{X}_t(A)| \notag \\
 	&\leq & |X_t(A)-s^\star |A|| + |\bar{X}_t(A)-s^\star |A||.
 	\end{eqnarray*}
Combining this with Lemma \ref{xtasa}, we have 
 	\begin{equation*}
 	\E_{\sigma_0,\sigma} (D(T^+_n(\gamma))\mid \kE) \leq C\sqrt{n},
 	\end{equation*}
for some $C>0$. By Lemma \ref{lem:xt-ns}(ii)
 	\begin{equation*}
 	\lim_{\gamma \rightarrow \infty} \limsup_{n\rightarrow \infty}\pp_{\sigma_0,\sigma} (\kE^c) =0.
 	\end{equation*} 
Combining the last two estimates with Chebyshev's inequality we obtain the desired result.
 \end{proof}


 \begin{proof}[Proof of Lemma \ref{lem:h2g}]
 	Let $A=\{i\colon \sigma_0(i)=1\}$ and  define 
 	\begin{equation*}
 	\kE=\{|X_{T_n^+(\gamma_1/2)}-ns^\star | \leq n \delta, \, |\bar{X}_{T_n^+(\gamma_1/2)}-ns^\star | \leq n \delta \}.
 	\end{equation*}
We observe that by definition of $\Theta$,
 	\begin{eqnarray} 
	\label{xnit-app}
 	\{\xi_t \not \in \Theta\} &=& \{X_t(A) \leq \delta n/8\} \cup \{X_t(A^c) \leq \delta n/8\} \notag \\ 
 	&& \quad \cup \{|A|-X_t(A) \leq \delta n/8\}  \cup \{|A^c|-X_t(A^c) \leq \delta n/8\}.
 	\end{eqnarray}	
It follows from Lemma \ref{xtasa} and Chebyshev's inequality that for all $B \subset [n]$ and $t\geq T_n^+(\gamma_1)$,
 	\begin{equation} 
	\label{xtbs-app}
 	\pp((s^\star-\delta)|B| \leq X_t(B) \leq (s^\star+\delta) |B| \mid \kE ) \leq \frac{C_1 n}{|B|^2},
 	\end{equation}	
for some $C_1=C_1(\delta)$. Since $\sigma_0 \in \Omega_n^\delta$,
 	\begin{equation} 
	\label{asd-app}
 	(s^\star - 2 \delta) n \leq |A|\leq (s^\star + 2 \delta) n .
 	\end{equation}
Therefore, by \eqref{xtbs-app},
 	\begin{eqnarray} 
	\label{ftme-app}
 	\pp(F_t^c \mid \kE )  \leq C/n,
 	\end{eqnarray}
where $C=C(C_1,s^\star)=C(\delta,s^\star)>0$ and
 	\begin{equation*}
 	F_t =\{|X_t(A)-(s^\star)^2n| \leq 3\delta n, \, |X_t(A^c)-s^\star(1-s^\star)n| \leq 3\delta n \}.
 	\end{equation*}
Using \eqref{xnit-app}, \eqref{asd-app} and the fact that $|X_t(B)-X_{t+1}(B)| \leq 1$ for all $t$ and $B$, we achieve that for all $t_0$,
 	\begin{equation} 
	\label{xnit2-app}
 	\bigcup_{t=t_0}^{t_0+\delta n} \{\xi_t \not \in \Theta\}
 	\subset F_{t_0}.
 	\end{equation}
Therefore,
 	\begin{equation} \label{xi-t12}
 	\pp\left(\bigcup_{t=T_n^+(\gamma_1)}^{T_n^+(\gamma_2)} \{\xi_t \not \in \Theta\} \mid \kE \right) \leq \frac{C(\gamma_2-\gamma_1)}{n},
 	\end{equation}
by using \eqref{ftme-app}. Similarly,
 	\begin{equation} \label{bxi-t12}
 	\pp\left(\bigcup_{t=T_n^+(\gamma_1)}^{T_n^+(\gamma_2)} \{\bar{\xi}_t \not \in \Theta\} \mid \kE \right) \leq \frac{C(\gamma_2-\gamma_1)}{n}.
 	\end{equation}
 	
By Lemmas \ref{lem:hit} and \ref{lem:stay}, there exist positive constants $C_1=C_1(\delta)$ and $C_2=C_2(\delta, \gamma_1)$, such that 
	\begin{equation*}
	\pp \left( |X_{T_n^+(\gamma_1/2)}-ns^\star | \leq 2 n \delta \right) \leq \frac{C_1}{n}  + \frac{C_2 (\log n)^2}{n^2}.
	\end{equation*}
Using exactly the same proof as in Lemmas \ref{lem:hit} and \ref{lem:stay} (just replacing $2\delta$ by $\delta$ in the whole argument), we can show that 
	\begin{equation*}
	\pp \left( |X_{T_n^+(\gamma_1/2)}-ns^\star | \leq  n \delta \right) \leq \frac{C_3}{n}  + \frac{C_4(\log n)^2}{n^2},
	\end{equation*}	
for some $C_3=C_3(\delta)$ and $C_4=C_4(\delta, \gamma_1)$. The same inequality holds for $\bar{X}_{T_n^+(\gamma_1/2)}$. Therefore, 
 	\begin{equation} 
	\label{ke}
 	 \pp_{\sigma_0,\sigma}(\kE)  \leq C/n,
 	\end{equation}
for some $C=C(\gamma_1)$ (note that $\delta$ is fixed). Combining \eqref{xi-t12}, \eqref{bxi-t12} and \eqref{ke},  we obtain the desired result in \eqref{gen-Lem7.9}.
The conclusion in \eqref{needed-Lem7.9} follows immediately, as we first take $n\rightarrow \infty$ followed by $\gamma\rightarrow \infty$.
\end{proof}
 

\fi
\bigskip 

\begin{ack} {\rm We thank Anton Bovier 
for stimulating discussions and fruitful comments.  
The work of V. H. Can is supported by  the fellowship no. 17F17319 of the Japan Society for the Promotion of Science, and by the  Vietnam National Foundation for Science and Technology Development (NAFOSTED) under grant number 101.03--2019.310. 
The work of RvdH is supported by the Netherlands Organisation for Scientific Research (NWO) through the Gravitation Networks grant 024.002.003.    
The work of TK is supported by the 
JSPS KAKENHI Grant Number JP17H01093 and by the Alexander von Humboldt Foundation.}
\end{ack}

\end{document}